\scrollmode

\documentclass[11pt]{amsbook}
\usepackage{amssymb}
\usepackage[all]{xy}
\usepackage{graphicx}

\theoremstyle{plain}
\newtheorem{prop}{Proposition}
\newtheorem{thm}[prop]{Theorem}
\newtheorem{cor}[prop]{Corollary}
\newtheorem{lem}[prop]{Lemma}
\newtheorem{fact}[prop]{Fact}

\newtheorem{ques}{Question}

\theoremstyle{remark}
\newtheorem{exs}{Examples}
\newtheorem{example}{Example}
\newtheorem{rem}[prop]{Remark}

\numberwithin{prop}{chapter}
\numberwithin{equation}{chapter}

\theoremstyle{definition}
\newtheorem{defi}{Definition}

\newcommand{\ab}{\mathrm{ab}}

\newcommand{\dbl}{[\![}
\newcommand{\dbr}{]\!]}
\newcommand{\dbml}{\langle\!\langle}
\newcommand{\dbmr}{\rangle\!\rangle}
\newcommand{\dblaurl}{(\!(}
\newcommand{\dblaurr}{)\!)}
\newcommand{\ZpG}{\Z_p\dbl G\dbr}
\newcommand{\Zen}{\mathrm{Z}}
\DeclareMathOperator{\bra}{Br}

\DeclareMathOperator{\res}{res}

\DeclareMathOperator{\ccd}{cd}
\DeclareMathOperator{\rank}{rk}
\DeclareMathOperator{\grad}{gr}

\newcommand{\F}{\mathbb{F}}
\newcommand{\Z}{\mathbb{Z}}
\newcommand{\I}{\mathbb{I}}
\newcommand{\Q}{\mathbb{Q}}
\newcommand{\N}{\mathbb{N}}
\newcommand{\MiK}{\mathcal{K}}
\newcommand{\sep}{\mathrm{sep}}
\newcommand{\bKsep}{\bar{K}^{\sep}}

\newcommand{\chr}{\mathrm{char}}

\newcommand{\spa}{\mathrm{Span}}
\DeclareMathOperator{\Gal}{Gal}

\newcommand{\opno}{\lhd_o}
\newcommand{\opsgp}{\leq_o}

\newcommand{\clsgp}{\leq_c}

\newcommand{\kernel}{\mathrm{ker}}

\newcommand{\image}{\mathrm{im}}
\newcommand{\iid}{\mathrm{id}}
\newcommand{\Hom}{\mathrm{Hom}}

\newcommand{\tor}{\mathrm{tor}}
\newcommand{\rk}{\mathrm{rk}}
\newcommand{\Ext}{\mathrm{Ext}}
\newcommand{\Tor}{\mathrm{Tor}}

\newcommand{\argu}{\hbox to 7truept{\hrulefill}}

\newcommand{\doublespaced}{\renewcommand{\baselinestretch}{2}\normalfont}
\newcommand{\singlespaced}{\renewcommand{\baselinestretch}{1}\normalfont}
\doublespaced

\begin{document}

\title{Cohomology of Absolute Galois Groups}
\author{Claudio Quadrelli}

\pagenumbering{roman}



\begin{center}

{\small \textsf{Universit\`a degli Studi di Milano-Bicocca and Western University}}

\medskip

\includegraphics[scale=0.8]{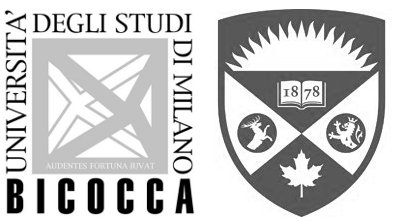}

\bigskip

{\bf \Large \textsf{COHOMOLOGY}

\textsf{OF} \textsf{ABSOLUTE} \textsf{GALOIS} \textsf{GROUPS}}

\bigskip

{\huge \textsf{Claudio Quadrelli}}

\medskip

\bigskip

{\large \textsf{December 2014}}

\end{center}

\bigskip

\bigskip

\noindent\makebox[-0.06in][l]{\rule[2ex]{2in}{.3mm}}
\singlespaced
\textsf{Advisors:\\
{\large Prof. Thomas S. Weigel} \\
{\large Prof. J\'an Min\'a\v{c}} }

\bigskip

\bigskip
\noindent\makebox[-0.06in][l]{\rule[2ex]{5in}{.5mm}}
\singlespaced
Dottorato di Ricerca in Matematica Pura e Applicata,\\
Dipartimento di Matematica e Applicazioni, Universit\`a Milano-Bicocca\\
Ph.D. Program in Mathematics\\
Department of Mathematics, Western University

\bigskip

\noindent\singlespaced
{\small A thesis submitted in partial fulfillment of the requirements for the degree of Doctor of Philosophy
- {\textcopyright} C. Quadrelli 2014}

\newpage

{\flushright
{\small
A Elisa

\bigskip

\bigskip

 ``{\it E propri adess, che te me fet vid\'e 'n suris,\\
la radio parla de questo cielo semper p\"us\'e gris.\\
E propri adess, che te strenget la me mann,\\
g'ho dum\`a 'l temp de tir\`a s\"o\"o tri strasc e 'namm.}''\\

}}

\newpage

\newpage

\chapter*{Acknowledgments}

\begin{quotation}
 ``Vita che la se incastra e l'\`e disegnada a lisca de pess,\\
vita che la sbrisiga verso quel sit 'd te gh'eet v\"oja de v\`ess. \\
Unda che la te porta fin a la spunda de l\`a \\
per fatt sentii p\"us\'e forta la v\"oja de turn\`a a caa.''\footnote{``Life, blocks fitting together,
life shaped like a fishbone; life, which slips toward the place where you'd like to be.
Wave, which pulls you to the other shore, to make you feel homesick once more and again''. From the song {\it La terza onda}.} \\
(D. Van de Sfroos)
\end{quotation}

First of all, I am bound to thank my first advisor, {\it maestro}, and so far also work-mate
(and friend), Prof. Thomas S. Weigel from Milano-Bicocca.
In particular, I am grateful to him also for never letting me to rest on my laurels,
and for being a ``warmduscher'' only in real life, and never in Mathematics.

And of course, I am very grateful to my second advisor, Prof. J\'an Min\'a\v{c} from Western, a wonderful teacher --
whose enthusiasm for Mathematics (and not only) is really contagious, and whose laugh is very well known
by many mathematicians --, great scorer (in Mathematics and in soccer), work mate and friend as well.
And thanks also to his best friend, secretary, driver, proof-reader, actress, singer (and when she has some free time, also wife),
Leslie... without her, what would be of Professor Min\'a\v{c}?

Sincere thanks to the member of my defense board:
Prof. D.~Riley and Prof. E.~Schost from Western, Prof. R.~Schoof from Roma Tor Vergata and Prof.ssa F.~Dalla Volta
from Milano-Bicocca.
I am particularly grateful to D.~Riley as Department Chair because of his effort to let me visit
Western University as undergraduate and to let me became Ph.D. student there; and I am grateful to him as professor
for all the things I learned from him about pro-$p$ groups and restricted Lie algebras.

My sincere thanks to S.K.~Chebolu, from Illinois State University, for his mathematical insight and for his friendship.
I hope we will be working together for a long time

There are some mathematicians I have met during my ``graduate career'' whom I would like to thank
for the interest shown towards my work: L.~Bary-Soroker, I.~Efrat, G.~Fernandez-Alcober, D.~Krashen, J.~Labute,
Ch.~Maire, D.~Neftin, J.~Sonn, P.~Symonds, D.~Vogel, P.~Zalesski\u{i}.
And thanks to my young collegues, in particular to A.~Chapman and J.~G\"artner.

Two people from the Math Deaprtment in Milan deserve my thanks: Tommaso~T., whom I consider as my old
``mathematical brother'', and Giusy~C., the true ``pillar'' of the department.
Same treatment is deserved by N.D.~T\^an, from Western, for his thoughtful questions and for his ``homeworks'',
without which Section~4.4 would not exists.

Last -- but not least -- I have to thank my mates, who have been good companions throughout my whole studies...
they would be too many, but I can't help to mention at least some.
From my undergraduate studies: Chiara~B., Claudio~S., Cristian~C., Davide~C., Davide~M., Matteo~C., Mattia~B.,
Paolo Nicola~C., Stefano~P.; 
and from my graduate studies: Chiara~O., Dario~C., Francesca~S., Jacopo~C., Linda~P., Marina~T., Sara~V., Simone~C.
And also from Western: Ali~F., Asghar~G., Masoud~A., Mike~R. and Sajad~S.

A special mention is deserved by office-mates in Milan:
Gianluca~P. (one of the few people in the world able to understand, appreciate and enjoy my humor,
together with Paolo Nicola!), Martino~B. (the first to leave the office), Nicola~S. (even if he's an analyst!),
Federico William~P. (even if he's abusive!), Elena~V. (even if she studies... ehm, doesn't matter),
and also Ilaria~C. and Raimundo~B. for their ``loud presence'' (cit.).

Thanks to my landlady in London-ON, Alison~W., and my room-mates (in particular Silke~D.),
for providing me a welcoming home while staying in Canada.

I can't help to thank explicitly at least some (but only for the lack of space) of my dearest friends
(who are not mentioned above):
Teo~C., the ``Syncefalia'', Samu \& Giulia, Coco \& JPF, Teo \& Chiara~B., il Biondo, il White,
and many other people from the Oratorio di Binago.

Last, but not least, thanks to my parents, Alfredo and Paola, for what I am is also because of them; and to my sister,
Giulia, for always supporting me (also with cakes and cookies).
And to Elisa (and any word here would seem inadequate, as indeed she often leaves me speechless).

\newpage

\tableofcontents

\newpage
\pagenumbering{arabic}

\chapter*{Introduction}

\begin{quote}
 ``Die Zahlentheorie nimmt unter den mathematischen Disziplinen eine \"anlich idealisierte Stellung ein
wie die Mathematik selbst unter den anderen Wissenschaften.''\footnote{``Number Theory, among the mathematical 
disciplines, occupies a similar ideaized position to that held by Mathematics itself among the sciencies''.
From the introduction of Neukirch's book {\it Algebraic Number Theory}, and quoted at the beginning
of the Introduction of \cite{nsw:cohm}.} \\ (J.~Neukirch)
\end{quote}

One may very well say that the Theory of Profinite Groups is ``daughter'' of Galois Theory.
Indeed profinite groups arose first in the study of the Galois groups of infinite Galois extensions of fields,
as such groups are profinite groups, and they carry with them a natural topology, the {\it Krull topology},
which is induced by the Galois subextensions of finite degree, and under this topology they are Hausdorff compact
and totally disconnected topological groups: these properties characterize precisely profinite groups.
Also, the Krull topology allows us to state the Fundamental Theorem of Galois Theory
also for infinite Galois extensions: one has a bijective correspondence between subextensions
and {\it closed} subgroups, and in particular between {\it finite} subextensions and {\it open} subgroups.

In particular, the tight relation between profinite groups and Galois groups is stated by the following theorem,
proved first by H.Leptin in 1955 (cf. \cite{leptin:galois theorem}).

\begin{thm}\label{0thm:G profinite}
Let $G$ be a profinite group. Then there exists a Galois extension of fields $L/K$ such that $G=\Gal(L/K)$.
\end{thm}

The proof of this theorem one commonly refers to nowadays is due to L.~Ribes (1977, 
see also \cite[\S~2.11]{ribeszalesskii:profinite}).
Note that the aforementioned theorem does not say anything about the nature of the field $K$
nor about the extension $L/K$.
In fact, the essence of the whole Galois theory is to ``lose information'', as one passes from a field, i.e., an
algebraic structure with two compatible operations, to a (profinite) group, i.e., an algebraic (and topological) 
structure with only one operation (compatible with the topology).

Every field $K$ comes equipped with a distinguished Galois extension: the {\it separable closure} $\bar K^{\text{sep}}$.
Its Galois group $G_K=\Gal(\bar K^{\text{sep}}/K)$ is called the {\bf absolute Galois group} of $K$.
Such extension collects all (finite and infinite) Galois extensions of $K$; in particular, all Galois groups
of $K$ are ``encoded'' in $G_K$: this is why absolute Galois groups of fields have a prominent place in Galois theory.

Unfortunately (or fortunately, otherwise I would have wasted the last four years of my life), it is impossible
to have a result like Theorem~\ref{0thm:G profinite} for absolute Galois groups: not every profinite group
is realizable as absolute Galois group, and in general it is a very hard problem to understand which profinite group
is an absolute Galois group.
For example, the celebrated Artin-Schreier Theorem states that the only finite group which is realizable as
absolute Galois group is the finite group of order two (for example, as $\Gal(\mathbb{C}/\mathbb{R})$,
see Section~2.3 for the complete statement of the Artin-Schreier Theorem).

\bigskip

Thus, the problem to understand which profinite groups are realizable as absolute Galois groups --
and also to recover some arithmetic information from the group --  has caught the attention
of algebraists and number theorists in the last decades, and many are working on it
from different points of view, and using various tools.

A very powerful one is Galois Cohomology.
Actually, the first comprehensive exposition of the theory of profinite groups appeared in the book 
{\it Cohomologie Galoisienne} by J-P.~Serre in 1964, which is a milestone of Galois theory.
The introduction of Galois cohomological techniques is doubtlessly one of the major landmarks of 20th century
algebraic number theory and Galois theory.
For example, class field theory for a field $K$ is nowadays usually formulated via cohomological duality properties
of the absolute Galois group $G_K$.
 
A recent remarkable developement in Galois cohomology is the complete proof of the {\bf Bloch-Kato conjecture}
by V.~Voevodsky, with the substantial contribution of M.~Rost (and the ``patch'' by C.~Weibel).
The first formulation of this conjecture is due to J.~Milnor, with a later refinement by J.~Tate (see
Section~2.2 for a more detailed history of the conjecture).
The conjecture states that there is a tight relation between the cohomology of the absolute Galois group $G_K$ of
a field $K$ (a group-theoretic object), and the {\it Milnor $\mathcal{K}$-ring} of $K$ (an arithmetic object).
In particular, one has that the {\it Galois symbol}
\begin{equation}
\xymatrix{ \mathcal{K}_n^M(K)/m.\mathcal{K}_n^M(K)\ar[rr]^-{h_K}
   && H^n\left(G_K,\mu_m^{\otimes n}\right) } 
\end{equation}
from the $n$-th Milnor $\mathcal{K}$-group of $K$ modulo $m$ to the $n$-th cohomology group of $G_K$ with coefficients in
$\mu_m^{\otimes n}$, with $\mu_m$ the group of $m$-th roots of unity lying in $\bar K^{\text{sep}}$,
is an isomorphism for every $n\geq1$ and for every $m\geq2$ such that the characteristic of the field $K$
does not divide $m$.

Therefore, after the proof of what is nowadays called the Rost-Voevodsky theorem, one has these two hopes:
\begin{itemize}
 \item[(1)] to recover information about the structure of the absolute Galois group
from the structure of its cohomology;
 \item[(2)] to recover arithmetic information from the structure of the cohomology of the absolute Galois group
-- and thus, possibly, from the group structure of $G_K$ itself. 
\end{itemize}
 
\bigskip

Yet, in general it is still rather hard to handle an absolute Galois group.
Thus, for a prime number $p$, we shall focus our attention to the {\it pro-$p$ groups} ``contained'' in an absolute Galois group $G_K$:
the pro-$p$-Sylow subgroups of $G_K$ (which are again absolute Galois groups) and, above all, the 
{\bf maximal pro-$p$ Galois group} $G_K(p)$ of $K$, i.e., the maximal pro-$p$ quotient of the absolute Galois
group $G_K$.\footnote{Note that every pro-$p$-Sylow subgroup of an absolute Galois group, i.e., every absolute Galois group which is pro-$p$, is 
also the maximal pro-$p$ quotient of itself, thus the class of maximal pro-$p$ Galois groups is more general
than the class of absolute Galois pro-$p$ groups, and every result which holds for maximal pro-$p$ Galois groups,
holds also for absolute Galois groups which are pro-$p$.}
Indeed, pro-$p$ groups are much more understood than profinite groups, and this reduction is not an ``abdication'',
as such groups bring substantial information on the whole absolute Galois group, and in some cases they determine the structure of the field.
Also, many arguments form Galois cohomology and from the study of Galois representations suggest that one should
focus on pro-$p$ quotient (cf. \cite[Introduction]{bogomolovtschinkel:birat}).

In particular, the Bloch-Kato conjecture has the following corollary: if the field $K$ contains
a primitive $p$-th root of unity (and usually one should assume that $\sqrt{-1}$ lies in $K$, if $p=2$),
then the Galois symbol induces the isomorphisms of (non-negatively) graded $\F_p$-algebras
\begin{equation}
 \frac{\mathcal{K}_\bullet^M(K)}{p.\mathcal{K}_\bullet^M(K)}\simeq H^\bullet\left(G_K,\mu_p\right)\simeq
H^\bullet\left(G_K(p),\F_p\right),
\end{equation}
where the finite field $\F_p$ is a trivial $G_K(p)$-module, and $H^\bullet$ denotes the cohomology ring, equipped
with the {\it cup product}.
Since the Milnor $\mathcal{K}$-ring $\mathcal{K}_\bullet^M(K)$ is a {\it quadratic algebra} --
i.e., a graded algebra generated by elements of degree one and whose relations are generated in degree two.
also the $\F_p$-cohomology ring of the maximal pro-$p$ Galois group $G_K(p)$ is a quadratic algebra over the field $\F_p$
This provides the inspiration for the definition of a {\bf Bloch-Kato pro-$p$ group}: a pro-$p$ group
such that the $\F_p$-cohomology ring of every closed subgroup is quadratic.
Bloch-Kato pro-$p$ groups were first introduced in \cite{bcms:bk}, and then defined and studied in \cite{claudio:BK}.

\bigskip

Another tool to study maximal pro-$p$ Galois group is provided by the {\it cyclotomic character}, induced by 
the action of the absolute Galois group $G_K$ on the roots of unity lying in $\bar K^{\text{sep}}$:
in the case $K$ contains $\mu_p$, then the cyclotomic character induces a continuous homomorphism from the 
maximal pro-$p$ Galois group $G_K(p)$ to the units of the ring of $p$-adic integers $\Z_p^\times$, called the
{\it arithmetic orientation} of $G_K(p)$.

Thus, a continuous homomorphism of pro-$p$ groups $\theta\colon G\to\Z_p^\times$ is called an orientation for $G$,
and if $G$ is a Bloch-Kato pro-$p$ group and certain conditions on the induced Tate-twist module $\Z_p(1)$ are satisfied,
the {\bf orientation} is said to be {\bf cyclotomic} (in particular, the group $H^2(G,\Z_p(1))$
has to be torsion-free, see Subsection 2.4.2), and the group $G$ is said to be {\bf cyclo-oriented}.
Cyclo-oriented pro-$p$ groups are a generalization of maximal pro-$p$ Galois groups; we may say that they are 
``good candidates'' for being realized as maximal pro-$p$ Galois groups, as they have the right cohomological
properties.

For a pro-$p$ group $G$ with cyclotomic orientation $\theta$ such that either $\image(\theta)$ is trivial
or $\image(\theta)\simeq\Z_p$ -- which is always the case if $p\neq2$ -- one has an epimorphism of graded $\F_p$-algebras
\begin{equation}\label{0eq:exterior epimorphism}
\xymatrix{ \bigwedge_\bullet\left(H^1(G,\F_p)\right)\ar@{->>}[r] & H^\bullet(G,\F_p) }, 
\end{equation}
i.e., the $\F_p$-cohomology ring of $G$ is an epimorphic image of the exterior algebra over $\F_p$ generated by
the grup $H^1(G,\F_p)$.

If $G$ is finitely generated, then we are given two bound-cases: when the morphism \eqref{0eq:exterior epimorphism}
is trivial and when it is an isomorphism.
The former case is precisely the case of a free pro-$p$ group, in the latter case the group is said to be
{\bf $\theta$-abelian}.
One may represent this situation with the picture below.\footnote{I displayed this picture the first time in a talk 
during a workshop organized at the Technion, Haifa, in honor of Prof. J.~Sonn.}

\includegraphics[scale=0.5]{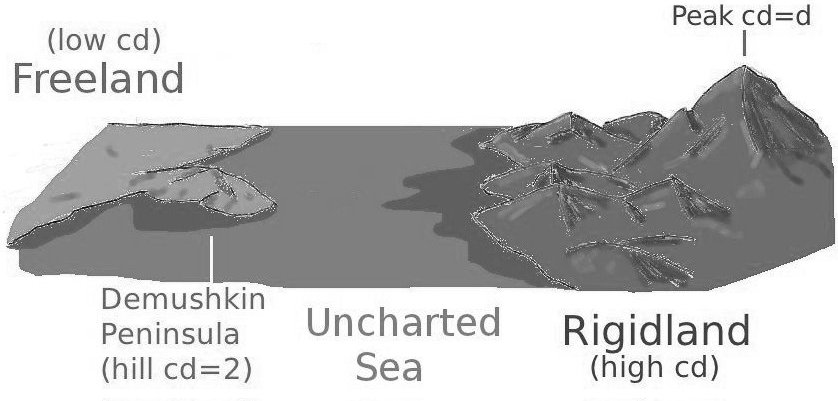}

In fact, from the ``mountain'' side it is possible to recover the full structure of the group $G$, and also the 
arithmetic of the base field, in the case $G$ is a maximal pro-$p$ Galois group, as stated by the following theorem
(see Theorem~\ref{3thm:equivalence theta-abelian} and Theorem~\ref{3thm:rigidity tutto}).

\begin{thm}\label{0thm:thetabelian}
 Let $G$ be a finitely generated cyclo-oriented pro-$p$ group.
Then the following are equivalent:
\begin{itemize}
 \item[(1)] the epimorphism \eqref{0eq:exterior epimorphism} is an isomorphism;
 \item[(2)] the cohomological dimension of $G$ is equal to the minimal number of generators of $G$;
 \item[(3)] $G$ has a presentation
\[ G=\left\langle \sigma,\tau_1,\ldots,\tau_d\:\left|\:\sigma\tau_i\sigma^{-1}=\tau_i^{1+p^k},
\tau_i\tau_j=\tau_j\tau_i \ \forall\;i,j=1,\ldots,d\right.\right\rangle \]
with $d\geq1$ and $k\in\N\cup\{\infty\}$ such that $\image(\theta)=1+p^k\Z_p$.
\end{itemize}
Moreover, if $G$ is the maximal pro-$p$ Galois group of a field $K$ containing a primitive $p$-th root of unity,
the above conditions hold if, and only if, $K$ is a $p$-rigid field, i.e.,
$K$ has a $p$-Henselian valuation of rank $d$.
\end{thm}

This last point has particular relevance, since there is much interest in construction of non-trivial valuations
of fields. Such constructions became particularly important in recent years in connection with the so-called
{\it birational anabelian geometry}, (cf. \cite{bogomolovtschinkel:birat}, \cite{pop:birat}).
This line of research originated from ideas of A.~Grothendieck and of J.~Neukirch:
as stated, the goal is to recover the arithmtic structure of a field from its various canonical Galois groups.
The point is that usually the first step is to recover enough valuations from their cohomological ``footprints''.

The ``mountain-case'' is discriminant for Bloch-Kato pro-$p$ groups also in the sense specified by
the following Tits alternative-type result (see Theorem~\ref{3thm:titsalternative}).

\begin{thm}\label{0thm:tits}
 Let $G$ be a Bloch-Kato pro-$p$ group.
Then either the epimorphism \eqref{0eq:exterior epimorphism} is an isomorphism, or $G$ contains
a free non-abelian closed subgroup.
\end{thm}

Thus, every Bloch-Kato pro-$p$ group which is floating in the ``uncharted sea'' contains a trace from the West shore.
On the other hand, it is possible to generalize the situation of $\theta$-abelian groups in the following way.
Set the {\bf $\theta$-centre} of a cyclo-oriented pro-$p$ group $G$ to be the (normal) subgroup
\begin{equation}
 \Zen_\theta(G)=\left\{\tau\in\kernel(\theta) \: \left|\:\sigma\tau\sigma^{-1}=\tau^{\theta(\sigma)}\text{ for all }\sigma\in G \right.\right\}
\end{equation}
Then $\Zen_\theta(G)$ is the maximal abelian normal subgroup of $G$ (cf. Proposition~\ref{4prop:Zmxlabeliannormal}),
and the short exact sequence 
\begin{equation}\label{0eq:sesG}
\xymatrix{ 1\ar[r] & \Zen_\theta(G)\ar[r] & G\ar[r] & G/\Zen_\theta(G)\ar[r] & 1}
\end{equation}
splits (cf. Theorem~\ref{4thm:Zcentresplit}).
Note that in the case of a $\theta$-abelian group, one has $\Zen_\theta(G)=\kernel(\theta)$, and the short exact
sequence \eqref{0eq:sesG} clearly splits, as the presentation in Theorem~\ref{0thm:thetabelian} provides
an explicit complement of the $\theta$-centre in $G$.
And as in Theorem~\ref{0thm:thetabelian}, the $\theta$-centre of a maximal pro-$p$ Galois group detects
the existence of non-trivial valuations, and its presence can be deduced also from the cohomology ring
(cf. Theorem~\ref{4thm:arithmetic thetacentre}).

\begin{thm}\label{0thm:arithmetic thetacentre}
Let $K$ be a field containing a primitive $p$-th root of unity, with maximal pro-$p$ Galois group $G_K(p)$ equipped
with arithmetic orientation $\theta\colon G_K(p)\to\Z_p^\times$. The following are equivalent:
\begin{itemize}
 \item[(1)] the $\theta$-centre of $G_K(p)$ is non-trivial;
 \item[(2)] the $\F_p$-cohomology ring of $G_K(p)$ is the skew-commutative tensor product of an exterior algebra
 with a quadratic algebra;
 \item[(3)] the field $K$ has a $p$-Henselian valuation of rank equal to the rank of $\Zen_\theta(G)$ as abelian pro-$p$ group.
\end{itemize}
\end{thm}

The above result is particularly relevant for the importance of being able to find valuations, as underlined above.
Indeed, with Theorem~\ref{0thm:arithmetic thetacentre} we come full circle, as it completes the picture 
with the results contained in \cite{ek98:abeliansbgps} and \cite{efrat:libro}.
Also, it shows that cyclotomic orientations provide an effective way to express such results.

It is possible to go a bit further, in order to see how the existence of a cyclotomic orientation for a pro-$p$ group
affects the structure of the whole group.
For example, we show that the torsion in the abelianization of a finitely generated pro-$p$ group with cyclotomic
orientation is induced by the ``cyclotomic action'' of the group (cf. Theorem~\ref{4thm:Nab}).
Note that the existence of a cyclotomic orientation is a rather restrictive condition: for example, certain
free-by-Demushkin groups cannot be equipped with a cyclotomic orientation, as shown in Subsection~4.2.1.

\bigskip

Given a pro-$p$ group $G$, one may associate to $G$ another graded $\F_p$-algebra, besides the $\F_p$-cohomology ring:
the graded algebra $\grad_\bullet(G)$ induced by the {\it augmentation ideal} of the completed group algebra
$\F_p\dbl G\dbr$.

In many relevant cases -- such as free pro-$p$ groups, Demushkin groups, $\theta$-abelian groups -- the
$\F_p$-cohomology ring and the graded algebra of a maximal pro-$p$ Galois group happen to be related
via {\bf Koszul duality} of quadratic algebra, and both algebras are {\bf Koszul algebras}
(for the definition of Koszul dual of a quadratic algebra see Definition~\ref{5defi:koszul dual},
and for the definition of Koszul algebra see Definition~\ref{5defi:koszulness}).

Moreover, one has that if the relations of $G$ satisfy certain ``reasonable'' conditions, then 
$H^\bullet(G,\F_p)$ and $\grad_\bullet(G)$ are Koszul dual (cf. Theorem~\ref{5thm:koszul duality 1}).
Thus, we conjecture that if $K$ is a field containing a primitive $p$-th root of unity with finitely generated
maximal pro-$p$ Galois group $G_K(p)$, then $\F_p$-cohomology ring and the graded algebra of $G_K(p)$ are Koszul dual,
and also that both algebras are Koszul algebras (cf. Question~\ref{5ques:koszul}).

Here we study the graded algebra $\grad_\bullet(G)$ of a pro-$p$ group $G$ via the {\it restricted Lie algebra}
induced by the {\bf Zassenhaus filtration} of $G$.
The study of the graded algebras induced by filtrations of pro-$p$ groups has gained much interest recently,
in particular the algebras induced by the Zassenhaus filtration, as well as the algebras induced  
by the $p$-descending central series (see \cite{labute:dcs1,labute:dcs2,labute:mild}, \cite{minacspira:witt},
\cite{suniljan:quotients,cem:quotients}).
For example, we prove the following result (Theorem~\ref{3thm:lambda3Phi2}), the proof of which uses indeed
the Zassenhaus filtration of $G$, which generalizes \cite[Theorem~A]{cmq:fast}:

\begin{thm}\label{0thm:quotients}
It is possible to detect whether a finitely generated Bloch-Kato pro-$p$ group $G$ is $\theta$-abelian
from the third element of its $p$-descending central series.
\end{thm}

Also, the Zassenhaus filtration is proving to be closely related to Massey products, which have shifted from their
original field of application (topology) toward number theory:
see \cite{jochen:thesis}, \cite{efrat:zassenhaus} and \cite{jantan:massey}.

\medskip

Albeit we have still only a glimpse of the structure of maximal pro-$p$ Galois groups but in few specific cases 
(such as the two shores of the picture)
there is a conjecture which states how a finitely generated maximal pro-$p$ Galois group should look like, the 
so called {\bf Elementary Type Conjecture} (or ETC).
Formulated first by I.~Efrat in \cite{efrat:etconj}, the ETC states that finitely generated maximal pro-$p$ Galois groups
have a rather rigid structure: namely, they can be built starting from ``elementary blocks'' such as $\Z_p$ and
Demushkin groups, via rather easy group-theoretic operations, such as free pro-$p$ products
and cyclotomic fibre products (defined in Definition~\ref{4defi:cyclotomic fibre product}).

The only evidences we have for this conjecture are:
\begin{itemize}
 \item[(1)] we have no counterexamples;
 \item[(2)] it seems ``just'' it should be so;
\end{itemize}
which are not very strong.
Yet, we show that all the classes of pro-$p$ groups we study -- Bloch-Kato pro-$p$ groups, cyclo-oriented pro-$p$
groups, Koszul duality groups and Koszul groups --, which are kind of generalizations of maximal pro-$p$ Galois groups,
are closed with respect to free pro-$p$ products and cyclotomic fibre products, and this provides at least more sense
to this conjecture.

\medskip
Therefore, the aim of this thesis is to show that our approach toward Galois theory via the 
cohomology of maximal pro-$p$ Galois groups (in particular studying Bloch-Kato pro-$p$ groups and cyclotomic
orientations) is particularly powerful and effective, as indeed it provides new consistent knowledge
on maximal pro-$p$ Galois groups, and it promises to bring more results in the future.

\medskip

The thesis is structured in the following chapters:
\medskip
\begin{itemize}

 \item[(1)] The first introductory chapter presents the theoretical background of the thesis.
 In particular, it introduces some preliminaries on pro-$p$ groups, together with cohomology of profinite groups,
 Galois cohomology and restricted Lie algebras. 

\medskip

 \item[(2)] First we introduce quadratic algebras and their properties.
Then we present the Bloch-Kato conjecture, and we define Bloch-Kato pro-$p$ groups and cyclo-oriented pro-$p$ groups.
Here we explore the first properties of these groups -- for example, we prove an Artin-Schreier-type
result for cyclo-oriented pro-$p$ groups (cf. Corollary~\ref{2cor:torsion}).
Also, we introduce Demushkin groups and the ETC.

\medskip

 \item[(3)] Here we study the ``mountain side'' of cyclo-oriented pro-$p$ groups, i.e., $\theta$-abelian groups.
We study the cup produt of such groups and we prove Theorem~\ref{0thm:tits}.
In order to describe the group structure of $\theta$-abelian groups we study {\it locally powerful} pro-$p$ groups,
and we show that the two classes of group coincides (cf. Theorem~\ref{3thm:equivalence theta-abelian}).
In order to explain the ``arithmetic role'' of $\theta$-abelian groups we introduce $p$-rigid fields, 
and we prove Theorem~\ref{0thm:thetabelian}.
Then we compute the Zassenhaus filtration for such groups and we prove Theorem~\ref{0thm:quotients}.
Part of the content of this chapter is published in \cite{claudio:BK} and \cite{cmq:fast}.

\medskip

 \item[(4)] In the fourth chapter we study free products and cyclotomic fibre products of cyclo-oriented pro-$p$
groups. In particular, we show that \eqref{0eq:sesG} splits and we prove Theorem~\ref{0thm:arithmetic thetacentre}.
Also, we show that the defining relations of a finitely generated cyclo-oriented pro-$p$ group which induce
torsion in the abelianization are induced by the cyclotomic action of the orientation (cf. Theorem~\ref{4thm:Nab}).
Part of the material contained in this chapter is being developed in \cite{qw:orientations}.

\medskip

 \item[(5)] First we define the koszul dual of a quadratic algebra, and we study Koszul duality for cyclo-oriented
pro-$p$ groups.
Then we define Koszul algebras.
Part of the material contained in this chapter is being developed in \cite{mqrtw:koszul}.
\end{itemize}


\chapter{Preliminaries}
\section[Cohomology]{Group cohomology and Galois cohomology}

Throughout the whole thesis, subgroups are assumed to be closed with respect to the profinite topology,
and the generators are assumed to be topological generators (i.e., we consider the closed subgroup
generated by such elements).

\subsection{Cohomology of profinite group}
We recall briefly the construction of the cohomology groups for profinite groups, and the property we will use further.
We refer mainly to \cite[Ch.~I]{nsw:cohm}.

Let $G$ be a profinite group.

\begin{defi}
 A {\bf topological $G$-module} $M$ is an abelian Hausdorff topological group which is an abstract $G$-module
such that the action
\[G\times M\longrightarrow M\]
is a continuous map (with $G\times M$ equipped with the product topology).
\end{defi}

For a closed subgroup $H\clsgp G$, we denote the subgroup of $H$-invariant elements in $M$ by $M^H$, i.e.
\[M^H=\left\{m\in M \left|\; h.m=m \text{ for all } h\in H\right.\right\}.\]

Assume now that $M$ is a discrete module.
For every $n\geq1$, let $G^{\times n}$ be the direct product of $n$ copies of $G$.
For a $G$-module $M$, define $C^n(G,M)$ to be the {\bf group of (inhomogeneous) cochains} of $G$
with coefficients in $M$, i.e., $C^n(G,M)$ is the abelian group of the continuous maps 
$G^{\times n}\rightarrow M$, with the group structure induced by $M$. 
(Note that, if $M$ is discrete, then a continuous map from $G^{\times n}$ to $M$ is a map which
is locally constant.) Also, define $C^0(G,M)=M$.

The $n+1$-coboundary operator $\partial^{n+1}\colon C^n(G,M)\rightarrow C^{n+1}(G,M)$ is given by
\begin{eqnarray*}
 \partial^1a(g)&=&g.a-a\quad\text{for }a\in M\\
 \partial^2f(g_1,g_2)&=&g_1.f(g_2)-f(g_1g_2)+f(g_1),\quad f\in C^1(G,M)\\
 &\vdots&\\
 \partial^{n+1}f(g_1,\ldots,g_{n+1}) &=& g_1.f(g_2,\ldots,g_{n+1})+\\
 && \sum_{i=1}^n(-1)^if(g_1,\ldots,g_{i-1},g_ig_{i+1},g_{i+2},\ldots,g_{n+1})\\
 && +(-1)^{n+1}f(g_1,\ldots,g_n)\quad\text{for }f\in C^n(G,M).
\end{eqnarray*}
Then one sets $Z^n(G,M)=\kernel(\partial^{n+1})$, called the group of (inhomogeneous) $n$-cocycles,
and $B^n(G,M)=\image(\partial^n)$, called the group of (inhomogeneous) $n$-coboundaries.

\begin{defi}
 For $n\geq0$, the quotient
\[H^n(G,M)=Z^n(G,M)/B^n(G,M)\]
is called the {\bf $n$-th cohomology group} of $G$ with coefficients in the $G$-module $M$.
\end{defi}

One has the following facts:

\begin{fact}\label{1fact:elementary cohomology}
 \begin{itemize}
 \item[(i)] The 0-th cohomology group of a profinite group $G$ with coefficients in $M$ is the subgroup of 
$G$-invariant elements, i.e.,
\[ H^0(G,M)=M^G.\]
In particular, if $G$ acts trivially on $M$, one has $H^0(G,M)=M$.
 \item[(ii)] The 1-cocycles are the continuous maps $f\colon G\rightarrow M$ such that
\[ f(g_1g_2)=f(g_1)+g_1.f(g_2)\quad\text{for every }g_1,g_2\in G.\]
They are also called crossed homomorphisms.
The 1-coboundaries are the continuous maps $a\colon G\rightarrow M$ such that $a(g)=g.a-a$ for every $g\in G$.
In particular, if $G$ acts trivially on $M$, one has
\[ H^1(G,M)=\Hom(G,M),\]
where $\Hom(G,M)$ is the group of (continuous) group homomorphisms from $G$ to $M$.
\end{itemize}
\end{fact}

The following proposition states a fundamental property of group cohomology
(cf. \cite[Theorem~1.3.2]{nsw:cohm}).

\begin{prop}
For an exact sequence $0\rightarrow A\rightarrow B\rightarrow C\rightarrow 0$ of $G$-modules,
one has connecting homomorphisms
\[\delta^n\colon H^n(G,C)\longrightarrow H^{n+1}(G,A)\]
for every $n\geq0$, such that 
\[\xymatrix@C=.6truecm{
\cdots\ar[r] & H^n(G,A)\ar[r] & H^n(G,B)\ar[r] & H^n(G,C)\ar[r]^-{\delta^n} & H^{n+1}(G,A)\ar[r] & \cdots}\]
is a long exact sequence.
\end{prop}

Also, the following result states the behavior of low-degree comology with respect of normal subgroups and quotients
(cf. \cite[Prop.~1.6.7]{nsw:cohm}).

\begin{prop}\label{1prop:5tes}
Let $N$ be a normal closed subgroup of $G$.
Then one has an exact sequence
 \[ \xymatrix@C=1.4truecm{
0\ar[r] & H^1(G/N,A^N)\ar[r]^-{\inf_{G,N}^1} & H^1(G,A)\ar[r]^-{\res_{G,N}^1} & 
H^1(N,A)^G\ar`r[d]`[l]^-{\text{tg}_{G,N}} `[dlll] `[dll] [dll]   \\
 & H^2(G/N,A^N)\ar[r]^-{\inf_{G,N}^2} & H^2(G,A) & } \]
called the {\it five term exact sequence}.
The map $\text{tg}_{G,N}$ is called transgression.
\end{prop}

Let $A,B,C$ be $G$-modules with bilinear pairings
\begin{equation}\label{2eq:pairings of modules}
 A\times B\longrightarrow A\otimes_{\Z}B\longrightarrow C.
\end{equation}
Then \eqref{2eq:pairings of modules} induces the map
\[\xymatrix{H^n(G,A)\times H^m(G,B)\ar[r]^-{\cup} & H^{n+m}(G,C)}\]
called the {\bf cup-product}, which has a very important role in the Bloch-Kato conjecture, as we shall see later.
The cup-product has the following property (cf. \cite[Prop.~1.4.4]{nsw:cohm}).

\begin{prop}\label{1prop:cup product}
The cup-product is associative and skew-commutative, i.e., for $\alpha\in H^n(G,A)$, $\beta\in H^m(G,B)$ and
$\gamma\in H^h(G,C)$, one has
\[(\alpha\cup\beta)\cup\gamma=\alpha\cup(\beta\cup\gamma)\quad \text{and}\quad\alpha\cup\beta=(-1)^{nm}\beta\cup\alpha,\] 
with the identifications 
\[(A\otimes B)\otimes C=A\otimes(B\otimes C)\quad \text{and} \quad A\otimes B=B\otimes A.\]
\end{prop}

\begin{rem}\label{1rem:cup product}
In the case the $G$-module is a ring $k$ such that $k\otimes_{\Z}k=k$ (e.g., $k=\Z,\F_p$),
the cup product makes the cohomology of $G$ with $M$-coefficients a graded-commutative $k$-algebra.
\end{rem}

For a profinite group $G$ the {\it first Bockstein homomorphism} 
\begin{equation}\label{1eq:Bockstein}
 \xymatrix{ \beta^1\colon H^1(G,\F_p)\ar[r] & H^2(G,\F_p) }
\end{equation}
is the connecting homomorphism arising from the short exact sequence of trivial $G$-modules
\[\xymatrix{ 0\ar[r] & \Z/p\Z\ar[r]^p & \Z/p^2\Z\ar[r] & \Z/p\Z\ar[r] & 0. }\]
In the case $p=2$, the first Bockstein morphism is related to the cup product in the following way
(cf. \cite[Lemma~2.4]{idojan:series}).

\begin{lem}\label{1lem:Bockstein and cup}
Let $G$ be a profinite group.
Then, for every $\chi\in H^1(G,\F_2)$ one has $\beta^1(\chi)=\chi\cup\chi$.
\end{lem}

\subsection{Continuous cohomology}

Fix a profinite group $G$.
The complete group algebra of $G$ over $\Z_p$ is the topological inverse limit
\[\Z_p\dbl G\dbr=\varprojlim_{U\opno G}\Z_p[G/U],\]
where $U$ runs through the open normal subgroups of $G$,
and $\Z_p[G/U]$ is the group algebra defined in the usual way for finite groups.
Then $\ZpG$ is a compact $\Z_p$-algebra.
Thus, if $M$ is a topological $G$-module which is also a (continuous) $\Z_p$-module,
then $M$ is a topological $\ZpG$-module.

Let $\Q_p$ be the field of $p$-adic numbers.
Then the embedding $\Z_p\subseteq \Q_p$ induces a short exact sequence
\begin{equation}\label{1eq:sesZpQpIp}
 0\longrightarrow\Z_p\longrightarrow\Q_p\longrightarrow\I_p\longrightarrow0
\end{equation}
which defines the discrete $\Z_p$-module $\I_p$.
In particular, $\Z_p$ and $\I_p$ can be considered as topological $\ZpG$-modules,
the former complete and the latter discrete, with the topology inherited from $\Q_p$.
An essential tool for working with profinite groups and $\Z_p$-modules is {\bf Pontryagin duality}.

\begin{defi}
 Let $M$ be a Hausdorff, abelian and locally compact $\Z_p$-module (e.g., a pro-$p$ group,
a compact $\Z_p$-module or a discrete $\Z_p$-module).
We call the group 
\[M^\vee=\Hom_{\text{cts}}(M,\I_p)\]
the {\bf Pontryagin dual} of $M$.
\end{defi}

For every $\Z_p$ module $M$, one has $(M^\vee)^\vee=M$.
In particular, if $M$ is a compact $\Z_p$-module, then $M^\vee$ is a discrete $\Z_p$-module, and conversely. 
Thus, Pontryagin duality induces an equivalence (of categories) between the two categories of compact $\Z_p$-modules
and discrete $\Z_p$-modules.

For a profinite group $G$, the category of compact $\ZpG$-modules has sufficiently many projectives,
and the category of discrete $\ZpG$-modules has sufficiently many injectives.
Hence, it is possible to define Ext-functors for the two categories in the usual way --
see \cite{brumer:pseudocompact}, \cite[\S~V.2]{nsw:cohm}) and \cite{symonds_thomas:prop}.
Then, one has the following result.

\begin{prop}\label{1prop:ext_tor}
For continuous $\ZpG$-module $N$, one has the isomorphisms
 \[H^n(G,N)\simeq\Ext_{\ZpG}^n(\Z_p,N)\]
for every $n\geq0$
Moreover, if $N$ is discrete, one has the canonical isomorphisms
\[H_n(G,N^\vee)\simeq H^n(G,N)^\vee.\]
\end{prop}

For a detailed exposition on continuous cohomology of profinite groups (which would go far beyond
the aims of this thesis) see the aforementioned works by A.~Brumer, by J.~Neukirch and by P.~Symonds and Th.~Weigel.
For a brief introduction to cohomology and homology of profinite groups, see also \cite[\S~3.4-7]{ribes:lux}.

For a (topologically) finitely generated torsion-free $\ZpG$-module $M$, one has also the following property.
Set $V=\Q_p\otimes_{\Z_p}M$ and $N=\I_p\otimes_{\Z_p}M$.
In particular, $V$ is a $\Q_p$-vector space of finite dimension, and
$N$ is a discrete divisible $p$-primary torsion group, and \eqref{1eq:sesZpQpIp} induces the short exact sequence
\begin{equation}\label{1eq:short exact sequence modules Tate}
\xymatrix{0\ar[r] & M\ar[r] & V\ar[r] & N\ar[r] & 0.}
\end{equation}
Then one has the following result (cf. \cite[Prop.~2.3]{tate:K2cohomology} and \cite[Prop.~2.7.11]{nsw:cohm}).

\begin{prop}\label{1prop:tate proposition}
For every $n\geq0$ one has isomorphisms
\[H^n(G,V)\simeq H^n(G,M)\otimes_{\Z_p}\Q_p,\]
and in the exact cohomology induced by \eqref{1eq:short exact sequence modules Tate},
the kernel of the boundary homomorphism $\delta^{n}\colon H^{n-1}(G,N)\rightarrow H^n(G,M)$
is the maximal divisible $\Z_p$-submodule of $H^{n-1}(G,N)$, and its image is the torsion
$\Z_p$-submodule of $H^n(G,M)$.
\end{prop}

\begin{defi}
The cohomological dimension $\ccd(G)$ of a profinite group $G$
is the least non-negative integer $n$ such that
\[H^m(G,M)=0\quad\text{for all } m>n\]
and for every discrete $G$-module $M$ which is a torsion abelian group, and it is infinite if no such integer exists.
\end{defi}

If $G$ is a pro-$p$ group, then $\ccd(G)\leq n$ if, and only if $H^{n+1}(G,\F_p)=0$,
where $\F_p$ is a trivial $G$-module.

\section{Preliminaries on pro-$p$ groups}

Let $G$ be a profinite group, and let $m$ be a positive integer.
We define $G^m$ and $[G,G]$ to be the closed subgroups of $G$ generated by the $m$-th powers of $G$,
resp. by the commutators\footnote{We use the left notation for the conjugation of group elements: namely, 
\[{^xy}=xyx^{-1} \quad \text{and} \quad [x,y]={^xy}\cdot y^{-1}=xyx^{-1}y^{-1}.\]} of $G$. I.e.,
\[G^m=\langle g^m\:|g\in G\rangle\clsgp G\quad\text{and}\quad 
[G,G]=\langle[g_1,g_2]\:|g_1,g_2\in G\rangle\clsgp G.\] 

\begin{defi}\label{1defi:pro-p groups}
Let $G$ be a profinite group.
\begin{itemize}
 \item[(i)] A subset $\mathcal{X}$ of $G$ is said to be a {\it system of (topological) generators} of $G$
if the closed subgroup of $G$ generated by $\mathcal{X}$ is $G$ itself.
Moreover, the system $\mathcal{X}$ is said to be {\it minimal} if any proper subset of $\mathcal{X}$
generates a proper subgroup of $G$.
 \item[(ii)] The group $G$ is {\it finitely generated} if it has a generating system which is a finite subset.
The minimal number of generators of $G$ -- i.e., the (finite) cardinality of a minimal generating system
$\mathcal{X}$ of $G$ -- will be denoted by $d(G)$, and sometimes it will be called the ``dimension''
of $G$.\footnote{Following the notation of \cite{ddsms:analytic}: indeed this makes sense in the case of
analytic pro-$p$ groups.}
 \item[(iii)] Let $F$ be a free profinite group.
A short exact sequence of profinite groups
\begin{equation}\label{1eq:presentation}
 \xymatrix{ 1\ar[r] & R\ar[r] & F\ar[r]^\pi & G\ar[r] & 1},
\end{equation}
is called a {\it presentation} of $G$.
The the elements $R$ are called the {\it relations} of $G$,
and the generators of $R$ as (closed) normal subgroup of $F$ are called {\it defining relations} of $G$.
If $G$ is finitely generated, then the presentation \eqref{1eq:presentation} is said to be {\it minimal} if $d(F)=d(G)$. 
Moreover, if the minimal number of generators of $R$ as normal subgroup is finite, then it is denoted by $r(G)$,
and $G$ is said to be {\it finitely presented}.
In particular, it is possible to present explicitly $G$ in the following way:
\[G=\langle x\in\mathcal{X}\:|\:\rho=1,\rho\in\mathcal{R}\rangle,\]
with $\mathcal{R}\subset R$ a minimal system of defining relations.
Finally, if $G$ is a pro-$p$ group, then
\begin{equation}\label{1eq:number relations}
 r(G)=\dim_{\F_p}\left(H^2(G,\F_p)\right)
\end{equation}
(cf. \cite[Prop.~3.9.5]{nsw:cohm}).
 \item[(iv)] The {\it rank} of $G$ is defined by
\[\rank(G)=\sup\{d(C)\:|\:C\clsgp G\}=\sup\{d(C)\:|\:C\opsgp G\}\]
(see \cite[\S~3.2]{ddsms:analytic}).
\end{itemize}
\end{defi}

\subsection{Filtrations of pro-$p$ groups}

Recall that the terms of the {\it lower central series} are defined by
\[\gamma_1(G)=G,\quad \gamma_2(G)=[G,G],\quad\ldots\quad\gamma_i(G)=[\gamma_{n-1}(G),G]\]
for every $i\geq1$.

The {\it Frattini subgroup} $\Phi(G)$ of a profinite group $G$ is the intersection of all the maximal proper
open subgroups of $G$.
The Frattini subgroup of $G$ is a closed subgroup, and if $\mathcal{X}\subset G$ is a system of generators of $G$, 
then $\mathcal{X}\Phi(G)/\Phi(G)$ generates $G/\Phi(G)$ (cf. \cite[Prop.~1.9]{ddsms:analytic}).
Note that if \eqref{1eq:presentation} is a minimal presentation of $G$, then $R\subseteq \Phi(F)$.

If $G$ is a pro-$p$ group, then one has the following (cf. \cite[Prop.~1.13]{ddsms:analytic}).

\begin{prop}\label{1prop:frattini subgroup}
 If $G$ is pro-$p$, then
\[\Phi(G)=G^p[G,G],\]
and the quotient group $G/\Phi(G)$ is a $\F_p$-vector space.
In particular, $G$ is finitely generated if, and only if, $G/\Phi(G)$ has finite dimension as $\F_p$-vector space,
and $d(G)=\dim(G/\Phi(G))$.
\end{prop}

The {\bf $p$-descending central series} of a pro-$p$ group is a sort of ``$p$-refinement'' of the central series,
and it is defined by
\[\lambda_1(G)=G,\quad \lambda_2(G)=\Phi(G),\quad\ldots\quad \lambda_{i+1}(G)=\lambda_i(G)^p\left[G,\lambda_i(G)\right],\]
for every $i\geq1$.

\begin{rem}
In \cite{ddsms:analytic} the notation $P_i(G)$ is used instead of $\lambda_i(G)$, and D.~Segal
informally refers to this series as ``the pigs series''.
Other authors, as I.~Efrat and J.~Min\'a\v{c}, use the notation $G^{(i)}$ instead of $\lambda_i(G)$
(see \cite{cem:quotients,cmq:fast,idojan:series}).
Finally, authors from the ``Neukirch's school'' use the notation $G^i$ instead of $\lambda_i(G)$, and $G_i$ instead of $\gamma_i(G)$
(see \cite[\S~III.8]{nsw:cohm} and \cite{jochen:thesis}).
\end{rem}

The $p$-descending central series has the following properties.

\begin{prop}
Let $G$ be a pro-$p$ group.
\begin{itemize}
 \item[(i)] $[\lambda_i(G),\lambda_j(G)]\leq\lambda_{i+j}(G)$ for all $i,j\geq1$.
 \item[(ii)] If $G$ is finitely generated then $\lambda_i(G)$ is open for every $i\geq1$, and the set 
$\{\lambda_i(G),i\geq1\}$ is a basis of neighbourhoods of $1\in G$. 
\end{itemize}
\end{prop}

The {\bf Zassenhaus filtration} of an arbitrary group $G$ (not necessarily a pro-$p$ group) is 
the descending series defined by $D_1(G)=G$ and
\[ D_i(G)=D_{\lceil i/p\rceil}(G)^p\prod_{j+h=i}\left[D_j(G),D_h(G)\right]\]
for every $i\geq1$, where $\lceil i/p\rceil$ is the least integer $m$ such that $pm\geq i$.
This series is also called {\it Jennings filtration} and the groups $D_i(G)$ are called
{\it dimension subgroups} of $G$ in characteristic $p$.
Note that if $G$ is pro-$p$, then $D_2(G)=\Phi(G)$.
Such series has the following properties (cf. \cite[\S~11.1]{ddsms:analytic}).

\begin{prop}\label{1prop:properties Zassenhaus filtration}
\begin{itemize}
 \item[(i)] For each $i,j\geq1$, one has
\[[D_i(G),D_j(G)]\leq D_{i+j}(G)\quad\text{and}\quad D_i(G)^p\leq D_{pi}(G),\]  
and it is the fastest descending series with these properties.\footnote{Namely, if $\{\tilde D_i(G)\}$
is another descending series satisfying these two properties, then $\tilde D_i(G)\supseteq D_i(G)$ for every $i\geq1$.} 
 \item[(ii)] For each $i\geq1$, one has the formula
\begin{equation}\label{1eq:lazard forumula}
  D_i(G)=\prod_{jp^h\geq i}\gamma_j(G)^{p^h},
\end{equation}
due to M.~Lazard.
\item[(iii)] If $G$ is a finitely generated pro-$p$ group, then $G$ has finite rank if, and only if, 
$D_i(G)=D_{i+1}(G)$ for some $i\geq1$ (this result is due to M.~Lazard and D.~Riley).
\end{itemize}
\end{prop}

\section{Restricted Lie algebras}

Unless stated otherwise, if $L$ is a Lie algebra (over an arbitrary commutative ring),
we denote by $(\argu,\argu)\colon L\times L\rightarrow L$ its Lie bracket, and for $v\in L$ the notation
\[\text{ad}(v)\colon L\longrightarrow L,\quad w\longmapsto(v,w)\] denotes the adjoint endomorphism.
Throughout this section, let $\F$ denote a field of characteristic $p>0$.

\begin{defi}
 A restricted Lie algebra over $\F$ is a $\F$-Lie algebra $L$ equipped with a map
\[\argu^{[p]}\colon L\longrightarrow L\]
satisfying the following properties:
\begin{itemize}
 \item[(i)] $(a v)^{[p]}=a^pv^{[p]}$, with $a\in\F$ and $v\in L$;
 \item[(ii)] $(v+w)^{[p]}=v^{[p]}+w^{[p]}+\sum_{i=1}^{p-1}\frac{s_i(v,w)}{i}$, where $v,w\in L$, and
 the $s_i(v,w)$ denote the coefficients of $\lambda^{i-1}$ in the formal expression $\text{ad}(\lambda v+w)^{p-1}(v)$;
 \item[(iii)] $\text{ad}(v^{[p]})=\text{ad}(v)^p$, with $v\in L$.
\end{itemize}
\end{defi}

In fact it is possible to see a restricted Lie algebra over $\F$ as a subalgebra of a $\F$-algebra.
For any associative $\F$-algebra $A$, let $A_L$ be the induced Lie algebra, i.e., $A_L=A$ and
\[(a,b)=ab-ba\quad\text{for all }a,b\in A.\]
A Lie subalgebra $L$ of $A_L$ is said to be restricted if $a^p\in L$ for every $a\in L$.
Equivalently, a $\F$-Lie algebra $L$ together with a unitary operation $\argu^{[p]}$ is said to be restricted
if there exists an associative $\F$-algebra $A$ and a Lie algebras monomorphisms $\iota\colon L\rightarrow A_L$
such that \[\iota\left(a^{[p]}\right)=\iota(a)^p\quad\text{for every }a\in L.\]
The pair $(A,\iota)$ is called a {\it restricted envelope} for $L$.

An ideal $\mathfrak{r}$ of a restricted Lie algebra $L$ is an ideal of the underlying Lie
algebra of $L$ which is also closed under the operation $\argu^{[p]}$.
Thus the quotient $L/\mathfrak{r}$ is again a restricted $\F$-Lie algebra.
A homomorphism of restricted Lie algebras is a homomorphism of Lie algebras that commutes with the
operations $\argu^{[p]}$.

\begin{defi}
Let $L$ be a restricted Lie algebra over $\F$.
An associative $\F$-algebra $\mathcal{U}_p(L)$ equipped with a monomorphism of restricted Lie algebras
\begin{equation}\label{1eq:universal envelope embedding}
 \psi_L\colon L\longrightarrow \mathcal{U}_p(L)
\end{equation}
is called {\bf universal restricted envelope} of $L$,
if it is a restricted envelope for $L$ and it satisfies the usual universal property:
for any restricted Lie algebras homomorphism $\phi\colon L\rightarrow B_L$, with $B$ a $\F$-algebra,
there exists a unique $\F$-algebra homomorphism $\tilde\phi\colon\mathcal{U}_p(L)\rightarrow B$ such that
$\tilde\phi\circ\psi_L=\phi$.
\end{defi}

For every restricted Lie algebras there exists a unique universal restricted envelope
\cite[Prop.~1.2.4]{jochen:thesis}, and the morphism $\psi_L$ is injective \cite[Prop.~1.2.5]{jochen:thesis}.

Free restricted Lie algebras over $\F$ are defined via the usual property.
It is possible to construct such a restricted Lie algebra as follows.
Given a set of indeterminates $\mathcal{X}=\{X_i,i\in\mathcal{I}\}$, let $\F\langle \mathcal{X}\rangle$ be
the free associative $\F$-algebra on $\mathcal{X}$, and let $L$ be the restricted subalgebra of
$\F\langle \mathcal{X}\rangle_L$ generated by $\mathcal{X}$.
Then $L$ is the free restricted $\F$-Lie algebra with free generating set $\mathcal{X}$,
denoted by $L_p(\mathcal{X})$.
In particular, one has the following (cf. \cite[Prop.~1.2.7]{jochen:thesis}).

\begin{prop}\label{1prop:free restricted lie algebras}
The restricted universal envelope of the free restricted Lie algebra $L_p(\mathcal{X})$ is the algebra 
$\F\langle \mathcal{X}\rangle$, via the monomorphism which sends every $X_i$ to itself.
\end{prop}

\subsection{The restricted Lie algebra induced by the Zassenhaus filtration}

We are interested in restricted Lie algebras, as the Zassenhaus filtration of a pro-$p$ group induces a structure of
restricted Lie algebra over $\F_p$.
Indeed, given a pro-$p$ group put 
\[L_i(G)=D_i(G)/D_{i+1}(G)\quad\text{and}\quad L_\bullet(G)=\bigoplus_{i=1}^\infty L_i(G)\]
for $i\geq1$.
Then every $L_i(G)$ is a $\F_p$-vector space, and the graded space $L_\bullet(G)$ comes equipped with
two further operations.
Let $v\in L_i(G)$, $v=xD_{i+1}(G)$, and $w\in L_j(G)$, $w=yD_{j+1}(G)$, for some $x,y\in G$.
Then one may define
\begin{eqnarray*}
 && (v,w)=[x,y]D_{i+j+1}\in L_{i+j}(G),\\
 && v^{[p]}=x^pD_{pi+1}\in L_{pi}(G),
\end{eqnarray*}
which make $L_\bullet(G)$ a restricted Lie algebra over $\F_p$ (cf. \cite[Theorem~12.8.(i)]{ddsms:analytic}).

Moreover, the restricted $\F_p$-Lie algebra $L_\bullet(G)$ of a pro-$p$ group $G$ is thightly related
to the group algebra $\F_p[G]$.
Let $\omega_{\F_p[G]}$ be the augmentation ideal of $\F_p[G]$ (i.e., $\omega_{\F_p[G]}$ is
the two-sided ideal of $\F_p[G]$ generated by the elements $x-1$, with $x\in G$), and put
\begin{eqnarray*}
 && \grad_1(G)=\F_p[G]/\omega_{\F_p[G]}\simeq\F_p,\quad\grad_{i+1}(G)=\omega_{\F_p[G]}^i/\omega_{\F_p[G]}^{i+1},\\
 && \grad_\bullet(G)=\bigoplus_{i=1}^\infty\grad_i(G)
\end{eqnarray*}
for $i\geq1$.
Then every $\grad_i(G)$ is a $\F_p$-vector space, and $\grad_\bullet(G)$ is a graded associative $\F_p$-algebra.
Then one has the following result, due to S.~Jennings (cf. \cite[Theorems~12.8-12.9]{ddsms:analytic}).

\begin{thm}\label{1thm:envelope Zassenhaus filtration}
Let $\vartheta\colon L_\bullet(G)\rightarrow\grad_\bullet(G)$ be the morphism of graded $\F_p$-algebras defined by
\[\vartheta(xD_{i+1})=(x-1)+\omega_G^{i+1},\quad \text{for }xD_{i+1}\in L_i(G), i\geq1.\]
Then $\vartheta$ is a monomorphism of restricted Lie algebras, and $\grad_\bullet(G)$ is
the universal restricted envelope of $L_\bullet(G)$.
In particular, one has
\[D_i(G)=(1+\omega_G^i)\cap G\]
for every $i\geq1$.
\end{thm}

Let $x$ be an element of $G$, and let $i\geq1$ such that $x\in D_i(G)\smallsetminus D_{i+1}(G)$.
Then the image $\bar x$ of $x$ in the quotient $D_i(G)/D_{i+1}(G)$ is said to be the {\it initial form} of
$x$ in $L_\bullet(G)$ (and in $\grad_\bullet(G)$).

\begin{rem}
 The restricted Lie algebra $L_\bullet(G)$ of a pro-$p$ group $G$ and the restricted envelope $\grad_\bullet(G)$
are to be intended as ``linearizations'' of the group itself: since Lie algebras and associative algebras are 
easier to handle than groups, they are supposed to provide new tools to study pro-$p$ groups.
\end{rem}

\subsection{Restricted Lie algebras of free groups}

The class of pro-$p$ groups for which the induced restricted Lie algebras are best understood are free pro-$p$ groups.
Let $F$ be a finitely generated free pro-$p$ groups with minimal generating system $\{x_1,\ldots, x_d\}$,
with $d=d(F)$, and let $\mathcal{X}=\{X_1,\ldots,X_d\}$ be a set of non commutative indeterminates.
Then one has the following (cf. \cite[Theorem~1.3.8]{jochen:thesis}).

\begin{thm}\label{1thm:L for free groups}
Let $F$ be a finitely generated free pro-$p$ group, and let $\mathcal{X}$ be as above, with $d=d(F)$.
The restricted Lie algebra over $\F_p$ induced by the Zassenhaus filtration of $F$ is free, i.e., one has 
\[L_\bullet(F)\overset{\sim}{\longrightarrow} L_p(\mathcal{X}),\quad \bar x_i\longmapsto X_i\]
for every $i=1,\ldots,d$, where $\bar x_i\in L_1(F)$ is the initial form of $x_i$.
\end{thm}

In particular, by Proposition~\ref{1prop:free restricted lie algebras} and
Theorem~\ref{1thm:envelope Zassenhaus filtration}, the graded algebra $\grad_\bullet(F)$ is isomorphic to
the algebra $\F_p\langle\mathcal{X}\rangle$, considered with the grading induced by the degree of the monomials.

\begin{rem}\label{1rem:magnus algebra}
Let $\F_p\dbml\mathcal{X}\dbmr$ be the {\it Magnus algebra}, i.e., the $\F_p$-algebra of formal series on the 
non-commutative indeterminates $\mathcal{X}$. 
It is well known that one has a continuous isomorphisms $\F_p\dbl F\dbr\simeq\F_p\dbml\mathcal{X}\dbmr$
given by $x_i\mapsto1+X_i$ for every $i=1,\ldots,d(F)$.
Then the powers of the augmentation ideal $\omega_{\F_p\dbl F\dbr}$ of $\F_p\dbl F\dbr$, and the powers of the
ideal $\langle X_1,\ldots,X_{d(F)}\rangle$ of the Magnus algebra induce an isomorphism of graded $\F_p$-algebras
$\grad_\bullet(F)\simeq\F_p\langle\mathcal{X}\rangle$ given by $\bar x_i\mapsto X_i$ for every $i=1,\ldots,d(F)$.
\end{rem}

One has also the following useful property (cf. \cite[Prop.~1.2.6]{jochen:thesis})

\begin{prop}\label{1prop:presentation restricted envelope}
Let $L$ be a restricted $\F$-Lie algebra and let $\mathfrak{r}\vartriangleleft L$ be an ideal.
Let $\mathcal{R}$ denote the left ideal of $\mathcal{U}_p(L)$ generated by $\psi_L(\mathfrak{r})$
Then $\mathcal{R}$ is a two-sided ideal and
the canonical epimorphism $L\twoheadrightarrow L/\mathfrak{r}$ induces an exact sequence
\[\xymatrix{ 0\ar[r] & \mathcal{R}\ar[r] & \mathcal{U}_p(L)\ar[r] & \mathcal{U}_p(L/\mathfrak{r})\ar[r] & 0 }.\]
\end{prop}

\section[Kummer theory]{Cohomology of Galois groups and Kummer theory}

Fix a field $K$, and let $L/K$ be a Galois extension of fields.
Then the multiplicative group $L^\times$ is a module of the Galois group $\Gal(L/K)$.
The celebrated Hilbert's Satz 90 (Noether's version) states that the first cohomology group of $\Gal(L/K)$
with coefficients in $L^\times$ is trivial, i.e.,
\begin{equation}\label{1eq:HilbertSatz90}
 H^1\left(\Gal(L/K),L^\times\right)=1.
\end{equation}
Given a field $K$ and a positive integer $m$, let $\mu_m=\mu_m(K)$ be the group of the roots of unity of order $m$
contained in $K$, namely,
\[ \mu_m(K)=\left\{\zeta\in K^\times\:\left|\;\zeta^m=1\right.\right\}.\]
Thus, one has the short exact sequence of $G_K$-modules
\begin{equation}\label{1eq:mu_rootsofunity}
\xymatrix{ 1\ar[r] & \mu_m(\bKsep)\ar[r] & (\bKsep)^\times\ar[r]^m & (\bKsep)^\times\ar[r] & 1,}
\end{equation}
which induces the exact sequence in cohomology
\[K^\times\overset{m}{\longrightarrow} K^\times \longrightarrow H^1\left(G_K,\mu_m(\bKsep)\right)
\longrightarrow H^1\left(G_K,(\bKsep)^\times\right)=0,\]
and hence an isomorphism $K^\times/(K^\times)^m \simeq H^1(G_K,\mu_m(\bKsep))$.
In particular, assume that the field $K$ contains a primitive $m^{th}$-root of unity, i.e.,
$\mu_m(\bKsep)\subseteq K$.
Then $\mu_m$ is a trivial $G_K$-module which is isomorphic to $\Z/m\Z$.
This yields the isomorphism
\begin{equation}\label{1eq:kummer_duality}
 \xymatrix{\phi_K\colon\dfrac{K^\times}{\left(K^\times\right)^m}\ar[r]^-{\sim} & H^1\left(G_K,\Z/m\Z\right)
=\Hom\left(G_K,\Z/m\Z\right),}
\end{equation}
where the equality follows from Fact~\ref{1fact:elementary cohomology}, called the {\it Kummer isomorphism}.
In particular, fix an isomorphism of $G$-modules $\mu_m\simeq\Z/m\Z$.
Then one has
\begin{equation}\label{1eq:kummer_duality2}
 \phi_K(\bar a)(\sigma)=\frac{\sigma\left(\sqrt[m]{a}\right)}{\sqrt[m]{a}}\in\mu_m,
\end{equation}
with $\bar a=a\bmod(K^\times)^m$ (cf. \cite[(3.1)]{cmq:fast}).

Now let $K(p)$ be the {\bf $p$-closure} of $K$, i.e., $K(p)/K$ is the maximal $p$-extension of $K$,
and let $G_K(p)=\Gal(K(p)/K)$ be the {\it maximal pro-$p$ Galois group} of $K$.
Assume that $K$ contains all the roots of unity of $p$-power order.
Then Hilbert's Sats 90 holds for the extension $K(p)/K$, and \eqref{1eq:kummer_duality2}, together with a fixed
isomorphism of trivial $G_K(p)$-modules $\mu_p\simeq\F_p$, yields an isomorphism
\begin{equation}
 \label{1eq:kummer_dualityp}
 K^\times/(K^\times)^p \simeq H^1(G_K(p),\F_p)
\end{equation}

Recall from Proposition~\ref{1prop:frattini subgroup} that for a pro-$p$ group $G$, the quotient $G/\Phi(G)$
is a $\F_p$-vector space.
Then one has the following (cf. \cite[Prop.~3.9.1]{nsw:cohm}).

\begin{cor}
 Let $G$ be a pro-$p$ group.
Then the cohomology group $H^1(G,\F_p)$ is the Pontryagin dual of the quotient $G/\Phi(G)$, i.e.,
\begin{equation}\label{1eq:pontryagin_duality} 
 H^1(G,\F_p)=\Hom(G,\F_p) = \Hom(G/\Phi(G),\I_p)=(G/\Phi(G))^\vee,
\end{equation}
and the groups $H^1(G,\F_p)$ and $G/\Phi(G)$ are isomorphic as (discrete) $\F_p$-vector spaces.
In particular, if $G$ is finitely generated, then 
\[\dim_{\F_p}(H^1(G,\F_p))=\dim_{\F_p}(G/\Phi(G))=d(G).\]
\end{cor}

Thus, if $\mathcal{X}=\{x_i,i\in \mathcal{I}\}\subset G$ is a minimal system of generators of $G$,
the dual $\mathcal{X}^\vee$ is a basis for $H^1(G,\F_p)$, with $\mathcal{X}^\vee=\{\chi_i,i\in\mathcal{I}\}$
such that $\chi_i(x_j)=\delta_{ij}$ for every $i,j\in\mathcal{I}$.
Therefore, the isomorphisms \eqref{1eq:kummer_dualityp} and \eqref{1eq:pontryagin_duality} yield the isomorphism
of discrete $\F_p$-vector spaces \[K^\times/(K^\times)^p \simeq G_K(p)/\Phi\left(G_K(p)\right)\]
(with $K$ satisfying the above conditions).
Hence, if $\{a_i,i\in\mathcal{I}\}\subset K^\times\smallsetminus(K^\times)^p$ is a set of representatives
of $K^\times/(K^\times)^p$, one has the minimal generating sistem $\mathcal{X}$ of $G_K(p)$ such that
\begin{equation}\label{1eq:kummer_duality explicitely}
\dfrac{x_j(\sqrt[p]{a_i})}{\sqrt[p]{a_i}}=\left\{\begin{array}{cc} \zeta & \text{if }i=j \\ 1 & \text{if }i\neq j
\end{array}\right. 
\end{equation}
with $\zeta\in K$ a fixed primitive $p$-th root of unity, and the dual basis
$\mathcal{X}^\vee\subset H^1(G_K(p),\F_p)$.

Finally, one has this classic result (see also \cite[\S~3.6]{wdg:lux} and \cite[\S~3]{bogomolovtschinkel:birat}).

\begin{prop}\label{1prop:roots and kummer}
Assume $K$ contains a primitive $p$-th root of unity, and let $K(\sqrt[p]{K})=K(\sqrt[p]{a},a\in K)$.
Then $K(\sqrt[p]{K})/K$ is Galois, and it is the compositum of all finite cyclic extensions of $K$ of degree $p$. 
Its Galois group is
\[\Gal\left(K(\sqrt[p]{K})/K\right)=G_K(p)/\Phi(G_K(p))\simeq \left(K^\times/(K^\times)^p\right)^\vee,\]
i.e., $G_{K(\sqrt[p]{K})}=\Phi(G_K(p))$.
\end{prop}

One has an arithmetical interpretation also for the second cohomology group of a Galois group. 
Indeed, for a Galois extension $L/K$ one has canonical isomorphisms
\[ \xymatrix{ H^2(\Gal(L/K),\bar L^{\times})\ar[r]^-{\sim} & \bra(L/K) }, \]
where $\bra(L/K)$ is the {\it Brauer group} of $L/K$ (cf. \cite[Theorem~6.3.4]{nsw:cohm}).
In particular, one has $H^2(G_K,(\bar K^{\sep})^\times)\simeq\bra(K)$, with $\bra(K)$ the Brauer group of $K$.
Thus, for a positive integer $m$ the short exact sequence \eqref{1eq:mu_rootsofunity},
together with Hilbert's Satz 90, induces the exact sequence in cohomology
\[ \xymatrix{ 0\ar[r] & H^2(G_K,\mu_m)\ar[r] & \bra(K)\ar[r]^m & \bra(K) }, \]
which implies the following.

\begin{cor}\label{1cor:H2 and Brauer group}
 Let $m$ be a positive integer such that the characteristic of $K$ does not divide $m$.
Then $H^2(G_K,\mu_m)$ is (canonically) isomorphic to $\bra_m(K)$, the subgroup of the Brauer group of $K$
of exponent $m$.
\end{cor}


\chapter{The Bloch-Kato conjecture}

\section{Quadratic algebras}

Fix a ring $R$, and let $A_\bullet$ be a non-negatively graded $R$-algebra, i.e.,
$A_\bullet=\bigoplus_{n\geq0}A_n$ with $A_0=R$ and $A_n$ a $R$-module for every $n\geq1$,
together with a product such that if $a\in A_n$ and $b\in A_m$, then $ab\in A_{n+m}$.

A graded algebra $A_\bullet$ is said to be {\bf quadratic} if $A_n$ is generated by products of $n$ elements of $A_1$, with $n>0$,
and the defining relations of $A_\bullet$ are generated by elements of degree 2: namely
\begin{equation}\label{2eq:quadratic algebra}
 A_\bullet\simeq\frac{\mathcal{T}^\bullet(A_1)}{\mathcal{R}},\quad
\mathcal{T}^\bullet(A_1)=\bigoplus_{n\geq0}A_1^{\otimes n},
\end{equation}
where $\mathcal{R}$ is a two-sided ideal of $\mathcal{T}^\bullet(A_1)$ generated by elements of $A_1\otimes_RA_1$.

\begin{exs}Let $\F$ be a field, and let $V$ be a $\F$-vector space.
\begin{itemize}
 \item[(a)] 
The symmetric algebra $S_\bullet(V)$ is a quadratic $\F$-algebra, as
\[S_\bullet(V)=\frac{\mathcal{T}^\bullet(V)}{\langle v\otimes w-w\otimes v\rangle},
\quad v,w\in V.\]
 \item[(b)] The exterior algebra $\bigwedge_\bullet(V)$ is a quadratic $\F$-algebra, as
\[\bigwedge_{n\geq0}(V)=\frac{\mathcal{T}^\bullet(V)}{\langle v\otimes w+w\otimes v\rangle},
\quad v,w\in V.\]
\end{itemize}
\end{exs}

\begin{defi}
Fix a field $K$.
For a positive integer $n$, the {\bf $n^{th}$ Milnor $\MiK$-group} of $K$ is the quotient 
\[\MiK_n^M(K)=\left(K^\times\otimes_{\Z}\cdots\otimes_{\Z} K^\times\right)/I_n,\]
where $I_n$ is the subgroup generated by the elements $a_1\otimes\cdots\otimes a_n$
such that $a_i+a_j=1$ for some $i\neq j$.
(Such relations are called of {\it Steinberg type}.)
Moreover, set $\MiK_0^M=\Z$.
Then the {\bf Milnor $\MiK$-ring} of $K$ is the quadratic $\Z$-algebra
\[\MiK_\bullet^M(K)=\bigoplus_{n\geq0}\MiK_n^M(K)=
\frac{\mathcal{T}^\bullet(K^\times)}{\langle a\otimes b\:|\;a+b=1\rangle},\]
equipped with the addition induced by the multiplication of $K$,
and with the multiplication induced by the tensor product.
\end{defi}

Given $a_1,\ldots, a_n\in K^\times$, the symbol $\{a_1,\ldots,a_n\}$ denotes the image
of $a_1\otimes\cdots\otimes a_n$ in $\MiK_n^M(K)$.
A particular feature of Milnor $\MiK$-rings of fields is the anti-commutativity:
\begin{equation}\label{2eq:anticommutativity_MiK}
 \{a,b\}=-\{b,a\}\quad\text{ for all }a,b\in K^\times
\end{equation}
(cf. \cite[Prop.~24.1.2]{efrat:libro}).
Therefore, the algebra $\MiK_\bullet^M(K)$ is in fact a quotient of the exterior algebra $\bigwedge_\bullet K^\times$.

For a positive integer $m$, one may consider the Milnor $\MiK$-ring modulo $m$, i.e., 
\[\MiK_\bullet^M(K)/m=\bigoplus_{n\geq0}\MiK_n^M(K)/m.\MiK_n^M(K).\]
Then $\MiK_\bullet^M(K)/m$ is a quadratic, anti-commutative $\Z/m\Z$-algebra.
In particular, one has the following (cf. \cite[Prop.~24.1.3]{efrat:libro}).

\begin{prop}
 For positive integers $m$, one has
\[\MiK_\bullet^M(K)/m=\frac{\mathcal{T}^\bullet\left(K^\times/(K^\times)^m\right)}
{\langle \bar a\otimes \bar b\:|\;a+b\in (K^\times)^m\rangle},\]
where $\bar a=a\bmod (K^\times)^m$ and $\bar b=b\bmod (K^\times)^m$.
\end{prop}

\subsection{Constructions with quadratic algebras}

Let $A_\bullet$ and $B_\bullet$ be two quadratic algebras over a field $\F$.
There are several constructions which produce a new quadratic algebra $C_\bullet$ generated by $A_1$ and $B_1$,
i.e., \[C_\bullet=\frac{\mathcal{T}^\bullet(A_1\oplus B_1)}{\mathcal{R}},\]
with $\mathcal{R}$ the two-sided ideal generated by $\mathcal{R}_1\leq(A_1\oplus B_1)^{\otimes2}$.

Let $\mathcal{R}_1^A\leq A_1\otimes A_1$ and $\mathcal{R}_1^B\leq B_1\otimes B_1$ be the relations of $A_\bullet$,
resp. of $B_\bullet$.
Then one has the following constructions.
\begin{itemize}
 \item[(a)] The {\it direct product} $C_\bullet=A_\bullet\sqcap B_\bullet$ (also denoted as $A_\bullet\oplus B_\bullet$)
is such that $C_n=A_n\oplus B_n$ for every $n\geq1$.
Thus, one has
\[\mathcal{R}_1=\mathcal{R}_1^A\oplus\mathcal{R}_1^B\oplus(A_1\otimes B_1)\oplus (B_1\otimes A_1).\]
In particular, $ab=ba=0$ for any $a\in A_n$, $b\in B_m$, $n,m\geq1$.
 \item[(b)] The {\it free product} $C_\bullet=A_\bullet\sqcup B_\bullet$ (also denoted as $A_\bullet\ast B_\bullet$)
is the associative algebra generated freely by $A_\bullet$ and $B_\bullet$.
One has
\[\mathcal{R}_1=\mathcal{R}_1^A\oplus\mathcal{R}_1^B.\]
 \item[(c)] The {\it symmetric tensor product} $A_\bullet\otimes^1 B_\bullet$ (also called 1-tensor product)
is such that $ab=ba$ for any $a\in A_n$, $b\in B_m$, $n,m\geq1$.
Thus,
\[\mathcal{R}_1=\mathcal{R}_1^A\oplus\mathcal{R}_1^B\oplus\langle ab-ba\rangle,\]
with $a\in A_1$ and $b\in B_1$.
 \item[(d)] The {\it skew-commutative tensor product} $A_\bullet\otimes^{-1} B_\bullet$ (also called $(-1)$-tensor product,
or graded tensor product, see \cite[\S~23.2]{efrat:libro})
is such that $ab=(-1)^{nm}ba$ for any $a\in A_n$, $b\in B_m$, $n,m\geq1$.
Thus,
\[\mathcal{R}_1=\mathcal{R}_1^A\oplus\mathcal{R}_1^B\oplus\langle ab+ba\rangle,\]
with $a\in A_1$ and $b\in B_1$.
\end{itemize}
See \cite[\S~3.1]{pp:quadratic algebras} for further details.

\section[The RVW Theorem]{The Rost-Voevodsky Theorem}

Assume now that the characteristic of $K$ does not divide $m$, and let $\mu_m$ be the group of the $m^{th}$-roots
of unity as in (\ref{1eq:mu_rootsofunity}).
The isomorphism (\ref{1eq:kummer_duality}) induces the isomorphism 
$\MiK_1(K)/m\simeq H^1(G_K,\mu_m)$.
Thus the cup-product induces a morphism
\begin{equation}\label{2eq:normresidue1}
  \frac{K^\times}{(K^\times)^m}\otimes_{\Z/m\Z}\cdots\otimes_{\Z/m\Z}\frac{K^\times}{(K^\times)^m}
\longrightarrow H^n\left(G_K,\mu_m^{\otimes n}\right)
\end{equation}
defined by $\bar a_1\otimes\cdots\otimes \bar a_n \mapsto \chi_{a_1}\cup\cdots\cup\chi_{a_n}$,
whith the $\chi_{a_i}$'s as in \eqref{1eq:kummer_duality explicitely}.
In fact, J.~Tate showed in \cite{tate:K2cohomology} that the map (\ref{2eq:normresidue1}) factors through
$\MiK_n^M(K)/m$, thus one has the homomorphism 
\[ h_K\colon \MiK_n^M(K)/m \longrightarrow H^n\left(G_K,\mu_m^{\otimes n}\right)\]
(cf. \cite[Theorem~6.4.2]{nsw:cohm}) called the {\bf Galois symbol} (or the {\bf norm residue map}) of degree $n$.

\begin{thm}[Bloch-Kato Conjecture]\label{2thm:BKconjecture}
 For every field $K$ and for every positive $m$ prime to $\chr(K)$, the Galois symbol $h_K$ is an isomorphism
for every degree $n$.
\end{thm}

For $m=2$ this was conjectured first by J.~Milnor in the late '60s, whereas the full formulation of the conjecture
was stated by S.~Bloch and K.~Kato in \cite{bk:BK}.
The complete proof required several steps:
\begin{itemize}
 \item[(i)] for $n=1$ the Galois symbol is just the Kummer isomorphism;
 \item[(ii)] for $n=2$ it was proved by A.S.~Merkur'ev and A.A.~Suslin in the early '80s (cf. \cite{merkurjevsuslin:1});
 \item[(iii)] for $n=3$ and $m=2$ it was proved by Merkur'ev and Suslin in the late '80s
(cf. \cite{merkurjevsuslin:2}) and, independently, by M.~Rost;
 \item[(iv)] for any $n$ and $m$ a power of 2 it was proved by V.~Voevodsky at the beginning of the century (cf. \cite{voevodsky:1}).
\end{itemize}
For this, Voevodsky was awarded the Fields Medal in 2002.
Later, he announced a proof of the full conjecture in a preprint.
Yet, it required a while to see a complete (and published) proof, which appeared in 2011 (cf. \cite{voevodsky:2}).
Also Rost gave a remarkable and substantial contribution to achieve the final proof -- further, C.~Weibel solved
some technical problems via what is now known as the ``Chuck Weibel's patch'' -- so that now the Bloch-Kato
Conjecture is adressed as the {\bf Rost-Voevodsky theorem}, see also \cite{weibel:proof1,weibel:proof2}.

\section{Bloch-Kato pro-$p$ groups}

Let $p$ be a prime.
If $K$ contains a primitive $p^{th}$-root of unity, then $\mu_p$ is isomorphic to $\F_p$ as (trivial) $G_K$-modules.
In particular, one has $\mu_p^{\otimes n}\simeq\F_p$ for every $n\geq1$.
Therefore, one has the cup product 
\[\xymatrix{H^n(G_K,\F_p)\times H^m(G_K,\F_p)\ar[r]^-{\cup}&H^{n+m}(G_K,\F_p)},\]
with $n,m\geq0$, and by Remark~\ref{1rem:cup product} we may define the graded-commutative $\F_p$-algebra
\[H^\bullet(G_K,\F_p)=\bigoplus_{n\geq0}H^n(G_K,\F_p),\] equipped with the cup-product.
Since the Galois symbol is defined by the cup product, $h_K$ induces an isomorphism of graded $\F_p$-algebras
\[ \xymatrix{h_K\colon \MiK_\bullet^M(K)/p\ar[r]^-{\sim} & H^\bullet(G_K,\F_p)},\]
which implies that the $\F_p$-cohomology ring of $G_K$ is a quadratic $\F_p$-algebra.

For a profinite group $G$ let $\mathcal{O}_p(G)$ be the subgroup
\[\mathcal{O}_p(G)=\left\langle C\in \text{Syl}_\ell(G)\:|\; \ell\neq p\right\rangle,\]
where $\text{Syl}_\ell(G)$ is the set of the Sylow pro-$\ell$ subgroups;
namely one has the short exact sequence
\begin{equation}\label{2eq:ses_mxlpquot}
 1\longrightarrow \mathcal{O}_p(G) \longrightarrow G\longrightarrow G(p)\longrightarrow 1,
\end{equation}
with $G(p)$ the maximal pro-$p$ quotient of $G$, i.e., the Galois group of the extensions $K(p)/K$ (cf. \cite[\S~2]{claudio:BK}).
Therefore, in the case $G$ is the absolute Galois group $G_K$ of a field with the above properties,
$\mathcal{O}_p(G_K)$ is the absolute Galois group of the $p$-closure $K(p)$,
and by the Bloch-Kato Conjecture the cohomology ring $H^\bullet(\mathcal{O}_p(G_K),\F_p)$ is quadratic,
and thus generated by $H^1(\mathcal{O}_p(G_K),\F_p)$.
Since $\mathcal{O}_p(G_K)$ is $p$-perfect, by Fact~\ref{1fact:elementary cohomology} one has
\[H^1\left(\mathcal{O}_p(G_K),\F_p\right)=\Hom_{\F_p}\left(\mathcal{O}_p(G_K),\F_p\right)=0,\]
and hence $H^n(\mathcal{O}_p(G_K),\F_p)=0$ for every $n\geq1$.
Thus in the Lyndon-Hochschild-Serre spectral sequence (cf. \cite[\S~II.4]{nsw:cohm})
induced by (\ref{2eq:ses_mxlpquot}), the terms
\[E_2^{rs}=H^r\left(G_K(p), H^s(\mathcal{O}_p(G_K))\right)\]
vanish for $s>0$, and the spectral sequence collapses at the $E_2$-term.
Hence the inflation map
\begin{equation}\label{2eq:inflation BlochKato}
 \text{inf}_{G_K,\mathcal{O}_p(G_K)}^n\colon H^n\left(G_K(p),\F_p\right)\longrightarrow H^n\left(G_K,\F_p\right)
\end{equation}
is an isomorphism for every $n\geq0$ (cf. \cite[Lemma~2.1.2]{nsw:cohm}).
This proves the following.

\begin{prop}
 Let $K$ be a field containing a primitive $p^{th}$-root of unity, with $p$ a fixed prime.
Then the $\F_p$-cohomology ring $H^\bullet(G_K(p),\F_p)$ of the maximal pro-$p$ Galois group of $K$
is a quadratic $\F_p$-algebra.
\end{prop}

\begin{defi}[\cite{claudio:BK}]
 Let $G$ be a pro-$p$ group.
Then $G$ is said to be a {\bf Bloch-Kato pro-$p$ group} if for every closed subgroup $C\leq G$ the 
$\F_p$-cohomology ring $H^\bullet(C,\F_p)$ is a quadratic $\F_p$-algebra.
\end{defi}

Thus, the maximal pro-$p$ Galois group $G_K(p)$ of a field $K$ containing $\mu_p$ is a Bloch-Kato pro-$p$ group:
indeed, every closed subgroup of the maximal pro-$p$ Galois group of a field is again a maximal pro-$p$ Galois group
(and this justifies such a requirement in the definition).

By Proposition~\ref{1prop:cup product}, the cohomology ring of a pro-$p$ group $G$ with $\F_p$-coefficients is skew-commutative.
In particular, for $G$ a Bloch-Kato pro-$p$ group let $\{\chi_i,i\in \mathcal{I}\}$ be a basis
for $H^1(G,\F_p)$ as $\F_p$-vector space, and let $\mathcal{R}$ be the defining relations of $H^\bullet(G,\F_p)$
as in \eqref{2eq:quadratic algebra}, i.e.,
\begin{equation}\label{2eq:quadratic cohomology for BK 1}
 H^\bullet(G,\F_p)\simeq\frac{\mathcal{T}^\bullet\left(H^1(G,\F_p)\right)}{\mathcal{R}}.
\end{equation}
Then one has $\chi_i\otimes\chi_j+\chi_j\otimes\chi_i\in \mathcal{R}$ for every $i,j\in\mathcal{I}$,
and \eqref{2eq:quadratic cohomology for BK 1} induces an epimorphism of quadratic $\F_p$-algebras
\begin{equation}\label{2eq:BK epimorphism exterior algebra}
 \xymatrix{\bigwedge_\bullet H^1(G,\F_p)\ar@{->>}[rr] && H^\bullet(G,\F_p)},
\end{equation}
so that, in the case $p$ odd, we may consider the exterior algebra generated by $H^1(G,\F_p)$ as
an ``upper bound'' for the $\F_p$-cohomology of Bloch-Kato pro-$p$ groups, namely,
this is the case when the ideal $\mathcal{R}$ is the smallest possible.
On the other hand, the ideal $\mathcal{R}$ may be generated by the whole group $H^2(G,\F_p)$, i.e., the
$\F_p$-cohomology ring of $G$ is concentrated in degrees 0 and 1, and this case can be considered as 
``lower bound''.
The latter case is completely explained by the following classic result
(cf. \cite[Prop.~3.3.2 and Prop.~3.5.17]{nsw:cohm}).

\begin{prop}\label{2prop:cohomological dimension prop}
Let $G$ be a pro-$p$ group.
\begin{itemize}
 \item[(i)] for $n\geq0$, one has $\ccd(G)=n$ if, and only if, $H^{n+1}(G,\F_p)=0$.
 \item[(ii)] the group $G$ is free if, and only if, $\ccd(G)=1$.
\end{itemize}
\end{prop}

In particular, if $p$ is odd and $G$ is a finitely generated Bloch-Kato pro-$p$ group, then
\eqref{2eq:BK epimorphism exterior algebra} and Proposition~\ref{2prop:cohomological dimension prop}
imply the following (see also \cite[Prop.~4.1 and Prop.~4.3]{claudio:BK}).

\begin{cor}\label{1cor:cd d r for BK}
For $G$ Bloch-Kato and $p$ odd, one has the following:
\begin{itemize}
 \item[(i)] $\ccd(G)\leq d(G)$, and the equality holds if, and only if \eqref{2eq:BK epimorphism exterior algebra}
is an isomorphism;
 \item[(ii)] $r(G)\leq\binom{d(G)}{2}$.
\end{itemize}
\end{cor}

The celebrated Artin-Schreier theorem states that the only non-trivial finite subgroups
of an absolute Galois group is the cyclic group of order two.\footnote{In particular, the Artin-Schreier theorem
states that if the absolute Galois group $G_K$ of a field $K$ is finite (and non-trivial), then $G_K\simeq\Z/2\Z$
and $K$ is a {\it real closed field}, i.e., $K$ is of characteristic 0 with an unique {\it ordering} such that
every positive element is a square and every polynomial in $K[X]$ of odd degree has a zero in $K$.}
A similar condition holds for Bloch-Kato pro-$p$ groups:
a Bloch-Kato pro-$p$ group $G$ may have torsion only if $p=2$.
In particular one has the following.

\begin{prop}\label{2prop:tor p}
Bloch-Kato pro-$p$ group $G$ is torsion if, and only if, $G$ is abelian and of exponent 2,
and moreover any of such group is a Bloch-Kato pro-2 group.
\end{prop}

\begin{proof}
Let $p$ be any prime and let $C_p$ be the cyclic group of order $p$, and assume that $G$ admits a finite subgroup.
Then $G$ has $C_p$ as a (closed) subgroup.
If $p$ is odd, then the cohomology ring $H^\bullet(C_p,\F_p)$ is not quadratic, as $H^2(C_p,\F_p)$
contains an element which is not a combination of cup products of elements in $H^1(C_p,\F_p)$, i.e.,
the map \eqref{2eq:BK epimorphism exterior algebra} is not an epimorphism in degree 2.
Thus $G$ can't be a Bloch-Kato pro-$p$ group.

Moreover, if $p=2$ and $G$ has $2^n$-torsion, with $n>1$, the above argument works also for $C_{2^n}$ instead
of $C_p$, with $C_{2^n}$ the cyclic group of order $2^n$.

On the other hand, if $p=2$ and $G$ is abelian and of exponent 2, then the cohomology ring $H^\bullet(G,\F_2)$
is isomorphic to the symmetric $\F_2$-algebra $S_\bullet(H^1(G,\F_2))$, and thus $G$ is Bloch-Kato.
\end{proof}

For Bloch-Kato pro-$p$ groups one may deduce the following result from \cite{cem:quotients}.

\begin{thm}\label{2thm:cem}
 Let $G$ be a Bloch-Kato pro-$p$ group.
Then the inflation map induces an isomorphism
\[ \xymatrix{ H^\bullet(G/\lambda_3(G),\F_p)_{\text{dec}}\ar[r]^-{\sim} & H^\bullet(G,\F_p), } \]
where $H^\bullet(\argu,\F_p)_{\text{dec}}$ is the {\it decomposable part} of the cohomology ring, i.e.,
the subalgebra generated by products of elements of degree one.
Moreover, if $\phi\colon G_1\to G_2$ is a homomorphism of Bloch-Kato pro-$p$ group, then 
the induced map in cohomology \[\phi^\ast\colon H^\bullet(G_2,\F_p)\longrightarrow H^\bullet(G_1,\F_p)\]
is an isomorphism if, and only if, the induced map on the quotients $G_1/\lambda_3(G_1)\to G_2/\lambda_3(G_2)$
is an isomorphism.
\end{thm}

We shall study more in detail the quotient $G/\lambda_3(G)$ in Section~\ref{3sec:finite quotients}

\section{Oriented pro-$p$ groups}

Fix a prime $p$ and a field $K$, and let $\mu_{p^\infty}=\mu_{p^\infty}(K)$ be
the group of all the roots of unity of $p$-power order, i.e., $\mu_{p^\infty}(K)=\varinjlim_{k\geq0}\mu_{p^k}(K)$.
In particular, $\mu_{p^\infty}$ is a discrete torsion $\Z_p$-module, and one has
\[ \text{Aut}_K\left(\mu_{p^\infty}(\bar K^{\sep})\right)\simeq\Z_p^\times,\]
with $\text{Aut}_K(\mu_{p^\infty})$ the group of $K$-automorphisms (cf. \cite[Lemma~1]{ware:galp}).
Therefore, the action of the absolute Galois group on the separable closure induces an homomorphism of
pro-$p$ groups
\begin{equation}\label{2eq:Aut(mu)}
\theta_{(p,K)}\colon G_K\longrightarrow\Z_p^\times,
\end{equation}
called the $p$-cyclotomic character of $K$.

If $K\supseteq \mu_p$, then the $p$-closure $K(p)$ contains $\mu_{p^\infty}(\bar K^{\sep})$, and the maximal
pro-$p$ Galois group $G_K(p)$ acts on $\mu_{p^\infty}$,
namely, the cyclotomic character $\theta_{p,K}$ is trivial on the subgroup $\mathcal{O}_p(G)$, and it factors
in a unique way through a homomorphism $\theta\colon G_K(p)\rightarrow\Z_p^\times$.
We will call the map $\theta$ the {\it arithmetical orientation} of the Galois group $G_K(p)$.
This is the arithmetical justification for the following definition.

\begin{defi}
Let $G$ be a profinite group.
A continuous homomorphism of profinite groups $\theta\colon G\to\Z_p^\times$ will be called a
{\bf $p$-orientation} of $G$, and the group $G$ will be called a {\bf $p$-oriented} profinite group. 
\end{defi}

Thus, absolute Galois groups (and maximal pro-$p$ Galois groups of fields containing a primitive
$p^{th}$-root of unity) are naturally equipped with a $p$-orientation.

A $p$-orientation $\theta\colon G\to\Z_p^\times$ of a profinite group $G$ induces a natural continuous
action of $G$ on $\Z_p$ in the following way: for every $g\in G$, the multiplication by $\theta(g)$ induces
an automorphisms of $\Z_p$.
Thus $\Z_p(1)$ will denote the $\ZpG$-module isomorphic to $\Z_p$ as $\Z_p$-module,\footnote{
In the arithmetical case, the module $\Z_p(m)$ is called the {\it $m^{th}$ Tate twist} of $\Z_p$, as
this construction was used first by J.~Tate in \cite{tate:K2cohomology} (see also \cite[(7.3.6)]{nsw:cohm}).}
and with the $G$-action induced by the orientation $\theta$.
Also, for $m\in\Z$, one has the $\ZpG$-module $\Z_p(m)$, together with the action
induced by $\theta^m$, i.e., $g.\lambda=\theta(g)^m\cdot\lambda$ for every $g\in G$ and $\lambda\in \Z_p$.

\subsection{$\theta$-abelian pro-$p$ groups}
From now on we will concentrate on pro-$p$ groups with $p$-orientations
(and we will omit the ``$p$'' in the notation, for simplicity).

\begin{rem}\label{1rem:p-adic and orientations}
If $G$ is a pro-$p$ group, then the image of an orientation $\theta\colon G\rightarrow\Z_p^\times$
is again a pro-$p$ group.
\begin{itemize}
 \item If $p$ is odd, then $\Z_p^\times$ is virtually pro-$p$, and the Sylow pro-$p$ group is $1+p\Z_p$,
with $|\Z_p^\times:(1+p\Z_p)|=p-1$.
In particular, one has that $\image(\theta)$ is contained in $1+p\Z_p$, and it is isomorphic to $\Z_p$.
 \item If $p=2$, then $\Z_2^\times$ is pro-$2$, and it is isomorphic to $\Z/2\Z\times\Z_2$.
Thus, one has that $\image(\theta)$ is torsion-free and isomorphic to $\Z_2$ if, and only if,
$\image(\theta)\leq1+4\Z_2$.
\end{itemize}
\end{rem}

In a pro-$p$ group with orientation, one has a distinguished subgroup.

\begin{defi}
 Let $(G,\theta)$ be an oriented pro-$p$ group.
The closed subgroup defined by 
\[\Zen_\theta(G)=\left\{h\in\kernel(\theta)\;\left|ghg^{-1}=h^{\theta(g)}\text{ for all }g\in G\right.\right\}\]
will be called the {\bf $\theta$-centre} of $G$. 
Moreover, the group $G$ will be called {\bf $\theta$-abelian}, if $\Zen_\theta(G)=\kernel(\theta)$.
\end{defi}

Namely, the $\theta$-centre of an oriented pro-$p$ group is given by the ``copies'' of the module $\Z_p(1)$,
induced by the orientation, lying inside $G$.
It is easy to see that $\Zen_\theta(G)$ is an abelian normal subgroup of $G$.
If a pro-$p$ group with an orientation $(G,\theta)$ is $\theta$-abelian, then it has a very simple structure
(see \cite[Prop.~3.4]{claudio:BK} and \cite[Remark~3.3]{cmq:fast}).

\begin{prop}\label{2prop:presentation theta-abelian}
 Let $(G,\theta)$ be a pro-$p$ group with an orientation, ans suppose that $\image(\theta)\simeq\Z_p$.
Then $G$ is $\theta$-abelian if, and only if, there exists a minimal generating system
$\mathcal{X}=\{x_0,x_i,i\in\mathcal{I}\}\subset G$, such that $G$ has a presentation
\begin{equation}\label{2eq:presentation theta-abelian}
 G=\left\langle x_0,x_i\left|\:x_0x_ix_0^{-1}=x_i^{\theta(x_0)},[x_i,x_j]=1, i,j\in\mathcal{I}\right.\right\rangle.
\end{equation}
In particular, $G$ can be expressed as a fibre product
\[G\simeq\image(\theta)\ltimes\Z_p(1)^{\mathcal{I}},\]
with the action induced by the orientation.
\end{prop}

\subsection{Cyclotomic orientations}

Let $K$ be a field containing the $p$th roots of unity $\mu_p$, and let $G_K(p)$ be the maximal pro-$p$ Galois group.
For every $n\geq0$, the $p^n$th power induces the short exact sequence of $\Z_p\dbl G_K(p) \dbr$-modules
\begin{equation}\label{2eq:short exact sequence mu}
 \xymatrix{ 0\ar[r] & \mu_{p^n}\ar[r] & K(p)^\times\ar[r]^{p^n} & K(p)^\times\ar[r] &0.}
\end{equation}
Since $\mu_{p^n}$ is isomorphic to $\Z_p(1)/p^n\Z_p(1)$ as $G_K(p)$-modules,
\eqref{2eq:short exact sequence mu} induces the commutative diagram in cohomology
\[ \xymatrix@C=.7truecm{
K^\times/(K^\times)^{p^n} \ar[r]^-{\sim}\ar@{->>}[d] & H^1\left(G_K(p),\mu_{p^n}\right)\ar@{=}[r]\ar[d]^{p^{n-1}} & H^1\left(G_K(p),\Z_p(1)/p^n\Z_p(1)\right)\ar[d]\\
K^\times/(K^\times)^{p^m} \ar[r]^-{\sim}             & H^1\left(G_K(p),\mu_{p^m}\right)\ar@{=}[r]                 & H^1\left(G_K(p),\Z_p(1)/p^n\Z_p(1)\right)
}\]
for every $n>m\geq1$, where the left-hand side horizontal arrows are isomorphisms by \eqref{1eq:kummer_dualityp}.
In particular, one has the epimorphism
\begin{equation}\label{2eq:epimorphismZp1/pn}
 \xymatrix{H^1\left(G_K(p),\Z_p(1)/p^n\Z_p(1)\right)\ar@{->>}[r]^{\pi_{nm}} & H^1\left(G_K(p),\Z_p(1)/p^m\Z_p(1)\right)}
\end{equation}
for every $n>m\geq1$.
In particular, $\{H^1\left(G_K(p),\Z_p(1)/p^n\Z_p(1)\right),\pi_{nm}\}$ is a projective system,
and by \cite[Corollary~2.7.6]{nsw:cohm} one has the isomorphism $H^1(G,\Z_p(1))\simeq\varprojlim_n H^1(G_K(p),\Z_p(1)/p^n\Z_p(1))$.

\begin{defi}
 Let $(G,\theta)$ be a pro-$p$ group with an orientation.
Then $\theta$ is called a {\bf cyclotomic} orientation (and $G$ is said to be {\bf cyclo-oriented})
if the following hold:
\begin{itemize}
 \item[(i)] $G$ is a Bloch-Kato pro-$p$ group;
 \item[(ii)] $H^2(C,\Z_p(1))$ is a torsion-free abelian group for all closed subgroups $C$ of $G$.
\end{itemize}
\end{defi}

\begin{rem}\label{2rem:cyclo-oriented free gps}
 If $F$ is a free pro-$p$ group, then $\ccd(F)=1$ by Proposition~\ref{2prop:cohomological dimension prop}.
Then for any orientation $\theta\colon F\to\Z_p^\times$ one has $H^2(F,\Z_p(1))=0$,
so that $(F,\theta)$ is cyclo-oriented for any $\theta$.
\end{rem}

One has the following elementary fact.

\begin{lem}\label{2lem:H2 torfree iff H1 projects}
 Let $(G,\theta)$ be a pro-$p$ group with orientation.
Then $H^2(G,\Z_p(1))$ is a torsion-free abelian group if, and only if, the map
\begin{equation}\label{2eq:p projection}
 \xymatrix{H^1(G,\Z_p(1))\ar[r] & H^1(G,\F_p) }
\end{equation}
induced by the short exact sequence of $\ZpG$-modules
\begin{equation}\label{2eq:ses Zp1Zp1Fp}
 \xymatrix{ 0\ar[r] & \Z_p(1)\ar[r]^-{p} & \Z_p(1)\ar[r] & \F_p\ar[r] & 0 }
\end{equation}
is an epimorphism.
\end{lem}

\begin{proof}
 The short exact sequence \eqref{2eq:ses Zp1Zp1Fp} induces the exact sequence in cohomology
\[ \xymatrix@C=.6truecm{ H^1(G,\Z_p(1))\ar[r] & H^1(G,\F_p)\ar[r] & H^2(G,\Z_p(1))\ar[r]^-{p} 
& H^2(G,\Z_p(1)), }\]
and $H^2(G,\Z_p(1))$ is torsion-free if, and only if, the right-hand side arrow is injective, namely,
if, and only if, the middle arrow is trivial and the left-hand side arrow is surjective.
\end{proof}

Consequently, by \eqref{2eq:epimorphismZp1/pn} one has that $H^1(G_K(p),\Z_p(1))$ projects onto $H^1(G_K(p),\F_p)$,
so that one has the following.

\begin{thm}
Let $K$ be a field containing a primitive $p$-th root of unity, and let $G_K(p)$ be its maximal pro-$p$ Galois group.
Then the arithmetical orientation $\theta\colon G_K(p)\to \Z_p^\times$ is cyclotomic. 
\end{thm}

The following result is a refinement of Proposition~\ref{2prop:tor p} for pro-2 groups with a cyclotomic orientation.

\begin{prop}\label{2prop:tor2}
Let $(G,\theta)$ be a pro-$2$ group with cyclotomic orientation.
\begin{itemize}
\item[(i)] If $G$ is non-trivial and torsion, then $G\simeq C_2$ and $\image(\theta)=\{\pm 1\}$.
\item[(ii)] If $\image(\theta)$ is torsion free, then $G$ is torsion free.
\end{itemize}
\end{prop}

\begin{proof} 
(i) By Proposition~\ref{2prop:tor p}, $G$ is an elementary abelian $2$-group.
Let $K=\kernel(\theta)$. Then $(K,\theta\vert_K)$ is an elementary abelian pro-$2$ group
with trivial cyclotomic $2$-orientation, namely,
$\theta\vert_K\equiv 1$ and $\Z_2(m)\simeq\Z_2(0)=\Z_2$ as $K$-modules for all $m\in\Z$.
Since $K$ is torsion, one has $H^1(K,\Q_2)=\Hom(K,\Q_2)=0$, 
and the connecting map $H^1(K,\I_2)\to H^2(K,\Z_2)$ induced in cohomology by the exact sequence
\eqref{1eq:sesZpQpIp} for $p=2$ is injective.
Thus one has the embedding 
\[\xymatrix{ K^\vee=H^1(K,\I_2)\ar@{^(->}[r] & H^2(K,\Z_2)=H^2(K,\Z_2(1))},\]
which is torsion-free.
Consequently $K^\vee=0$ and $K$ is the trivial group, and $\theta\colon G\to\Z_2^\times$ is injective.
This shows (i), and (ii) is a direct consequence of (i).
\end{proof}

From Proposition~\ref{2prop:tor p}, Proposition~\ref{2prop:tor2} and Sylow's theorem
one deduces the following corollary, which can be considered as a local version of the Artin-Schreier theorem
for profinite groups with cyclotomic $p$-orientation.

\begin{cor}\label{2cor:torsion}
Let $p$ be a prime number, and let $(G,\theta)$ be a pro-$p$ group with cyclotomic orientation.
\begin{itemize}
\item[(i)] If $p$ is odd, then $G$ has no $p$-torsion.
\item[(ii)] If $p=2$, then every non-trivial $2$-torsion subgroup is isomorphic
to $C_2$, the cyclic group of order $2$. Moreover, if $\image(\theta)$ has no 
$2$-torsion, then $G$ has no $2$-torsion. 
\end{itemize}
\end{cor}

\begin{rem}\label{2rem:2tor}
Let $\theta\colon \Z_2\to\Z_2^\times$ be given by $\theta(1+\lambda)=-1$ and $\theta(\lambda)=1$
for all $\lambda\in 2\Z_2$. Then $\theta$ is a $2$-orientation of $G=\Z_2$ satisfying
$\image(\theta)=\{\pm1\}$. As $\ccd_2(\Z_2)=1$, $\theta$ is also cyclotomic by Remark~\ref{2rem:cyclo-oriented free gps}.
This shows that for a torsion free pro-$2$ group $G$ with cyclotomic $2$-orientation
$\theta\colon G\to \Z_2^\times$, $\image(\theta)$ is not necessarily torsion free.
\end{rem}

On the other hand, if $\image(\theta)\leq1+4\Z_2$ for a cyclo-oriented pro-$2$ group $(G,\theta)$,
one has the following.

\begin{prop}\label{2prop:bock}
Let $(G,\theta)$ be a pro-$2$ group with cyclotomic $2$-orientation satisfying $\image(\theta)\simeq\Z_2$.
Then $\chi\cup\chi=0$ for all $\chi\in H^1(G,\F_2)$.
\end{prop}

\begin{proof}
Set $\Z_4(1)=\Z_2(1)/4\Z_2(1)$.
By Remark~\ref{1rem:p-adic and orientations}, $\Z_4(1)$ is a trivial $\Z_2\dbl G\dbr$-module,
thus $\Z_4\simeq\Z/4\Z$ as abelian groups.
Since the group $H^2(G,\Z_2(1))$ is torsion-free, one has a commutative diagram with exact rows
\begin{equation}\label{eq:tor1}
\xymatrix{H^1\left(G,\Z_2(1)\right)\ar[r]^{4}\ar[d]_{2}&
H^1\left(G,\Z_2(1)\right)\ar[r]\ar@{=}[d] & H^1\left(G,\Z_4(1)\right)\ar[r]\ar[d]^{\pi} & 0\\
H^1\left(G,\Z_2(1)\right)\ar[r]^{2} & H^1\left(G,\Z_2(1)\right)\ar[r] & H^1\left(G,\F_2\right)\ar[r] & 0}
\end{equation}
with $\pi$ the canonical morphism.
Hence by the weak four lemma applied to the diagram 
\eqref{eq:tor1} extended by another column of $0$'s to the right, $\pi$ is surjective. 
In particular, the first Bockstein morphism
\[\beta^1\colon H^1(G,\F_2)\longrightarrow H^2(G,\F_2)\]
is the trivial map, and the claim follows by Lemma~\ref{1lem:Bockstein and cup}.
\end{proof}

Let $(G,\theta)$ be a cyclo-oriented pro-$p$ group.
The above result, together with Remark~\ref{1rem:p-adic and orientations}, shows that if $\image(\theta)$
is pro-$p$-cyclic (possibly trivial), then the cup product in $H^\bullet(G,\F_p)$ is skew-commutative,
and Corollary~\ref{1cor:cd d r for BK} holds also in the case $\image\simeq\Z_2$.
In particular, from \eqref{2eq:BK epimorphism exterior algebra} one gets the epimorphism of
quadratic $\F_p$-algebras
\begin{equation}\label{2eq:projection exterior algebra}
 \xymatrix{\dfrac{\mathcal{T}^\bullet\left(H^1(G,\F_p)\right)}{\langle\chi\otimes\chi\rangle}\ar@{->>}[rr]
&& H^\bullet(G,\F_p)}
\end{equation}
for all cyclo-oriented pro-$p$ groups $(G,\theta)$ with $\image(\theta)\simeq\Z_p$,
or $\image(\theta)=\{1\}$ for $p$ odd.

\subsection{Modules of cyclo-oriented groups}

Given an oriented pro-$p$ group $(G,\theta)$, the orientation induces a $G$-action also on the $\Z_p$-modules
$\Q_p$ and $\I_p$.
Thus, one may define the $\ZpG$-modules $\Q_p(m)$ and $\I_p(m)$, for $m\in\Z$, 
and \eqref{1eq:sesZpQpIp} induces the short exact sequence 
\begin{equation}\label{2eq:sesZp1Qp1Ip1}
 0\longrightarrow \Z_p(m)\longrightarrow \Q_p(m)\longrightarrow \I_p(m)\longrightarrow0.
\end{equation}

\begin{fact}\label{fact:pontryagin dual} 
For $m\in\Z$, one has $\I_p(m)^\vee\simeq\Z_p(-m)$.
\end{fact}

\begin{proof}
By definition, one has  $\I_p\simeq\Z_p^\vee$ as $\Z_p$-modules.
Thus, for $\lambda\in\I_p$, $\phi\in\I_p(m)^*$ and $g\in G$ one has  
\[ (g.\phi)(\lambda)=\phi\left(g^{-1}.\lambda\right)=
\phi\left(\theta(g^{-1})^m\cdot\lambda\right)=\theta(g)^{-m}\cdot\phi(\lambda), \]
and this yields the claim.
\end{proof}

\begin{lem}\label{lem:homtensor equivalence}
For $G$ a profinite group and $\theta\colon G\rightarrow\Z_p^\times$ an orientation,
the Pontryagin dual induces the natural isomorphism of functors
\[ \Z_p(m)^\times\hat\otimes_G\Hom(\argu,\I_p)\longrightarrow \Hom_G(\Z_p(m),\argu)^\vee, \quad m\in\Z.\]
\end{lem}

\begin{proof}
Let $N$ be a discrete left $\ZpG$-module, and let $\lambda\in\Z_p(m)^\times$ and $f\in\Hom(N,\I_p)$.
Then define $\phi_{\lambda\otimes f}\in \Hom_G(\Z_p(m),N)^\vee$ such that 
$\phi_{\lambda\otimes f}(h)=f(h(\lambda))$ for every $h\in \Hom_G(\Z_p(m),N)$.
Then, for every $g\in G$, one has
\begin{eqnarray*}
&& \phi_{\lambda.g\otimes f}(h)=f(h(\lambda.g))=f\left(h(g^{-1}.\lambda)\right)= \\ &&
=f\left(g^{-1}.h(\lambda)\right)=g.f(h(\lambda))=\phi_{\lambda\otimes g.f}(h),
\end{eqnarray*}
and this yields the claim.
\end{proof}

In particular, for $m=0$ Lemma~\ref{lem:homtensor equivalence} implies the isomorphism
\begin{equation}\label{eq:isomorphism ext tor dual}
\Tor_n^{\ZpG}\left(\Z_p,M^\vee\right) \simeq \Ext_{\ZpG}^n\left(\Z_p,M\right)^\vee,
\end{equation}
with $M$ a discrete module (see also \cite[Corollary~5.2.9]{nsw:cohm}).

\begin{prop}\label{2prop:threequivalence homologycohomology}
Let $G$ be a profinite group with a $p$-orientation $\theta$.
Then for every $m\in\Z$ the following are equivalent:
\begin{itemize}
 \item[(i)] $H ^{m+1}(G,\Z_p(m))$ is a torsion-free $\Z_p$-module.
 \item[(ii)] $H^m(G,\I_p(m))$ is a $p$-divisible $\Z_p$-module. 
 \item[(iii)] $H_m(G,\Z_p(-m))$ is a torsion-free $\Z_p$-module.
\end{itemize}\end{prop}

\begin{proof}
The exact sequence \eqref{2eq:sesZp1Qp1Ip1} induces the exact sequence in cohomology
\begin{equation}\label{2eq:lngextseqI_p}
\xymatrix@C=0.7truecm{
H^n\left(G,\Z_p(m)\right)\ar[r] & H^n\left(G,\Q_p(m)\right)\ar[r] &
H^n\left(G,\I_p(m)\right)\ar`r[d]`[l]^-{\delta^m} `[dll] `[dl] [dl]   \\
 & H^{m+1}\left(G,\Z_p(m)\right)\ar[r] & H^{m+1}\left(G,\Q_p(m)\right) }
\end{equation}
for $m\geq0$.
Then by Proposition~\ref{1prop:tate proposition}, one has that $\kernel(\delta)$ is
the maximal $p$-divisible subgroup of $H^m(G,\I_p(m))$,
and $\image(\delta)$ is the torsion subgroup of $H^{m+1}(G,\Z_p(m))$.
Thus, $H^{m+1}(G,\Z_p(m))$ is a torsion-free $\Z_p\dbl G\dbr$-module if, and only if, 
$\delta$ is the zero map, i.e., $\kernel(\delta)=H^m(G,\I_p(m))$.
This establishes the equivalence between (i) and (ii).

Now set $M=\I_p(m)$.
Then (\ref{eq:isomorphism ext tor dual}) and Fact~\ref{fact:pontryagin dual} imply that
\[H_m(G,\Z_p(-m))\cong H^m\left(G,\I_p(m)\right)^\vee.\]
By duality, it follows that $H_m(G,\Z_p(-m))$ is a profinite torsion-free $\ZpG$-module
if, and only if, $H^m(G,\I_p(m))$ is a discrete $p$-divisible $\ZpG$-module.
This establishes the equivalence between (ii) and (iii).
\end{proof}

The above result has the following consequence.

\begin{cor}\label{2cor:kernelab torfree}
Let $(G,\theta)$ a cyclo-oriented pro-$p$ group.
The abelianization $\kernel(\theta)/[\kernel(\theta),\kernel(\theta)]$ of the kernel
of the orientation $\theta$ is a torsion-free group.
\end{cor}

\begin{proof}
Since the restriction $\theta|_{\kernel(\theta)}$ is trivial, $\Z_p(1)$ is a trivial
$\Z_p\dbl\kernel(\theta)\dbr$-module.
Thus, one has
\[H_1(\kernel(\theta),\Z_p(1))=H_1(\kernel(\theta),\Z_p)\simeq\frac{\kernel(\theta)}{[\kernel(\theta),\kernel(\theta)]}.\]
Therefore, $\kernel(\theta)/[\kernel(\theta),\kernel(\theta)]$ is torsion-free
by Proposition~\ref{2prop:threequivalence homologycohomology}.
\end{proof}

On the other hand, the following result describes the torsion of $H^1(G,\Z_p(1))$ as $\Z_p$-module
(\cite[Lemma~3.2]{qw:diederichsen}).

\begin{lem}\label{2lem:torsion H1}
Let $(G,\theta)$ be a cyclo-oriented pro-$p$ group such that $\image(\theta)=1+p^k\Z_p\simeq\Z_p$
for some $k\geq1$.
The torsion $\Z_p$-submodule of $H^1(G,\Z_p(1))$ is isomorphic to $\Z/p^k\Z$. 
\end{lem}

\begin{proof}
The exact sequence of $G$-modules \eqref{2eq:sesZp1Qp1Ip1} with $m=1$ induces the exact cohomology sequence
\[ \xymatrix{\cdots\ar[r] & H^0\left(G,\Q_p(1)\right)\ar[r] &  
H^0\left(G,\I_p(1)\right)\ar`r[d]`[l]^-{\delta^0} `[dll] `[dl] [dl] & \\
  & H^1\left(G,\Z_p(1)\right) \ar[r] & H^1(G,\Q_p(1))\ar[r] &\cdots } \]
As $\theta$ is non-trivial, one has $H^0(G,\Q_p(1))=\Q_p(1)^G=0$.
Moreover, $H^1(G,\Q_p(1))$ is a torsion free $\Z_p$-module, and thus by Proposition~\ref{1prop:tate proposition}
$\delta$ induces an isomorphism between $\I_p(1)^G$ and the torsion submodule of $H^1\left(G,\Z_p(1)\right)$.
Since \[\I_p(1)^G=\frac{1}{p^k}\Z_p(1)/\Z_p(1)\simeq\Z_p/p^k\Z_p\] as $\Z_p$-modules, this yields the claim.
\end{proof}

\section{The Elementary Type conjecture}

\subsection{Demushkin groups}
Let $\F$ be a field and let $V,W$ be two $\F$-vector spaces.
Recall that a {\it non-degenerate pairing} of $V$ and $W$ is a $\F$-bilinear form $V\times W\rightarrow \F$
which induces injections $V\hookrightarrow W^*$ and $W\hookrightarrow V^*$, where $V^*$ and $W^*$
are the $\F$-duals.
If $V$ and $W$ have finite dimensions, then such monomorphisms are isomorphisms.

\begin{defi}
 A finitely generated pro-$p$ group $G$ is called a {\bf Demushkin group} if the cohomology ring $H^\bullet(G,\F_p)$
satisfies the following properties:
\begin{itemize}
 \item[(i)] $H^1(G,\F_p)=1$ is isomorphic to $\F_p$ as $\F_p$-vector spaces.
 \item[(ii)] the cup product $\cup\colon H^1(G,\F_p)\times H^1(G,\F_p)\rightarrow H^2(G,\F_p)$ is a non-degenerate pairing.
\end{itemize}
\end{defi}

If $G$ is an infinite Demushkin group, then $G$ is a Poincar\'e group of cohomological dimension $\ccd(G)=2$.
On the other hand, the only torsion Demushkin group is the cyclic group of order 2
(cf. \cite[(3.9.9) and (3.9.10)]{nsw:cohm}).
In particular, one may find a basis $\{\chi_1,\ldots,\chi_d\}$ of $H^1(G,\F_p)$, with $d=d(G)$,
such that 
\begin{eqnarray*}
 \chi_1\cup\chi_2=\chi_3\cup\chi_4=\ldots=\chi_{d-1}\cup\chi_d,\quad\text{for }2\mid d,\\
 \chi_1\cup\chi_1=\chi_2\cup\chi_3=\ldots=\chi_{d-1}\cup\chi_d,\quad\text{for }2\nmid d,
\end{eqnarray*}
and $H^2(G,\F_p)$ is generated by such element.

Demuskin groups are arithmetically interesting, as they appear as Galois group of $p$-adic fields.
In particular, if $K$ is a {\it $p$-adic local field} (i.e., a finite extension of the field $\Q_p$)
containing the roots of unity of order $p$, then the maximal pro-$p$ Galois group $G_K(p)$ is a Demushkin group
with $d(G_K(p))=|K:\Q_p|+2$ (cf. \cite[Theorem~7.5.11]{nsw:cohm}).
From this, it follows that arithmetical Demushkin groups are also Bloch-Kato pro-$p$ groups.

By \eqref{1eq:number relations}, a Demushkin group $G$ is a {\it one-relator} pro-$p$ group, i.e., $r(G)=1$.
Thus, one has that either
\[G/[G,G]\simeq\Z_p^d \quad \text{or} \quad G/[G,G]\simeq \Z/p^k\Z\times\Z_p^{d-1},\]
with $k\geq1$.
Set $q_G=0$ in the former case and $q_G=p^k$ in the latter.
The structure of Demushkin groups has been studied in the '60s by S.P.~Demu\v{s}kin, J.-P.~Serre
and finally by J.~Labute, who completed the case $p=2$ and equipped Demushkin groups with a suitable orientation.
Altogether, one has the following (cf. \cite{labute:classification}).

\begin{thm}\label{2thm:Demushkin}
 Let $G$ be a finitely generated one-relator pro-$p$ group, and set $q_G$ as above.
Then $G$ is a Demushkin group if, and only if, it is isomorphic to the pro-$p$ group with a presentation
$\langle x_1,\ldots,x_d|\:r=1 \rangle$ with $d=d(G)$, such that 
\begin{itemize}
 \item[(i)] if $p$ is odd or if $p=2$ and $q_G\neq2$, one has \[r=x_1^{-q_G}[x_1,x_2][x_3,x_4]\cdots[x_{d-1},x_d];\]
 \item[(ii)] if $q_G=2$ and $d$ is odd, one has \[r=x_1^{-2}x_2^{-2^f}[x_2,x_3][x_4,x_5]\cdots[x_{d-1},x_d],\] 
for some $f=2,3,\ldots,\infty$;
 \item[(iii)]  if $q_G=2$ and $d$ is even, one has \[r=x_1^{-2-\alpha}[x_1,x_2]x_3^{-2^f}[x_3,x_4]\cdots[x_{d-1},x_d],\]
for some $f=2,3,\ldots,\infty$ and $\alpha\in 4\Z_2$.
\end{itemize}
Moreover, let $\theta\colon G\rightarrow\Z_p^\times$ be the orientation defined by
\begin{itemize}
 \item[(i)] $\theta(x_2)=(1-q_G)$ and $\theta(x_i)=1$ for $i\neq2$, if $q_G\neq2$;
 \item[(ii)] $\theta(x_1)=-1$, $\theta(x_3)=(1-2^f)$ and $\theta(x_i)=1$ for $i\neq1,3$, if $q_G=2$ and $d$ is odd;
 \item[(iii)] $\theta(x_2)=-(1+\alpha)^{-1}$, $\theta(x_4)=(1-2^f)$ and $\theta(x_i)=1$
for $i\neq2,4$, if $q_G=2$ and $d$ is even.
\end{itemize}
Then $\theta$ is a cyclotomic orientation for $G$.
\end{thm}

\begin{rem}\label{2rem:notation of labute}
 Note that in \cite{labute:classification}, the author uses a different notation:
\begin{itemize}
 \item[(a)] he uses the ``right-action notation'' for the commutator, i.e., $[x,y]=x^{-1}x^y=x^{-1}y^{-1}xy$,
that is why we get the signs minus in the relations;
 \item[(b)] he uses $\chi$ instead of $\theta$ for the orientation (and he does not use such term, as he calls 
$\chi$ simply a ``continuous homomorphism'');
 \item[(c)] he uses $\mathbf{U}_p$ instead of $\Z_p^\times$.
\end{itemize}
\end{rem}

It is interesting to compute also the restricted Lie algebra $L_\bullet(G)$ induced by the Zassenhaus filtration,
when $G$ is a Demushkin group.
Let $\mathcal{X}=\{x_1,\ldots,x_d\}$, $d=d(G)$, be a minimal generating system for $G$
such that $G$ has a presentation with a relation $r$ as described in Theorem~\ref{2thm:Demushkin}.
Then $r$ lies in $D_2(G)\smallsetminus D_3(G)$, and one has
\[r\equiv\left\{\begin{array}{cc} [x_1,x_2]+\ldots+[x_{d-1},x_d] \bmod D_3(G) & \text{if }q_G\neq2 \\ 
 x^2+[x_1,x_2]+\ldots+[x_{d-1},x_d] \bmod D_3(G) & \text{if }q_G=2,\:2\mid d \\
 x^2+[x_2,x_3]+\ldots+[x_{d-1},x_d] \bmod D_3(G) & \text{if }q_G=2,\:2\nmid d  
 \end{array}\right. \]
(where we used the additive notation for the quotient $D_2(G)/D_3(G)$).
If $q_G\neq2$, then \cite[Theorem~2.4.6]{jochen:thesis} and Theorem~\ref{1thm:envelope Zassenhaus filtration} imply
that the restricted Lie algebra $L_\bullet(G)$ is the quotient of the free restricted $\F_p$-Lie algebra
$L_p(\mathcal{X})$ over the ideal $\mathfrak{r}$ generated by the image of $r$ in $L_2(F)$,
where $F$ is the free pro-$p$ group generated by $\mathcal{X}$.
Moreover, by Proposition~\ref{1prop:presentation restricted envelope} one has
\begin{equation}\label{2eq:restricted envelope demushkin}
 \grad_\bullet(G)\simeq \mathcal{U}_p(L_\bullet(G))\simeq
\frac{\F_p\langle\mathcal{X}\rangle}{\langle[X_1,X_2]+\ldots+[X_{d-1},X_d]\rangle},
\end{equation}
where $\mathcal{X}=\{X_1\ldots,X_d\}$, with the grading induced by the degrees of the monomials.

\subsection{Free products and the Elementary Type Conjecture}

Let $G$ be a profinite group, and let $\{G_i,i\in \mathcal{I}\}$ be a collection of profinite groups
with $\mathcal{I}$ a set of indices.
For every $i\in\mathcal{I}$, let $\iota_i\colon G_i\to G$ be a continuous homomorphism.
If the set $\mathcal{I}$ is not finite, assume further that for every open neighborhood $U$ of 1 in $G$,
$U$ contains all but a finite number of the images $\image(\iota_i)$ (if this condition is satisfied, the family
$\{\iota_i,i\in\mathcal{I}\}$ is said to be {\it convergent}).

The profinite group $G$, together with the maps $\iota_i$, is  said to be the {\bf free profinite product}
of the groups $G_i$ if the following universal property is satisfied:
whenever $\{\psi_i\colon G_i\to H,i\in\mathcal{I}\}$ is a convergent family of continuous homomorphisms
into a profinite group $H$, there exists a unique continuous homomorphism $\psi\colon G\to H$ such that
\begin{equation}\label{2eq:free product}
 \xymatrix{ G_i\ar[r]^-{\iota_i}\ar[d]_{\psi_i} & G \ar@{-->}[dl]^{\psi}\\ H}
\end{equation}
commutes for every $i\in\mathcal{I}$.
Such free product is denoted by
\[G=\coprod_{i\in\mathcal{I}}G_i.\]

If the groups $G_i$'s are all pro-$p$, then $G$ is the {\bf free pro-$p$ product}
if \eqref{2eq:free product} commutes for every $i\in\mathcal{I}$ and for any pro-$p$ group $H$
(see, e.g., \cite[\S~9.1]{ribeszalesskii:profinite} and \cite[\S~2.7]{ribes:lux} for details).
In particular, if $\mathcal{I}$ is finite, let $G^{abs}$ denote the free {\it abstract} product of the $G_i$'s.
Then $G$ is the projective limit $\varprojlim_{\mathcal{N}}G^{abs}/N$,
where $\mathcal{N}$ is the family of normal subgroups of $G^{abs}$
\begin{equation}\label{2eq:free prop prod N}
 \mathcal{N}=\left\{N\;\left|\;|G^{abs}:N|=p^m,N\cap G_i\leq G_i\text{ open for all }i \right.\right\}.
\end{equation}
One may denote the free pro-$p$ product also by 
\[G=G_1\ast_{\hat p}\cdots\ast_{\hat p}G_n.\]

It is well known that free profinite products of absolute Galois groups of fields are absolute Galois groups.
Moreover, the free pro-$p$ product of absolute Galois groups which are pro-$p$ is a pro-$p$ group which is
realizable as absolute Galois group for some fields (see, e.g., \cite{ershov:products}
and \cite{haranjardenkoenigsmann} for details).

Yet, such a general statement is still missing for maximal pro-$p$ Galois groups of fields.
So far, only partial results are known: for example
\begin{itemize}
 \item[(a)] the free pro-$p$ product of two arithmentic Demushkin groups is realizable as
maximal pro-$p$ Galois group of fields (this is a consequence of \cite[Main Theorem]{efrat:etconj});
 \item[(b)] the free pro-$p$ product of $\theta$-abelian maximal pro-$p$ groups (and of $\Z/2\Z$ in the case $p=2$)
is realizable as maximal pro-$p$ Galois group of fields (this can be deduced from \cite{efrat:freeprod}).
\end{itemize}
The possibility to use free pro-$p$ products to construct new maximal pro-$p$ Galois groups provided the inspiration
to formulate the {\bf Elementary Type Conjecture} (=ETC) on maximal pro-$p$ Galois groups,
which states how a finitely generated maximal pro-$p$ Galois group of a field should ``look like''.

This conjecture originates form the theory of Witt rings: the {\it ETC on Witt rings of fields}
was concieved in the early '80s to study maximal 2-extension of fields via Witt rings, 
and it states that the Witt rings of fields of characteristic not 2 can be constructed from 
elementary ``building blocks'' by means of two standard operations (cf. \cite[\S~29.3]{efrat:libro}).
Later, B.~Jacob and R.~Ware translated the conjecture in the context of maximal pro-2 Galois groups of fields,
and they obtained a classification of such galois groups with small number of generators
(cf. \cite{jacobware:pro2} and \cite{jacobware:pro2bis}). 
In this case, finitely generated pro-2 Galois groups can be built from some ``basic'' pro-2 groups, i.e.,
$\Z_2$, Demushkin groups and $\Z/2\Z$.

In the '90s, I.~Efrat stated a stronger pro-$p$ arithmetical version of the ETC
which mimics Jacob and Ware's ``pro-2 translation'':
the maximal pro-$p$ Galois groups of fields of characteristic not $p$ can be constructed
from elementary ``building blocks'', such as $\Z_p$, Demushkin groups and $\Z/2\Z$, 
by iterating two rather simple operations, i.e., free pro-$p$ products and certain fibre products
(which we will define in Section~4.3), and it was stated explicitly in \cite[Question~4.8]{efrat:etconj}
and reformulated in \cite[Definition~1]{engler:etc}.

The ETC has been proven in some particular cases, such as maximal pro-$p$ Galois groups
of algebraic extensions of $\Q$, or fields satisfying certain condition on valuations.
A general answer for this conjecture seems to be still quite far from being obtained (so far).


\chapter[$\theta$-abelian groups]{The upper bound: $\theta$-abelian groups}

\section[The cup product]{$\theta$-abelian pro-$p$ groups and the cup product}

In this chapter we shall study the ``mountain boundary case'' of Bloch-Kato pro-$p$ groups,
namely, the case
\begin{equation}\label{3eq:wedge_isomorphic_cohomology}
\xymatrix{\bigwedge_\bullet H^1(G,\F_p) \ar[r]^-{\sim} & H^\bullet(G,\F_p).} 
\end{equation}

\begin{rem}
Most of the content of this chapter is contained (sometimes in a different shape) in \cite{claudio:BK}
and \cite{cmq:fast}, plus some generalizations for the two cases $p=2$ and $d(G)=\infty$.
\end{rem}

In particular, one has the following fact.

\begin{lem}\label{3lem:cup injective}
Let $(G,\theta)$ be a pro-$p$ group with orientation, and assume that $\theta$ is cyclotomic with
$\image(\theta)\simeq\Z_2$ if $p=2$.
Then the epimorphism \eqref{2eq:projection exterior algebra} is an isomorphism if, and only if, the map
\begin{equation}\label{3eq:wedge cup product}
\xymatrix{ H^1(G,\F_p)\wedge H^1(G,\F_p)\ar[rr]^-{\wedge_2(\cup)} && H^2(G,\F_p) },
\end{equation}
induced by the cup product, is injective.
In particular, if $G$ is finitely generated, then \eqref{3eq:wedge cup product} is an isomorphism
 if, and only if, 
\begin{equation}\label{3eq:cd=d}
 \ccd(G)=d(G).
\end{equation}
\end{lem}

\begin{proof}
 The proof is an easy argument from elementary linear algebra, together with Proposition~\ref{2prop:bock}
for the case $p=2$. See also \cite[Prop.~4.3]{claudio:BK}.
\end{proof}

Since every closed subgroup of a maximal pro-$p$ Galois group is again a maximal pro-$p$ Galois,
it is worth asking what does the injectivity of the morphism $\wedge_2(\cup)$ for every closed subgroup
imply in the case of Bloch-Kato pro-$p$ groups.
In particular, one has the following Tits alternative-type result.

\begin{thm}\label{3thm:titsalternative}
 Let $(G,\theta)$ be a Bloch-Kato pro-$p$ group, and assume that $\theta$ is cyclotomic with
$\image(\theta)\simeq\Z_2$ if $p=2$.
Then $G$ does not contain non-abelian closed free pro-$p$ group if, and only if, the map
\begin{equation}\label{3eq:wedge cup product closed subgroup}
 \xymatrix{ \wedge_2(\cup)\colon H^1(C,\F_p)\wedge H^1(C,\F_p)\ar[r] & H^2(C,\F_p) }
\end{equation}
is injective for every closed subgroup $C$ of $G$.
\end{thm}

\begin{proof}
Assume that $G$ contains a closed subgroup $C$ such that \eqref{3eq:wedge cup product closed subgroup}
is not injective.
Let $\{\chi_i,i\in\mathcal{I}\}$ be an $\F_p$-basis of the $H^1(C,\F_p)$.
Thus there exists a non-trivial element
\[\eta=\sum_{i,j\in\mathcal{J}}\,a_{ij}\chi_i\wedge\chi_j\in\kernel\left(\wedge_2(\cup)\right),\]
with $\mathcal{J}\subseteq\mathcal{I}$ finite.
As $\eta\not=0$, there exist $m,n\in\mathcal{J}$, $m\neq n$, such that $a_{mn}\not=0$.
Let $\{x_i,i\in\mathcal{I}\}\subset C$ be a minimal generating system of $C$ such that
every $\chi_i$ is dual to $x_i$ for all $i$, and let $S$ be the subgroup of $G$ generated by $x_m$ and $x_n$.
Then by duality \eqref{1eq:pontryagin_duality} the map 
\[\xymatrix{\res^1_{C,S}\colon H^1(C,\F_p)\ar[r] & H^1(S,\F_p)}\]
is surjective, and, by construction,
\[\kernel\left(\res^1_{C,S}\right)=\spa_{\F_p}\left\{\,\chi_i\mid  i\in\mathcal{I},\; i\not=n,m\,\right\}.\] 
From the surjectivity of $\res^1_{C,S}$ and the commutativity of the diagram
\begin{equation}\label{3eq:diagram tits alternative}
 \xymatrix{ H^1(C,\F_p)\wedge H^1(C,\F_p)\ar@<6ex>@{->>}[d]_{\res^1_{C,S}}\ar@<-6ex>@{->>}[d]_{\res^1_{C,S}}\ar[rr]^-{\wedge_2(\cup)} 
&& H^2(C,\F_p)\ar@{->}[d]^{\res^2_{C,S}}\\
H^1(S,\F_p)\wedge H^1(S,\F_p)\ar[rr]^-{\wedge_2(\cup)} && H^{2}(S,\F_p) }
\end{equation}
one concludes that the lower horizontal arrow of \eqref{3eq:diagram tits alternative} is the $0$-map.
Since $S$ is Bloch-Kato, one has that $H^2(S,\F_p)=0$, i.e., $S$ is a 2-generated free pro-$p$ group
by Proposition~\ref{2prop:cohomological dimension prop}, a contradiction.

Conversely, if $G$ contains a non-abelian free pro-$p$ group $C$, then $H^2(C,\F_p)=0$,
and \eqref{3eq:wedge cup product closed subgroup} is not injective.
\end{proof}

The following theorem is the key to study the group structure of Bloch-Kato pro-$p$ groups
whose $\F_p$-cohomology ring is an exterior algebra.
It is due to P.~Symonds and Th.~Weigel (cf. \cite[Theorem~5.1.6]{symonds_thomas:prop}).

\begin{thm}\label{3thm:peter thomas}
Let $G$ be a finitely generated pro-$p$ group.
Then the map \eqref{3eq:wedge cup product} is injective if, and only if, $G$ is powerful.
\end{thm}

\section[Powerful pro-$p$ groups]{$\theta$-abelian pro-$p$ groups and powerful pro-$p$ groups}

\subsection{Powerful pro-$p$ groups and Lie algebras}

Recall from Section~1.2 that $G^m$ is generated by the $m$-th powers of $G$, for $m\geq1$.
A pro-$p$ group is said to be {\it powerful} if 
\[[G,G]\leq \left\{ \begin{array}{cc} G^p & \text{for $p$ odd} \\  G^4 & \text{for }p=2.\end{array}\right. \]
Moreover, $G$ is called {\it uniformly powerful}, or simply {\it uniform}, if $G$ is finitely generated,
powerful, and 
\[\left|\lambda_i(G):\lambda_{i+1}(G)\right|=|G:\Phi(G)|\quad \text{for all }i\geq1\]
Thus, one has the following characterization for uniform pro-$p$ groups (\cite[Ch.~3-4]{ddsms:analytic}).

\begin{thm}
 A finitely generated powerful pro-$p$ group $G$ is uniform if, and only if, $G$ is torsion-free.
\end{thm}

Finally, a pro-$p$ group $G$ is called {\it locally powerful}
if every finitely generated closed subgroup $K$ of $G$ is powerful.
Moreover, for uniform pro-$p$ groups, one has the following property (cf. \cite[Prop.~4.32]{ddsms:analytic}).

\begin{prop}\label{3prop:proppresuniform}
 Let $G$ be a $d$-generated uniform pro-$p$ group, and let $\{x_1,\ldots,x_d\}$ be a generating set for $G$.
Then $G$ has a presentation $G=\langle x_1,\ldots,x_d|R\rangle$ with relations
\begin{equation}\label{3eq:proppresuniform1}
 R=\left\{[x_i,x_j]=x_1^{\lambda_1(i,j)}\cdots x_d^{\lambda_d(i,j)}, 1\leq i<j\leq d\right\},
\end{equation}
and for all $i,j$ one has $\lambda_n(i,j)\in p.\Z_p$ if $p$ is odd, and $\lambda_n(i,j)\in 4.\Z_2$ if $p=2$.
\end{prop}

If $G$ is a uniform pro-$p$ group, then it is possible to associate a $\Z_p$-Lie algebra
$\log(G)$ to it (see \cite[\S~4.5]{ddsms:analytic} and \cite[\S~3.1]{claudio:BK}),
i.e., $\log(G)$ is the $\Z_p$-free module generated by the generators of $G$,
equipped with the sum
\begin{equation}\label{deflie1}
 x+y=\lim_{n\rightarrow\infty}x+_ny,\quad x+_ny=\left(x^{p^n}y^{p^n}\right)^{p^{-n}},
\end{equation}
and the Lie brackets
\begin{equation}\label{deflie2}
(x,y)=\lim_{n\rightarrow\infty}(x,y)_n,\quad(x,y)_n=\left[x^{p^n},y^{p^n}\right]^{p^{-2n}}.                     
\end{equation}

\begin{rem}
 Note that the $\Z_p$-Lie algebra $\log(G)$ of a uniform pro-$p$ group is not related to the
restricted $\F_p$-Lie algebra $L_p(G)$ as defined in Section~1.3.
\end{rem}

The following result lists the properties of the $\Z_p$-Lie algebras associated to uniform pro-$p$ groups
we will need in the continuation (cf. \cite[\S~3.1 and \S~3.3]{claudio:BK}).

\begin{prop}\label{3prop:remarklocpwfetLie}
\begin{itemize}
 \item[(i)]  If $G$ is uniform with minimal generating system $\{x_1,\ldots,x_n\}$,
then it is possible to write every element $g\in G$ as $g=x_1^{\lambda_1}\cdots x_n^{\lambda_n}$,
with $\lambda_i\in\Z_p$, in a unique way.
Thus the map
\[G\longrightarrow\log(G),\quad x_1^{\lambda_1}\cdots x_n^{\lambda_n}\longmapsto\lambda_1x_1+\ldots+\lambda_nx_n\]
is a homeomorphism in the $\Z_p$-topology (cf. \cite[Theorem 4.9]{ddsms:analytic}).
\item[(ii)] If $G$ is locally powerful and torsion-free, let $C$ be a finitely generated subgroup of $G$.
Then one can construct the Lie algebra $\log(C)$, which is in fact a subalgebra of $\log(G)$.
In particular, the $\Z_p$-submodule $\spa_{\Z_p}\{x\in \mathcal{S}\}$ of $\log(G)$ is closed under Lie brackets
for every finite subset $\mathcal{S}\subseteq G$.
\item[(iii)] An uniform pro-$p$ group $G$ with orientation $\theta$ is $\theta$-abelian if, and only if, 
the associated $\Z_p$-Lie algebra $\log(G)$ has a basis $\{v_1,\ldots,v_d\}$, with $d=d(G)$, such that
\[(v_1,v_i)=\lambda v_i \quad \text{and} \quad (v_i,v_j)=0\]
for all $1<i,j\leq d$, where $\image(\theta)=1+\lambda\Z_p$ and $\lambda\in p\Z_p$, or $\lambda\in 4\Z_2$ if $p=2$
(cf. \cite[Prop.~3.6]{claudio:BK})
\end{itemize} 
\end{prop}

\subsection{Locally powerful pro-$p$ groups}

In spite of a rather convoluted definition, torsion-free locally powerful pro-$p$ groups have a very simple 
(and rigid) structure: in fact such a group is naturally equipped with an orientation $\theta$ which makes it
$\theta$-abelian.

\begin{thm}\label{3thm:locally poweful groups}
Let $G$ be a torsion-free pro-$p$ group.
Then $G$ is locally powerful if, and only if, there exists an orientation $\theta\colon G\rightarrow \Z_p^\times$
such that $(G,\theta)$ is $\theta$-abelian.
\end{thm}

\begin{proof}
If $G$ is $\theta$-abelian,
then, by Proposition~\ref{2prop:presentation theta-abelian}, $G$ is locally powerful and torsion-free.

Conversely, let $G$ be a torsion-free locally powerful pro-$p$ group,
and assume first that $G$ is finitely generated, with $d(G)=d\geq2$,
Thus by Proposition \ref{3prop:proppresuniform}, $G$ has a presentation $G=\langle x_1,\ldots,x_d|R\rangle$
with relations as in (\ref{3eq:proppresuniform1}).
Let $H_{ij}\clsgp G$ be the closed subgroup generated by the elements $x_i,x_j$.
Since $H_{ij}$ is uniform as well, we have that
\[H_{ij}=\left\langle x_i,x_j\left|[x_i,x_j]=x_i^{\lambda_i(i,j)} x_j^{\lambda_j(i,j)},\lambda_i,\lambda_j\in p.\Z_p\right.\right\rangle,\]
so that $R=\{[x_i,x_j]=x_i^{\lambda_i(i,j)} x_j^{\lambda_j(i,j)}, 1\leq i<j\leq d\}$
is the set of relations.

Since an abelian pro-$p$ group is ${\mathbf 1}$-abelian, where ${\mathbf 1}$ is the trivial orientation,
we may assume that $G$ is not abelian, i.e., we may assume without loss of generality that $x_1$ and $x_2$ do not commute.

\medskip
\textit{Step 1:} 
First suppose that $d=2$.
It is well known that if $G$ is nonabelian,
then $G$ has a presentation $\langle x,y|[x,y]y^{-p^k}\rangle$
for some uniquely determined positive integer $k$ (with $k\geq2$ if $p=2$, see \cite[Ch.~4, Ex.~13]{ddsms:analytic})
-- in fact $G$ is a 2-generated Demushkin group.
Thus, $G$ is $\theta$-abelian by Proposition~\ref{2prop:presentation theta-abelian},
with $\theta(x)=1+p^k$ and $\theta(y)=1$.

\medskip
\textit{Step 2:} Suppose $d=3$.
By the previously mentioned remark we may choose $x_1,x_2$ such that $[x_1,x_2]=x_2^\lambda$,
with $\lambda\in p.\Z_p$ (resp. $\lambda\in4.\Z_2$ if $p=2$).
Thus
\[G=\left\langle x_1,x_2,x_3\left| \left[x_1,x_2\right]=x_2^\lambda,\left[x_1,x_3\right]=x_1^{\lambda_1}x_3^{\lambda_2},
\left[x_2,x_3\right]=x_2^{\mu_1}x_3^{\mu_2}\right.\right\rangle,\]
with $\lambda_i,\mu_i\in p.\Z_p$ (resp. in $4.\Z_2$).
Let $H_{ij}$ be the subgroups as defined above, with $1\leq i<j\leq3$, and let $L=\log(G)$.
Clearly, $(x_i,x_j)_n\in H_{ij}$ for all $n$.
Hence $(x_i,x_j)\in\spa_{\Z_p}\{x_i,x_j\}$.
In particular, the Lie brackets in $L$ are such that
\[(x_1,x_2)=\alpha x_2,\quad (x_2,x_3)=\beta_2x_2+\beta_3x_3,\quad (x_1,x_3)=\gamma_1x_1+\gamma_3x_3,\]
with $\alpha, \beta_i, \gamma_i\in p.\Z_p$ (resp. in $4.\Z_2$).

By the Jacobi identity, one has
\begin{eqnarray*}
 0 &=& \left((x_1,x_2),x_3\right)+\left((x_2,x_3),x_1\right)+\left((x_3,x_1),x_2\right)\\
   &=& (\alpha x_2,x_3)+(\beta_2x_2+\beta_3x_3,x_1)-(\gamma_1x_1+\gamma_3x_3,x_2)\\
   &=& -\beta_3\gamma_1x_1+\left(\alpha\beta_2-\alpha\beta_2-\alpha\gamma_1+\beta_2\gamma_3\right)x_2
        +\left(\alpha\beta_3-\beta_3\gamma_3+\beta_3\gamma_3\right)x_3,
\end{eqnarray*}
hence $\beta_3\gamma_1x_1=0$, and thus $\beta_3=0$ or $\gamma_1=0$.
\begin{enumerate}
 \item[(1)] If $\beta_3=0$, then by definition $(x_2,x_3)\in\spa_{\Z_p}\{x_2\}$,
 i.e., the $\Z_p$-module generated by $x_2$ is an ideal of $L$.
Therefore we may choose without loss of generality $x_1$ and $x_3$ such that $(x_1,x_3)\in\spa_{\Z_p}\{x_3\}$,
and $(x_i,x_2)\in\spa_{\Z_p}\{x_2\}$ for $i=1,3$.
 \item[(2)] If $\gamma_1=0$, then by definition $(x_1,x_3)\in\spa_{\Z_p}\{x_3\}$,
i.e., the $\Z_p$-module generated by $x_2$ and $x_3$ is an ideal of $L$.
Therefore we may choose without loss of generality $x_2$ and $x_3$ such that $(x_2,x_3)\in\spa_{\Z_p}\{x_2\}$,
and $(x_1,x_i)\in\spa_{\Z_p}\{x_i\}$ for $i=2,3$.
\end{enumerate}
Altogether the Lie brackets in $L$ are
\[(x_1,x_2)=\alpha'x_2,\quad(x_2,x_3)=\beta'x_2,\quad(x_1,x_3)=\gamma'x_3,\]
with $\alpha',\beta',\gamma'\in p.\Z_p$ (resp. in $4.\Z_2$).
The matrix of $\text{ad}(\gamma x_3)$ with respect to the basis $\{x_1,x_2,x_3\}$ is given by
\[\text{ad}(\gamma'x_3)=\left(\begin{array}{ccc}
   0 & 0 & 0 \\ 0 & \beta'\gamma' & 0 \\ -\gamma'^2 & 0 & 0 \end{array}\right).\]
In particular, its trace is $\text{tr}(\text{ad}(\gamma'x_3))=\beta'\gamma'$.
Since $\text{ad}(\gamma'x_3)=(\text{ad}(x_1),\text{ad}(x_3))$,
one has $\text{tr}(\text{ad}(\gamma'x_3))=\beta'\gamma'=0$.
Therefore $\beta'=0$ or $\gamma'=0$.

\begin{enumerate}
 \item[(1)] If $\beta'=0$, let $v_1=x_1+x_2$ and $v_2=x_2+x_3$.
Then $(v_1,v_2)=\alpha'x_2+\gamma'x_3$.
By Proposition~\ref{3prop:remarklocpwfetLie}, one has that $(v_1,v_2)$ is a
$\Z_p$-linear combination of $v_1$ and $v_2$. 
Thus $(v_1,v_2)$ is necessarily a multiple of $v_2$, i.e., $\alpha'=\gamma'$.
 \item[(2)] If $\gamma'=0$ and $\beta'\neq0$, let $v=x_1+x_2$. Then $(v,x_3)=\beta'x_2$.
Again by Proposition~\ref{3prop:remarklocpwfetLie}, one has that $(v,x_3)\in\spa_{\Z_p}\{v,x_3\}$.
In particular, no multiple of $x_2$ lies in $\spa_{\Z_p}\{v,x_3\}$.
Therefore, this case is impossible.
 \item[(3)] If $\beta'=\gamma'=0$ then $\alpha'=0$ by (1). So $L$, and hence $G$, is abelian.
But this case was excluded previously.
\end{enumerate}

This yields $\beta'=0$ and $\alpha'=\gamma'\neq0$, with $\alpha'\in p.\Z_p$ (resp. in $4.\Z_2$),
and the claim follows from Proposition~\ref{3prop:remarklocpwfetLie}.

\medskip

\textit{Step 3:} Suppose that $G$ is locally powerful, torsion-free with $d(G)=n+1\geq4$,
and let $G$ be generated by $x_1,\ldots,x_{n+1}$.
Since $G$ is non-abelian we may assume without loss of generality that $x_1$ and $x_2$ do not commute.

Let $H\clsgp G$ be the subgroup generated by $x_1,\ldots,x_n$.
Thus by induction there is a unique (non-trivial) orientation $\theta\colon H\rightarrow\Z_p^\times$ such that $H$ is $\theta$-abelian.
In particular, we may assume that $[x_1,x_i]=x_i^\lambda$ and $[x_i,x_j]=1$ for all $2\leq i,j\leq n$,
where \[\lambda=\theta(x_1)-1\in p.\Z_p\smallsetminus\{0\}\] (resp. in $4.\Z_2\smallsetminus\{0\}$ for $p=2$).

Furthermore, let $H_i\clsgp G$ be the subgroup generated by $x_1,x_i,x_{n+1}$, for $2\leq i\leq n$.
By induction, for each $i$ there exists an orientation $\theta_i\colon H_i\rightarrow\Z_p^\times$ such that $H_i$ is $\theta_i$-abelian.

Since $\theta_i(x_1)=\theta(x_1)=1+\lambda$ and $\theta_i(x_i)=\theta(x_i)=1$ for all $i$,
then necessarily $\theta_i(x_{n+1})=1$ for all $i$;
i.e., \[ [x_1,x_{n+1}]=x_{n+1}^\lambda \quad\text{and}\quad [x_i,x_{n+1}]=1 \] for all $i$.
Hence we may extend $\theta$ to $G$ such that $\theta(x_{n+1})=1$.
Thus $G$ is $\theta$-abelian.

\medskip

{\it Step 4:} Assume now that $G$ is not finitely generated, and let $C$ be any finitely generated subgroup of $G$.
Thus $G$ is $\theta_C$-abelian, for some orientation $\theta_C\colon C\to \Z_p^\times$.
Also, let $L_C=\log(C)$ be the Lie algebra associated to $C$

Also, let $V_C$ be the commutator subalgebra $(L_C,L_C)\leq L_C$, and let $Z_C$ be the kernel of $\theta_C$,
considered as ideal of $L_C$.
Then $(v,w)=0$ for every $v\in Z_C$ and $w\in V_C$, and $V_C=\lambda_C Z_C$, for some $\lambda\in p\Z_p$
(and $\lambda\in4\Z_2$ for $p=2$).
Note that $V_C$ and $Z_C$ are abelian Lie algebras for every $C$.
Let 
\[H=\overline{\bigcup_{C<G}V_C} \quad\text{and}\quad Z=\overline{\bigcup_{C<G}Z_C},\]
where the line denotes the closure in the $\Z_p$-topology.
Then $H$ and $Z$ are abelian subgroups of $G$, $H$ is the commutator subgroup of $G$, and $G/Z\simeq\Z_p$.

For every element $x\in G\smallsetminus Z$, one has $\text{ad}(x)z=\lambda_xz$, for some $\lambda_x\in p\Z_p$
(resp., $\lambda_x\in4\Z_p$).
Let $x_0\in G\smallsetminus Z$ such that $\lambda_{x_0}$ has the minimal $p$-adic value.
Then we may define an orientation $\theta\colon G\to\Z_p^\times$ such that $\kernel(\theta)=Z$ and
$\theta(x_0)=1+\lambda_{x_0}$.
Then $\theta_C=\theta|_C$ for every finitely generated subgroup $C$, so that $G$ is $\theta$-abelian.
\end{proof}

\begin{rem}
Theorem~\ref{3thm:locally poweful groups} is the union of \cite[Theorem~A]{claudio:BK} (for the
finitely generated case) and \cite[Prop.~3.5]{cmq:fast} (for the infinitely generated case).
\end{rem}

Recall from Definition~\ref{1defi:pro-p groups} the definition of rank of a pro-$p$ group.
The following is a well-known result.

\begin{fact}\label{3fact:rank of free}
A non-abelian free pro-$p$ group has infinite rank. 
\end{fact}

From Theorem~\ref{3thm:titsalternative} and Theorem~\ref{3thm:locally poweful groups}
one gets the following.

\begin{thm}\label{3thm:equivalence theta-abelian}
 Let $G$ be a Bloch-Kato pro-$p$ group (and assume further that $G$ has a cyclotomic orientation $\theta'$
such that $\image(\theta')\simeq\Z_2$ for $p=2$).
Then the following are equivalent.
\begin{itemize}
 \item[(i)] $G$ is locally powerful.
 \item[(ii)] $G$ does not contain non-abelian closed free pro-$p$ subgroups.
 \item[(iii)] There exists an orientation $\theta\colon G\to\Z_p^\times$ such that $G$ is $\theta$-abelian. 
\end{itemize}
Moreover, if $G$ is finitely generated, the above conditions and the below ones are equivalent.
\begin{itemize}
 \item[(iv)] $G$ is powerful.
 \item[(v)] $G$ is $p$-adic analytic.\footnote{
A topological group $G$ is said to be a {\it $p$-adic analytic group} if $G$ has the structure of analytic manyfold
over $\Q_p$, and the function 
\[f\colon G\times G\longrightarrow G,\quad (x,y)\longmapsto xy^{-1}\] is analytic 
(see also \cite[Ch.~8]{ddsms:analytic})}
 \item[(vi)] the cohomology ring $H^\bullet(G,\F_p)$ is an exterior $\F_p$-algebra.
 \item[(vii)] $\ccd(G)=d(G)$.
 \item[(viii)] $d(C)=d(G)$ for every open subgroup $C$ of $G$, i.e., $G$ has
constant generating number on open subgroups. 
\end{itemize}
\end{thm}

\begin{proof}
 The equivalence between (i), (ii) and (iii) follows directly from Theorem~\ref{3thm:titsalternative}
and Theorem~\ref{3thm:locally poweful groups}.
Now assume that $G$ is finitely generated.

Condition (i) implies (iv) by definition.
Conversely, if $G$ is powerful, then the rank $\rank(G)$ is finite \cite[Theorem~3.13]{ddsms:analytic}.
Thus, by Fact~\ref{3fact:rank of free}, $G$ contains no non-abelian free subgroups, and by 
Theorem~\ref{3thm:titsalternative} $G$ is $\theta$-abelian, for some orientation $\theta\colon G\to\Z_p^\times$.
This establishes the equivalence between (iv) and the first three conditions.

By \cite[Theorem~8.18]{ddsms:analytic}, condition (i) implies condition (v), whereas
condition (v) implies condition (ii) follows by \cite[Theorem~8.32]{ddsms:analytic}.

Conditions (vi) and (vii) are equivalent by Lemma~\ref{3lem:cup injective}, and conditions
(iv) and (vi) are equivalent by Theorem~\ref{3thm:peter thomas}.

If $G$ is locally powerful, then it has finite rank by the argument above, thus every open subgroup $C$ of $G$ is
powerful.
Thus, (i) implies (viii) by \cite[Prop.~4.4]{ddsms:analytic}.
Conversely, if (viii) holds, then one has $\rank(G)=d(G)$, and $G$ contains no non-abelian free subgroups. 
\end{proof}

 For a finitely generated pro-$p$ group $G$, consider the property
\begin{equation}\label{3eq:constant generating number}
 d(C)-n=|G:C|(d(G)-n)
\end{equation}
for all open subgroups $C$ of $G$, with $n$ a fixed poritive integer.
In the early '80s, K.~Iwasawa observed that pro-$p$ groups satisfying \eqref{3eq:constant generating number}
have interesting representation-theoretic properties, and he raised the question of determining the groups 
satisfying \eqref{3eq:constant generating number} for each $n\geq1$.

For $n=1$, \eqref{3eq:constant generating number} becomes the well-known (topological) Schreier index formula,
which characterizes free pro-$p$ groups \cite[Corollary~3.9.6]{nsw:cohm}. 
In \cite{dummitlabute:formula}, D.~Dummit and Labute answer Iwasawa's question for $n=2$ in the case
of groups with one defining relation: $G$ is a Demushkin group if and only if G is a one-relator torsion-free
pro-$p$ group satisfying \eqref{3eq:constant generating number} with $n=2$.
 
Hnece, Theorem~\ref{3thm:equivalence theta-abelian}(viii) provides the answer to Iwasawa's question with $n=d(G)$
in the category of finitely generated Bloch-Kato pro-$p$ groups.
A similar answer has been proven in \cite{klopschsnopce:generating} for the category of $p$-adic analytic
pro-$p$ groups.
It is interesting to remark that the groups listed in \cite[Theorem~1.1.(1)-(2),~(4)]{klopschsnopce:generating}
are $\theta$-abelian for some orientation $\theta$ (in fact the groups listed in (4) are not powerful,
as the image of the orientation is torsion), whereas the groups listed in
\cite[Theorem~1.1.(3)]{klopschsnopce:generating} are not Bloch-Kato
(and thus they cannot be realized as maximal pro-$p$ Galois groups), for they have non-trivial 3-torsion.

\section[Rigid fields]{Rigid fields and $\theta$-abelian Galois groups}

Let $K$ be a field containing a root of unity of order $p$, and let $G_K(p)$ be
the maximal pro-$p$ Galois group of $K$, with arithmentic orientation $\theta\colon G_K(p)\to\Z_p^\times$.
The condition of $\theta$-abelianity of the group $G_K(p)$ can be translated in the arithmentic ``language''
(i.e., in conditions on the field $K$) in some different ways.

Recall from \eqref{1eq:kummer_dualityp} that the first $\F_p$-cohomology group of $G_K(p)$ is isomorphic to
the quotient $K^\times/(K^\times)^p$ as $\F_p$-vector space.
Also, fix an isomorphism of $G_K(p)$-modules $\mu_p\simeq\F_p$.
Then by \eqref{2eq:inflation BlochKato} and by Corollary~\ref{1cor:H2 and Brauer group} one has that
the second $\F_p$-cohomology group of $G_K(p)$ is canonically isomorphic to the subgroup $\bra_p(K)$ of
the Brauer group of $\bra(K)$.

It is possible to define a map from $K^\times/\left(K^\times\right)^p\wedge K^\times/\left(K^\times\right)^p$ to
$\bra_p(K)$ as follows.
For $a,b\in K^\times$, the {\bf cyclic $K$-algebra} generated by $a$ and $b$ is the $K$-algebra
\[(a,b)_K=\frac{K\langle u,v\rangle}{\left\langle u^p=a,v^p=b,uv=\zeta vu\right\rangle},\]
where $\zeta\in K^\times$ is the fixed primitive $p^{th}$-root of the unity
(which depends on the isomorphism $\mu_p\simeq\F_p$).
For $a,b\in K^\times$, the cyclic algebra $(a,b)_K$ represents a class of the Brauer group $\bra(K)$.
Moreover, for every $a,b\in K^\times$, the algebra $(a^p,b)_K$ splits, so that in fact the class of $(a,b)_K$
lies in $\bra_p(K)$, and we may define a map 
\begin{equation}\label{3eq:map cyclic algebra}
 \xymatrix{ K^\times/\left(K^\times\right)^p\wedge K^\times/\left(K^\times\right)^p
\ar[rr]^-{(\argu,\argu)_K} && \bra_p(K) }.
\end{equation}

In their work on the case $n=2$ of the Bloch-Kato conjecture,
Merkur'ev and Suslin proved that one has a commutative diagram
\begin{equation}\label{3eq:diagram MerkurevSuslin}
  \xymatrix{ K^\times/\left(K^\times\right)^p\wedge K^\times/\left(K^\times\right)^p\ar@<7ex>[d]^{\wr} \ar@<-7ex>[d]^{\wr} 
\ar[rr]^-{(\argu,\argu)_K} && \bra_p(K) \\
 H^1 \left(G_K(p),\F_p \right) \wedge H^1 \left (G_K(p),\F_p \right) \ar[rr]^-{\wedge_2(\cup)} 
&& H^2 \left(G_K(p),\F_p \right) \ar[u]_{\wr}}
\end{equation}
where the upper horizontal arrow (and thus also the lower one) is an epimorphism of vector spaces over $\F_p$.

\subsection{$p$-rigid fields}

Let $a$ be a $p$-power free unit of $K$ (i.e., $a\in K^\times\smallsetminus (K^\times)^p$).
Then $K(\sqrt[p]{a})/K$ is a Galois extension of degree $p$, with cyclic Galois group generated by an element
$\bar\sigma\in\Gal(K(\sqrt[p]{a})/K)$ of order $p$.
Recall that the {\bf norm} $\mathfrak{N}\colon K(\sqrt[p]{a})\to K$ of the extension $K(\sqrt[p]{a})/K$ is defined by 
\begin{eqnarray*}
 \mathfrak{N}(\gamma) &=& \gamma\cdot\bar\sigma(\gamma)\cdots\bar\sigma^{p-1}(\gamma) \\ &=&
\prod_{i=0}^{p-1}\left(c_0+c_1\zeta^i\sqrt[p]{a}+\ldots+c_{p-1}\zeta^{i(p-1)}\sqrt[p]{a^{p-1}}\right),
\end{eqnarray*}
with $\gamma=c_0+c_1\sqrt[p]{a}+\ldots+c_{p-1}\sqrt[p]{a^{p-1}}$.

\begin{defi}
Let $K$ be a field containing a primitive $p$-th root of unity (and containing also $\sqrt{-1}$, if $p=2$).
\begin{itemize}
 \item[(i)] An element $a\in K^\times$ which is $p$-power free is said to be {\bf $p$-rigid}
if $\mathfrak{N}(K(\sqrt[p]{a}))$ is contained in the set $\bigcup_{n=0}^{p-1}a^nK^p$.
 \item[(ii)] The field $K$ is said to be {\bf $p$-rigid} if
all of the elements in $K\smallsetminus K^p$ are $p$-rigid.
\end{itemize}
\end{defi}

\begin{rem}\label{3rem:norm map}
\begin{itemize}
 \item[(i)] Since $\mathfrak{N}(K)=K^p$, every $p$-power of $K$ is a norm of $K(\sqrt[p]{a})/K$.
 \item[(ii)] The restriction of the norm on the multiplicative group $K(\sqrt[p]{a})^\times$ is a group homomorphism.
 \item[(iii)] If $p$ is odd, then $\mathfrak{N}(\sqrt[p]{a^m})=a^m$, with $m$ a positive integer,
whereas if $p=2$ then $\mathfrak{N}(\sqrt{a})=-a$.
\end{itemize}
\end{rem}

The notion of $p$-rigidity  of fields was introduced by K.~Szymiczek in \cite{szym:rigid},
and it was developed and studed first in the case $p=2$, and then in the case of an odd prime. 
The notion of $p$-rigidity is important expecially for the consequences on the existence of valuations
(see \cite[\S~11.3-11.4]{efrat:libro}).

\begin{exs}
\begin{itemize}
 \item[(a)] Let $\ell$ be a prime different to $p$ and $k$ a positive integer such that $p\mid(\ell^k-1)$,
and let $K$ be the field $\F_{\ell^k}\dblaurl X\dblaurr$, i.e, $K$ is the field of Laurent series on
the indeterminate $X$ with coefficients in the finite field $\F_{\ell^k}$.
Then $K$ is $p$-rigid \cite[p.~727]{ware:galp}.
 \item[(b)] Let $\zeta$ be a primitive $p$-th root of unity and $\ell$ a prime different to $p$.
Then the $p$-adic field $\Q_\ell(\zeta)$ is $p$-rigid \cite[Exam.~3.1]{cmq:fast}.
\end{itemize}
\end{exs}

In the case $p=2$, one may understand a $2$-rigid field by studying the structure of the Witt ring of quadratic
forms of the field. 
The consequences of $2$-rigidity were studied in several papers, and today many results about 2-rigid fields
are known, and these fields are rather well understood (see, for example,
\cite{ware:rigid}, \cite{arasonelmanjacob:rigid}, \cite{jacobware:pro2} and \cite{leepsmith:2rigidity}).

The problem to generalize the results obtained in the case $p=2$ for odd primes
has been considered rather hard for long time. 
A first difficulty is that there is no usable version of Witt ring for higher degree forms.
Thus, one has to find an ``odd substitute'' of a quaternionic pairing.
(Indeed, quaternionic pairings were developed as a tool for studying abstract Witt rings.)
The first attempts to use the cup product in Galois cohomology to study the maximal pro-$p$ Galois groups
of $p$-rigid fields are to be found in \cite{ware:galp} first, and then, more explicitly, in \cite{hwangjacob:brauer}.
On the other hand, \cite{ek98:abeliansbgps} makes no use of cohomological tools. 

It is worth to stress that all the contribuitions to $p$-rigid fields listed so far make use of instruments
provided by the theory of valuations, whereas our approach (based only on Galois cohomology and the theory
of pro-$p$ groups) makes it possible to study $p$-rigid without any ``help'' from valuation theory.
In fact, the following result is the key to relate $p$-rigid fields with the cohomology of maximal
pro-$p$ Galois groups.

\begin{prop}\label{3prop:definition rigidity}
 Let $K$ be a field containing a root of unity of order $p$ (and also $\sqrt{-1}$, if $p=2$).
Then $K$ is $p$-rigid if, and only if, the map $\wedge_2(\cup)$ in the $\F_p$-cohomology of the Galois group
$G_K(p)$ is a monomorphism.
\end{prop}

\begin{proof}
The definition of $p$-rigidity states that for every $a$, the norms of $K(\sqrt[p]{a})/K$ modulo $(K^\times)^p$
are precisely the $\F_p$-multiples of the $\bar a^i$'s, with $\bar a^i=a^i\bmod (K^\times)^p$ and $0\leq i\leq p-1$,
i.e.,
\[\mathfrak{N}(\sqrt[p]{a})=\spa_{\F_p}\{\bar a,\ldots,\bar a^{p-1}\}\leq K^\times/(K^\times)^p.\]

It is well known that the class of $(a,b)_K$ is trivial in $\bra(K)$ if, and only if, $b$ is a norm
of $K(\sqrt[p]{a})/K$ -- in particular, if $b$ is a $p$-power of $K$ or if $\bar b$ is $\F_p$-linearly dependent
to $\bar a$ in $K^\times/(K^\times)^p$, by Remark~\ref{3rem:norm map}. 

Consequently, $K$ is $p$-rigid if, and only if, the algebra $(a,b)_K$ splits if, and only if, $\bar b$
is $\F_p$-linearly dependent to $\bar a^i$ in $K^\times/(K^\times)^p$ for some $i$, namely,
the map \eqref{3eq:map cyclic algebra} is a monomorphism.
Therefore the claim follows from the commutative diagram \eqref{3eq:diagram MerkurevSuslin}. 
\end{proof}

This provides a first immediate consequence.

\begin{cor}
Let $K$ be a field containing a primitive $p$-th root of unity (and also $\sqrt{-1}$, if $p=2$),
and assume that the quotient $K^\times/(K^\times)^p$ is finite.
Then $K$ is $p$-rigid if, and only if,
 the maximal pro-$p$ Galois group $G_K(p)$ is $\theta$-abelian, with $\theta\colon G_K(p)\to \Z_p^\times$
the arithmetical orientation.
\end{cor}

Note that under the assumption that $K$ contains $\sqrt{-1}$, one has that either $\image(\theta)\simeq\Z_2$
or $\theta\equiv{\mathbf 1}$.

In order to study explicitly the structure of maximal pro-$p$ Galois group of $p$-rigid fields, R.~Ware introduced
in \cite{ware:galp} also the notion of hereditary $p$-rigidity:
a field $K$ is said to be hereditarily $p$-rigid if every subextension of the $p$-closure $K(p)$ is $p$-rigid.
As Ware pointed out, to conclude that $K$ is hereditary $p$-rigid, it is enough to check that
each finite extension $L/K$ is $p$-rigid.
In \cite{ek98:abeliansbgps}, A.~Engler and J.~Koenigsmann showed that $p$-rigidity implies hereditary $p$-rigidity,
and later also the author, with S.K.~Chebolu and J.~Min\'a\v{c}, proved the same result using a different argument
(cf. \cite[\S~3.3]{cmq:fast}). The former proof makes use of techniques from the theory of valuations.

\begin{thm}\label{3thm:hereditary rigidity}
 If a field $K$ containing a primitive $p$-th root of unity (and $\sqrt{-1}$ for $p=2$) is $p$-rigid, 
then every $p$-extension $L$ of $K$ is rigid.
\end{thm}

\begin{rem}
Note that if the quotient $K^\times/(K^\times)^p$ is finite, then it is possible to prove that a $p$-rigid field 
is also hereditarily $p$-rigid without making use of arithmetical tools.
Indeed, assume that the field $K$ (containing a primitive $p$-th root of unity and also $\sqrt{-1}$ if $p=2$)
is $p$-rigid.
Then by Proposition~\ref{3prop:definition rigidity} and Theorem~\ref{3thm:peter thomas} the maximal pro-$p$
Galois group $G_K(p)$ is powerful, and by Theorem~\ref{3thm:equivalence theta-abelian} it is also locally powerful.
Thus again by Proposition~\ref{3prop:definition rigidity} and Theorem~\ref{3thm:peter thomas}
every $p$-extension $L$ of $K$ is $p$-rigid.
\end{rem}

Therefore, one has an ``arithmetical version'' of Theorem~\ref{3thm:equivalence theta-abelian}.

\begin{thm}\label{3thm:rigidity tutto}
 Let $K$ be a field containing a primitive $p$-th root of unity (and $\sqrt{-1}$ for $p=2$), and let $G$ be
the maximal pro-$p$ Galois group of $K$, equipped with the arithmetical orientation $\theta\colon G\to \Z_p^\times$.
Then, $K$ is $p$-rigid if, and only if:
\begin{itemize}
 \item[(i)] $G$ is $\theta$-abelian and $\image(\theta)=1+p^k\Z_p$, where $k$ is the maximum
positive integer such that $\mu_{p^k}\subseteq K$ and $\mu_{p^{k+1}}\nsubseteq K$
(in case $\mu_{p^\infty}\subseteq K$, then $k=\infty$ and $\theta\equiv\mathbf{1}$);
 \item[(ii)] there are no $p$-extensions $L/K$ such that the maximal pro-$p$ Galois group $G_L(p)$ is
free and non-abelian; 
 \item[(iii)] $G$ is solvable;\footnote{I.e., $G$ admits a finite normal series of closed subgroups such that
each successive quotient is abelian}
\end{itemize}
Assume further that the quotient $K^\times/(K^\times)^p$ is finite.
Then, $K$ is $p$-rigid if, and only if:
\begin{itemize}
 \item[(iv)] $G$ is powerful;
 \item[(v)] the cohomology ring $H^\bullet(G,\F_p)$ is isomorphic to the exterior $\F_p$-algebra
$\bigwedge_{i\geq0}K^\times/(K^\times)^p$;
 \item[(vi)] one has 
\[\ccd(G)=\dim_{\F_p}\left(K^\times/(K^\times)^p\right)=\frac{1}{p}\left|K^\times/(K^\times)^p\right|.\]
 \item[(vii)] for every finite $p$-extension $L/K$ one has 
\[\dim_{\F_p}\left(L^\times/(L^\times)^p\right)=\dim_{\F_p}\left(K^\times/(K^\times)^p\right).\]
\end{itemize}
\end{thm}

The above result extends and generalizes the results proved in \cite{ware:galp}
and some in \cite{ek98:abeliansbgps}.
In particular, Corollary~2, Theorem~2 and Theorem~4 in the former paper require the assumption that $K$
contains also a root of unity of order $p^2$, which is not necessary in our result.

\section[Finite quotients]{Finite quotients and $\theta$-abelian groups}\label{3sec:finite quotients}
It is possible to obtain further characterizations of $\theta$-abelian pro-$p$ groups
which involve finite quotients, and which have a particular arithmetical interpretation.

For a pro-$p$ group $G$, let $\Phi_2(G)$ be the Frattini subgroup of the Frattini subgroup, 
i.e., \[\Phi_2(G)=\Phi\left(\Phi(G)\right).\]
Thus, $G^{p^2}\leq\Phi_2(G)$, so that the quotient $G/\Phi_2(G)$ has exponent $p^2$.
In particular, if $G$ is finitely generated, then the quotient $G/\Phi_2(G)$ is a finite $p$-group.
Also, one has 
\begin{equation}\label{3eq:lambda3Phi2}
 \lambda_3(G)=\Phi_2(G)\cdot[G,\Phi(G)],
\end{equation}
where $\lambda_3(G)$ is the third element of the $p$-descending central series (see Subsection~1.2.1).

If $G$ is $\theta$-abelian for some orientation $\theta$, then one has a tighter relation between
$\lambda_3(G)$ and $\Phi_2(G)$.

\begin{thm}\label{3thm:lambda3Phi2}
Let $G$ be a finitely generated Bloch-Kato pro-$p$ group.
Then there exists an orientation $\theta\colon G\to\Z_p^\times$ such that $G$ is $\theta$-abelian if,
and only if, one has the equality \[\Phi_2(G)=\lambda_3(G).\]
\end{thm}

\begin{proof}
Assume first that $G$ is $\theta$-abelian, for some orientation $\theta$.
Then one has 
\[\Phi_2(G)=\Phi(G)^p[\Phi(G),\Phi(G)]=G^{p^2}=\lambda_3(G)\]
-- in fact one has $\lambda_i(G)=G^{p^{i-1}}$ for every $i\geq1$.

Conversely, assume that $\Phi_2(G)=\lambda_3(G)$.
Since 
\[ \left[D_2(G),D_2(G)\right]\leq D_4(G) \quad\text{and}\quad D_2(G)^p\leq D_{2p}(G),\]
one has \[\Phi_2(G)=D_2(G)^p\left[D_2(G),D_2(G)\right]\leq D_4(G),\] as $\Phi(G)=D_2(G)$.
Moreover, one has the inclusion $\gamma_3(G)\leq \lambda_3(G)$.
Therefore, one has the chain of inclusions
\begin{equation} \label{3eq:inclusions quotients}
\gamma_3(G)\leq \lambda_3(G)=\Phi_2(G)\leq D_4(G).
\end{equation}

We shall split the proof of this implication in three cases.
\begin{itemize}

 \item[(i)] Assume $p>3$.
By Lazard's formula \eqref{1eq:lazard forumula}, one has
\begin{eqnarray*}
&& D_3(G)=\prod_{ip^h\geq3}\gamma_i(G)^{p^h}=\gamma_3(G)\cdot G^p, \\
&& D_4(G)=\prod_{ip^h\geq4}\gamma_i(G)^{p^h}=\gamma_4(G)\cdot G^p.
\end{eqnarray*}
Therefore, \eqref{3eq:inclusions quotients} implies
\[ D_3(G)=\gamma_3(G)\cdot G^p\leq \lambda_3(G)=\Phi_2(G)\cdot G^p\leq D_4(G),\]
as $G^p\leq D_4(G)$.
Thus, one has the equality $D_3(G)=D_4(G)$.
Hence, Proposition~\ref{1prop:properties Zassenhaus filtration} implies that $\rk(G)$ is finite,
and thus by Theorem~\ref{3thm:equivalence theta-abelian} there is an orientation $\theta\colon G\to\Z_p^\times$
such that $G$ is $\theta$-abelian.

 \item[(ii)] Assume $p=2$.
From \eqref{1eq:lazard forumula} one obtains
\begin{eqnarray*}
&& D_3(G)=\prod_{i2^h\geq3}\gamma_i(G)^{2^h}=\gamma_3(G)\cdot \gamma_2(G)^2\cdot G^4, \\
&& D_4(G)=\prod_{i2^h\geq4}\gamma_i(G)^{2^h}=\gamma_4(G)\cdot \gamma_2(G)^2\cdot G^4.
\end{eqnarray*}
Therefore, \eqref{3eq:inclusions quotients} implies
\[D_3(G)=\gamma_3(G)\gamma_2(G)^2 G^4\leq \Phi_2(G)\gamma_2(G)^2 G^4\leq D_4(G),\]
as $\gamma_2(G)^2 G^4\leq D_4(G)$.
Thus, one has the equality $D_3(G)=D_4(G)$.
Hence, Proposition~\ref{1prop:properties Zassenhaus filtration} implies that $\rk(G)$ is finite,
and thus by Theorem~\ref{3thm:equivalence theta-abelian} there is an orientation $\theta\colon G\to\Z_p^\times$
such that $G$ is $\theta$-abelian.

\item[(iii)] Assume $p=3$.
By \eqref{1eq:lazard forumula}, one has
\begin{eqnarray*}
&& D_4(G)=\prod_{i3^h\geq4}\gamma_i(G)^{3^h}=\gamma_4(G)\cdot\gamma_2(G)^3\cdot G^9, \\
&& D_5(G)=\prod_{i3^h\geq5}\gamma_i(G)^{3^h}=\gamma_5(G)\cdot\gamma_2(G)^3\cdot G^9.
\end{eqnarray*}
Therefore, from \eqref{3eq:inclusions quotients} one obtains the chain of inclusions
\[ \gamma_4(G) = [G,\gamma_3(G)] \leq [G,D_4(G)] = [D_1,D_4(G)] \leq D_5(G),\]
which implies
\[ D_4(G) =\gamma_4(G)\cdot\gamma_2(G)^3\cdot G^9 \leq D_5(G),\]
as $G^9,\gamma_2(G)^3\leq D_5(G)$.
Thus, one has the equality $D_4(G)=D_5(G)$.
Hence, Proposition~\ref{1prop:properties Zassenhaus filtration} implies that $\rk(G)$ is finite,
and thus by Theorem~\ref{3thm:equivalence theta-abelian} there is an orientation $\theta\colon G\to\Z_p^\times$
such that $G$ is $\theta$-abelian.
\end{itemize}
\end{proof}

\begin{cor}
Let $G$ be a finitely generated Bloch-Kato pro-$p$ group.
Then there exists an orientation $\theta\colon G\to\Z_p^\times$ such that $G$ is $\theta$-abelian if,
and only if, one has 
\begin{equation}\label{3eq:zent_in_quotient}
 \Phi(G)/\Phi_2\leq\Zen\left(G/\Phi_2\right).
\end{equation}
\end{cor}

\begin{proof}
 Assume that $G$ is $\theta$-abelian.
Then by Theorem~\ref{3thm:lambda3Phi2} one has $\lambda_3(G)=\Phi_2(G)$,
and $\lambda_2(G)/\lambda_3(G)$ is central in $G/\lambda_3(G)$.

Conversely, assume that $\Phi(G)/\Phi_2$ is central in $G/\Phi_2(G)$.
This implies that the commutator subgroup $[G,\Phi(G)]$ is contained in $\Phi_2(G)$.
Since
\[\Phi_2(G)=\Phi\left(\Phi(G)\right)\geq\left(\Phi(G)\right)^p\quad
\text{and}\quad \lambda_3(G)=\left(\Phi(G)\right)^p\left[G,\Phi(G)\right],\]
it follows that $\Phi_2(G)$ contains $\lambda_3(G)$, and thus $\lambda_3(G)=\Phi_2(G)$.
\end{proof}

\subsection{Rigid fields and small Galois groups}
The above results can be translated in the arithmetical ``language'' in terms of finite extensions.
Let $K$ be a field containing a root of unity of order $p$, and let $G$ be the maximal
pro-$p$ Galois group $G_K(p)$.
Recall that by Proposition~\ref{1prop:roots and kummer} the quotient $G/\Phi(G)$ is the Galois group
of the extension of $K(\sqrt[p]{K})/K$.
Set $L=K(\sqrt[p]{K})$.
Then $L(\sqrt[p]{L})/K$ is again a Galois extension, and one has
\[\Gal\left(L(\sqrt[p]{L})/L\right)=\frac{\Phi(G)}{\Phi_2(G)}\quad\text{and}\quad
\Gal\left(L(\sqrt[p]{L})/K\right)=\frac{G}{\Phi_2(G)},\]
as $\Phi_2(G)$ is the maximal pro-$p$ Galois group of the field $L(\sqrt[p]{L})$.
An important subextension of $L(\sqrt[p]{L})$ is the following: let $A\subseteq L^\times$ be the set
\[A=\left\{ a\in L^\times\left|\: L\left(\sqrt[p]{a}\right)/K\text{ is Galois} \right.\right\}.\]
Consider the field $L(\sqrt[p]{A})$.
The following well-known fact (cf. \cite[Lemma~2.2]{cmq:fast}) is necessary to determine the Galois group of
$L(\sqrt[p]{A})/K$.

\begin{fact}\label{3fact:ledet}
Let $K$ be a field containing a primitive $p^{th}$-root of unity, and let $L/K$ be a $p$-extension
with $a\in L^\times$.
Then $L(\sqrt[p]{a})/K$ is Galois if, and only if, \[\frac{\sigma(a)}{a}\in L^p\]
for every $\sigma\in\Gal(L/K)$.
\end{fact}

\begin{prop}\label{3prop:coinvariants}
Let $K$ be a field containing a root of unity of order $p$ (and containing also $\sqrt{-1}$ in the case $p=2$),
and let $L/K$ and $A\subseteq L^\times$ be as above.
Then the subgroup $\lambda_3(G)$ of $G$ is the maximal pro-$p$ Galois group of $L(\sqrt[p]{A})$,
i.e., \[G/\lambda_3(G)=\Gal\left(L(\sqrt[p]{A})/K\right).\]
\end{prop}

\begin{proof}
By Proposition~\ref{1prop:roots and kummer} one knows that 
 \[\frac{L^\times}{(L^\times)^p}\simeq \left(\frac{\Phi(G)}{\Phi_2(G)}\right)^\vee.\]
By Fact~\ref{3fact:ledet}, the set $A$ corresponds to the $G$-invariant submodule $(L^\times/(L^\times)^p)^G$
of $L^\times/(L^\times)^p$, considered as discrete $G$-module.

On the other hand, \eqref{3eq:lambda3Phi2} implies that the quotient $\Phi(G)/\lambda_3(G)$ is 
isomorphic to the $G$-coinvariant quotient of the discrete $G$-module $\Phi(G)/\Phi_2(G)$.
Therefore, by duality one has the isomorphisms of discrete $G$-modules
\[ \frac{A}{\left(L^\times\right)^p}=\left(\frac{L^\times}{\left(L^\times\right)^p}\right)^G\simeq
\left(\left(\frac{\Phi(G)}{\Phi_2(G)}\right)_G\right)^\vee\simeq\left(\frac{\Phi(G)}{\lambda_3(G)}\right)^\vee, \]
so that the claim follows by Kummer duality.
\end{proof}

Theorem~\ref{3thm:lambda3Phi2} and Proposition~\ref{3prop:coinvariants} have the following consequence.

\begin{cor}
Let $\theta\colon G_K(p)\to\Z_p^\times$ be the arithmetical orientation of the maximal pro-$p$ Galois group of $K$.
Then $G_K(p)$ is $\theta$-abelian if, and only if, the extension $L(\sqrt[p]{a})/K$ is Galois for every $a\in L$.
\end{cor}

The equality of Galois groups $\lambda_3(G_K(p))=\Phi_2(G_K(p))$ is equivalent to the equality of extensions
\begin{equation}\label{3eq:equality of fields}
 L(\sqrt[p]{A})/K=L(\sqrt[p]{L})/K.
\end{equation}
Therefore, $K$ is $p$-rigid if, and only if, equality \eqref{3eq:equality of fields} holds. \footnote{
Note that the following notation is rather common in papers on rigid fields: $K^{(2)}=L$, $K^{(3)}=L(\sqrt[p]{A})$
and $K^{\{3\}}=L(\sqrt[p]{L})$, see for example \cite[\S~2.3]{cmq:fast}}
This is the main result in \cite{cmq:fast}, and it completes the the result obtained by D.~Leep and T.~Smith
for 2-rigid fields in \cite{leepsmith:2rigidity}.

In the case $p=2$ the quotient $G_K(p)/\lambda_3(G_K(p))$ has been extensively studied under the name of ``$W$-group'',
in particular in connection with quadratic forms (see, e.g., \cite{minacspira:pytagoras}).
It was further shown that in this case such quotient has great arithmetical significance, as  it encodes information
about orderings and valuations of $K$.
In the case $p$ odd the quotient $G_K(p)/\lambda_3(G_K(p))$ has been studied by S.K.~Chebolu, J.~Min\'a\v{c}
and Efrat in \cite{cem:quotients} (see also Theorem~\ref{2thm:cem}).

Also, one has the following (cf. \cite[\S~4.2]{cmq:fast}).

\begin{thm}
Let $K$ be a field containing a root of unity of order $p$ (and containing also $\sqrt{-1}$ in the case $p=2$).
Then $K$ is $p$-rigid if, and only if the $p$-closure of $K$ is obtained by adjoining
the $p$-power roots of all the elements of $K$, namely,
\[K(p)=K\left(\sqrt[p^m]{a},m\geq1,a\in K\right).\]
\end{thm}

Note that if the field $K$ is $p$-rigid, then one needs only the quotient $K^\times/(K^\times)^p$ and the image
of the arithmetical orientation $\theta\colon G_K(p)\to\Z_p^\times$ to recover the entire structures of the 
maximal pro-$p$ Galois group $G_K(p)$ and of the $p$-closure $K(p)/K$.

\section[The Zassenhaus filtration]{The Zassenhaus filtration for $\theta$-abelian groups}

The rather simple structure of a $\theta$-abelian pro-$p$ group makes it possible to compute explicitly 
the Zassenhaus filtration and the induced quotients.

Let $(G,\theta)$ be a $\theta$-abelian pro-$p$ group, with $k\in\N\cup\{\infty\}$ such that $\image(\theta)=1+p^k\Z_p$
(in particular, $k=\infty$ if, and only if, $\theta\equiv{\mathbf 1}$, as $p^\infty=0$ in the pro-$p$ topology).
Then one has \[\gamma_i(G)=\kernel(\theta)^{p^{k(i-1)}}\quad \text{for every }i>1.\]
For every $n\geq1$ set $\ell=\lceil\log_p(n)\rceil$, i.e., $\ell$ is the least integer such that $n\leq p^\ell$.
Hence, by Lazard's formula \eqref{1eq:lazard forumula}, one has
\begin{equation}\label{3eq:dimension_subgps}
 D_n(G)=G^{p^\ell}\prod_{ip^h\geq n}\kernel(\theta)^{p^{k(i-1)+h}},\quad \text{with }i\geq 2.
\end{equation}
We shall show that for every $i,h$ such that $i\geq2$ and $ip^h\geq n$, one has the inequality
\begin{equation}\label{3eq:stubid_inequality}
k(i-1)+h \geq \ell,
\end{equation}
so that $\kernel(\theta)^{p^{k(i-1)+h}}\leq G^{p^\ell}$, and $D_n(G)=G^{p^\ell}$.
If $h\geq \ell$, then inequality (\ref{3eq:stubid_inequality}) follows immediately.
Otherwise, notice that $i>p^{\ell-h-1}\geq1$, as $ip^h\geq n > p^{\ell-1}$, which implies
\begin{equation}\label{3eq:stubid_inequality2}
 k(i - 1) > k \left(p^{(\ell-h)-1}-1\right).
\end{equation}
Therefore, for $\ell-h\geq2$, the inequality (\ref{3eq:stubid_inequality2}) implies $k(i-1)\geq\ell-h$,
and thus (\ref{3eq:stubid_inequality}), whereas for $\ell-h=1$ (\ref{3eq:stubid_inequality})
follows from the fact that $i\geq2$.

Altogether, this shows that
\begin{equation}\label{3eq:Zassenhaus filtration for thetabelian}
 D_n(G)=G^{p^\ell}\quad \text{for } p^{\ell-1}<n\leq p^\ell.
\end{equation}
Consequently, the quotients induced by the Zassenhaus filtration are
\begin{equation}\label{3eq:quotients_Zassenhaus}
 L_i(G)=\frac{D_i(G)}{D_{i+1}(G)}\simeq\left\{\begin{array}{cc} G/\Phi(G) & \text{for $i$ a $p^{th}$-power} \\
0 & \text{otherwise} \end{array}\right.
\end{equation}
as $\F_p$-vector spaces.
Therefore, $L_\bullet(G)$ is an abelian restricted Lie algebra over $\F_p$, and by Proposition~\ref{1prop:presentation restricted envelope}
the restricted envelope is the polynomial $\F_p$-algebra in $d(G)$ indeterminates, i.e.,
\begin{equation}\label{3eq:restricted envelope theta-abelian}
 \grad_\bullet(G)\simeq\mathcal{U}_p(L_\bullet(G))\simeq\F_p\left[X_1,\ldots,X_{d(G)}\right],
\end{equation}
with the grading induced by the degrees of the monomials.

\begin{rem}
 It is worth to stress that the Zassenhaus filtration loses completely the information about the image of the
orientation $\theta$ -- as it happens for Demu\v{s}kin groups, see \eqref{2eq:restricted envelope demushkin}.
\end{rem}

\begin{thm}
 Let $(G,\theta)$ be a finitely generated pro-$p$ group with a cyclotomic orientation, and assume further that
$\image(\theta)\leq1+4\Z_2$ if $p=2$.
Then the following are equivalent:
\begin{itemize}
 \item[(i)] $G$ is $\theta$-abelian;
 \item[(ii)] the algebra $L_\bullet(G)$ is an abelian restricted Lie algebra over $\F_p$;
 \item[(iii)] the algebra $\grad_\bullet(G)$ is a commutative polynomial $\F_p$-algebra.
\end{itemize}
\end{thm}

\begin{proof}
 Assume first that $G$ is $\theta$-abelian.
Then (ii) and (iii) follow from \eqref{3eq:quotients_Zassenhaus} and \eqref{3eq:restricted envelope theta-abelian}.

Assume that the algebra $L_\bullet(G)$ is abelian.
Then \eqref{3eq:restricted envelope theta-abelian} follows immediately.
Moreover, one has $L_i(G)=D_i(G)/D_{i+1}(G)$ whenever $i$ is not a $p$-power.
In particular, $D_i(G)=D_{i+1}(G)$ for every such $i$'s, so that \cite[Theorem~11.4]{ddsms:analytic} implies that
$G$ has finite rank, and (i) follows by Theorem~\ref{3thm:equivalence theta-abelian}.

Assume now that $\grad_\bullet(G)$ is a commutative polynomial $\F_p$-algebra.
Since the map $\psi_{L_\bullet(G)}\colon L_\bullet(G)\to\grad_\bullet(G)$ is a monomorphism,
also the restricted Lie algebra $L_\bullet(G)$ has to be commutative, so that (iii) implies (ii).
\end{proof}

\begin{cor}
 Let $K$ be a field containing a primitive $p$-th root of unity
such that the quotient $K^\times/(K^\times)^p$ is finite, and assume further that
$\sqrt{-1}$ lies in $K$ if $p=2$.
Then the following are equivalent:
\begin{itemize}
 \item[(i)] $K$ is $p$-rigid;
 \item[(ii)] the algebra $L_\bullet(G_K(p))$ induced by the maximal pro-$p$ Galois group $G_K(p)$
is an abelian restricted Lie algebra over $\F_p$;
 \item[(iii)] the algebra $\grad_\bullet(G_K(p))$ is a commutative polynomial $\F_p$-algebra.
\end{itemize}
\end{cor}


\chapter{Products, relations, and the ETC}

\section{The group $H^1(G,\Z_p(1))$}

Let $(G,\theta)$ be a cyclo-oriented pro-$p$ group.
Our first goal is to study the structure of the cohomology group $H^1(G,\Z_p(1))$
and the continuous crossed homomorphism $f\colon G\to\Z_p(1)$ which represent the cohomology classes.

The following result is a consequence of \cite[Prop.~6]{labute:classification}.

\begin{prop}\label{4prop:inizio}
Let $(G,\theta)$ be a finitely generated pro-$p$ group with orientation.
Then $\theta$ is cyclotomic if, and only if, one may arbitrarily prescribe the values of
a continuous crossed homomorphisms $f\colon G\to\Z_p(1)$ on a minimal system of generators of $G$.
\end{prop}

\begin{proof}
First of all, recall that by Proposition~\ref{2lem:H2 torfree iff H1 projects}, $\theta$ is a cyclotomic orientation
for $G$ if, and only if, one has the epimorphism \eqref{2eq:p projection}.

Assume first that it is possible to arbitrarily prescribe the values of a crossed homomorphism $f$ on
a minimal system of generators of $G$.
Then it is easy to see that \eqref{2eq:p projection} is an epimorphism: indeed, for every $\chi\in H^1(G,\F_p)$,
$\chi$ is the image of a crossed homomorphism $f\colon G\to\Z_p(1)$ such that $f(x)\equiv\chi(x)\bmod p$
for every generator $x$ of $G$.

Conversely, assume that \eqref{2eq:p projection} is an epimorphism.
The composition of projections $\Z_p(1)\twoheadrightarrow\Z_p(1)/p^n\Z_p(1)\twoheadrightarrow\F_p$ for every $n\geq1$
implies that the map
\[\xymatrix{ H^1(G,\Z_p(1)/p^n\Z_p(1))\ar[r] & H^1(G,\F_p)}\]
is surjective for every $n\geq1$.
In particular, condition (1) of \cite[Prop.~6]{labute:classification} holds, and this yields the thesis.
\end{proof}

For $x,y\in G$ and $f\colon G\to\Z_p(1)$ a continuous crossed homomorphism, one has the formula
\begin{equation}\label{4eq:c on commutators}
 f\left([x,y]\right)=(\theta(x)-1)f(y)-(\theta(y)-1)f(x).
\end{equation}
Let \begin{equation}\label{4eq:presentation}
 \xymatrix{ 1\ar[r] & R\ar[r] & F\ar[r] & G\ar[r] & 1 }
\end{equation}
be a minimal presentation
for a finitely generated cyclo-oriented pro-$p$ group $(G,\theta)$.
Since $R\leq\Phi(G)$, we may extend $\theta$ to an orientation $\tilde{\theta}\colon F\rightarrow \Z_p^\times$
such that $\theta\circ\pi=\tilde\theta$.
Note that $(F,\tilde\theta)$ is again a cyclo-oriented pro-$p$ group by Remark~\ref{2rem:cyclo-oriented free gps}.
Then, one has the following.

\begin{fact}\label{4fact:crossed homomorphisms and relations}
\begin{itemize}
 \item[(i)] A continuous crossed homomorphism $f\colon G\to\Z_p(1)$ is uniquely determined by the values
 of $f$ on a minimal set of generators of $G$.
In particular, we may arbitrarily prescribe such values (cf. \cite[Prop.~6~(3)]{labute:classification}).
 \item[(ii)] For a crossed homomorphism $f\colon G\to\Z_p(1)$, let $\tilde f\colon F\to\Z_p(1)$ be the lift of $f$
induced by the values of $f$ on a minimal set of generators of $G$, i.e., $\tilde f=f\circ\pi$.
Then one has $\tilde f|_R\equiv\mathbf{0}$.
 \item[(iii)] If a crossed homomorphism $h\colon F\to\Z_p(1)$ is zero on a set of defining relations of $R$,
then $h|_R\equiv\mathbf{0}$.
Indeed, if $\rho\in R$ and $h(\rho)=0$, \eqref{4eq:c on commutators} implies that
\begin{eqnarray*}
 h\left({^x\rho}\right)&=& h\left([x,\rho]\rho^{-1}\right)\\ &=&
h([x,\rho])+\theta([x,\rho])h(\rho^{-1}) \\ &=&(\theta(x)-1)h(\rho)-(\theta(\rho)-1)h(x)+1\cdot 0=0,
\end{eqnarray*}
as $\theta(\rho)=1$.
\end{itemize}
\end{fact}

For every crossed homomorphism $f\colon G\to \Z_p(1)$, the restriction
$f|_{\kernel(\theta)}\colon\kernel(\theta)\rightarrow\Z_p(1)$ is a morphism of pro-$p$ groups.
In particular, if the orientation $\theta$ is trivial, then $\Z_p(1)$ is a trivial $\ZpG$-module, and 
\[H^1(G,\Z_p(1))\simeq\Hom(G,\Z_p).\]
Moreover, recall that from Corollary~\ref{2cor:kernelab torfree} one has that the abelianization
of $\kernel(\theta)$ is torsion-free.
Thus one has the isomorphism of abelian pro-$p$ groups
\[\frac{\kernel(\theta)}{[\kernel(\theta),\kernel(\theta)]}
\simeq\Hom(\kernel(\theta),\Z_p)\simeq H^1(\kernel(\theta),\Z_p(1)).\]

\begin{rem}\label{4rem:Gab torsionfree thetatrivial}
Using Proposition~\ref{4prop:H1Zp1} and Fact~\ref{4fact:crossed homomorphisms and relations} allows us to prove
Corollary~\ref{2cor:kernelab torfree} in a more ``direct'' way.
Indeed, let $(G,\theta)$ be a cyclo-oriented pro-$p$ group with $\theta$ trivial.
Let \ref{4eq:presentation} be a minimal presentation for $G$, and assume for contradiction that there is a relation
$\rho\in R$ and a generator $x\in F$ such that $\rho\equiv x^{p^m}\bmod[F,F]$ for some finite $m$.
With an abuse of notation, consider $x$ as element of $G$, and let $\chi\in H^1(G,\F_p)$ be the dual of $x$.
Also, let the crossed homomorphism $\tilde f\colon F\to \Z_p(1)$ be a ``lift'' of $\chi$, i.e.,
the cohomology class of $\tilde f$ is equivalent to $\chi$ modulo $p$. 
By assumption, $\Z_p(1)$ is a trivial $\ZpG$-module, and thus also a trivial $\Z_p\dbl F\dbr$-module.
Hence $\tilde f$ is an homomorphism, and 
\[0=\tilde f(\rho)=\tilde f\left(x^{p^m}\right)=p^m \tilde f(x)=p^m,\]
a contradiction.
Thus, the abelianization $G/[G,G]$ is torsion-free.
\end{rem}

Assume now that $\image(\theta)\simeq\Z_p$.
The orientation $\theta$ induces an action of $G$ on $H^1(\kernel(\theta),\Z_p(1))$, given by
\begin{equation}\label{4eq:action H1}
(g.f)(z)=\theta(g)\cdot f\left(g^{-1}z\right),\quad g\in G,z\in\kernel(\theta),
\end{equation}
for $f\in H^1(\kernel(\theta),\Z_p(1))$.
Note that every $g\in G$ can be written as $g=x^{\lambda}z$, with $z\in\kernel(\theta)$ and $x\in G$ such that 
$\theta(x)$ generates $\image(\theta)$.
Thus, if $f\colon G\to \Z_p(1)$ is a continuous crossed homomorphism such that $f(x^\lambda)=0$
for every $\lambda\in\Z_p$ (i.e., we may consider $f$ as element of $\Hom(\kernel(\theta),\Z_p)$),
then Fact~\ref{1fact:elementary cohomology} implies that 
\[f\left(x^\lambda z\right)=\theta\left(x^\lambda\right)\cdot f(z) \ \Longrightarrow \ 
\theta\left(x^{-\lambda}\right)\cdot f\left(x^\lambda z\right)=f(z) \]
for every $\lambda\in\Z_p$ and $z\in\kernel(\theta)$, namely, by \eqref{4eq:action H1} $G$ acts trivially on $f$.
Therefore, $H^1(\kernel(\theta),\Z_p(1))^G$ embeds in $H^1(G,\Z_p(1))$. 

Moreover, for every crossed homomorphism $f$ and for every $x,y\in G$, one has \[f(xy)\equiv f(x)+f(y)\mod q,\]
with $q$ a $p$-th power such that $\image(\theta)=1+q\Z_p$.
Also, by Fact~\ref{1fact:elementary cohomology} every $\lambda\in\Z_p(1)$ induces a crossed homomorphism
\[\lambda\colon G\longrightarrow\Z_p(1), \quad \lambda(x)=(\theta(x)-1)\lambda\]
which is a 1-coboundary.
Note that $\lambda|_{\kernel(\theta)}\equiv\mathbf{0}$, so that $\lambda$ factors
through the quotient $G/\kernel(\theta)$, and the image of $\lambda$ is $q\lambda\Z_p$.
Therefore, by Lemma~\ref{2lem:torsion H1}, the reduction modulo $q$ induces the isomorphism
\begin{equation}\label{4eq:torsion}
 \xymatrix{ \tor_{\Z_p}\left(H^1(G,\Z_p(1))\right)\ar[r]^-{\sim} & H^1(G/\kernel(\theta),\Z/q\Z).}
\end{equation}

Moreover, the five term exact sequence induced by the quotient $G/\kernel(\theta)$ (cf. \cite[Prop.~1.6.7]{nsw:cohm})
induces the exact sequence in cohomology
\[ \xymatrix@C=.5truecm{ 0\ar[r] & H^1(G/\kernel(\theta),\Z_p(1))\ar[r] & H^1(G,\Z_p(1))\ar[r] & 
H^1(\kernel(\theta),\Z_p(1))^G \ar`r[d]`[l] `[dlll] `[dll] [dll] \\ 
 & H^2(G/\kernel(\theta),\Z_p(1))\ar[r] & \cdots &} \]
which becomes
\[ \xymatrix@C=.4truecm{ 0\ar[r] & H^1(G/\kernel(\theta),\Z/q\Z)\ar[r] & H^1(G,\Z_p(1))\ar[r] & H^1(\kernel(\theta),\Z_p)^G \ar[r] & 0 } \]
as $\ccd(G/\kernel(\theta))=1$.
Consequently, one has the following.

\begin{prop}\label{4prop:H1Zp1}
 Let $(G,\theta)$ be a cyclo-oriented pro-$p$ group.
Then one has the isomorphism of $\Z_p$-modules
\[H^1(G,\Z_p(1))\simeq H^1(G/\kernel(\theta),\Z/q\Z)\oplus H(\kernel(\theta),\Z_p)^G.\]
\end{prop}

\begin{exs}
\begin{itemize}
 \item[(a)] Let $(G,\theta)$ be a finitely generated $\theta$-abelian pro-$p$ group with $d(G)=d$ and a presentation
\eqref{2eq:presentation theta-abelian}, with $\mathcal{I}={1,\ldots,d-1}$, and set $\theta(x_0)=1+p$.
Then $H^1(G,\Z_p(1))$ is the $\Z_p$-module generated by the classes of the continuous crossed homomorphisms
$f_i\colon G\to\Z_p(1)$, with $f_i(x_j)=\delta_{ij}$ for $i,j=0,\ldots,d-1$.
In particular, \[H^1(G,\Z_p(1))\simeq\F_p\oplus\Z_p^{d-1}\] as $\Z_p$-modules, and $\text{ord}(f_0)=p$, as
$pf_0$ is the crossed homomorphism induced by $1\in\Z_p(1)$.
Note that for every $i\geq1$ one has
\[(x_0.f_i)(x_i)=\theta(x_0)\cdot f\left(x_i^{\theta(x_0^{-1})}x_0^{-1}\right)=\theta(x_0)\theta(x_0)^{-1}f_i(x_i)+0\]
i.e., the action of $G$ fixes $\Hom(\Zen_\theta(G),\Z_p(1))$.
 \item[(b)] Let $G$ be as above, but equip it with the trivial orientation $\theta'\colon G\to\Z_p^\times$,
$\theta'\equiv\mathbf{0}$.
Then by Corollary~\ref{2cor:kernelab torfree} $(G,\theta')$ is not cyclo-oriented, as the abelianization of
$G=\kernel(\theta')$ is not torsion-free.
Moreover, $H^1(G,\Z_p)\simeq\Z_p$, and such group does not projects onto $H^1(G,\F_p)$ -- in particular,
it is not possible to lift the dual of $x_i$ from $H^1(G,\F_p)$ to $\Hom(G,\Z_p)$ for every $i\geq1$,
as one would have $p=f_i\left(x_i^p\right)=f_i([x_0,x_i])=0$.
\end{itemize}
\end{exs}

\section{Free products}

Recall from Subsection~2.5.2 the definition of free product in the category of pro-$p$ groups.
The class of Bloch-Kato pro-$p$ groups is closed under free pro-$p$ products
(cf. \cite[Theorem~5.2]{claudio:BK}).
We want now to extend such closure to the class of cyclo-oriented pro-$p$ groups.

Then for two oriented pro-$p$ groups $(G_1,\theta_1)$ and $(G_2,\theta_2)$,
the orientations $\theta_1$ and $\theta_2$
induce a new orientation on the free product, i.e., one has the commutative diagram
\begin{equation}\label{1eq:free product}
 \xymatrix{ G_i\ar[r]^-{\iota_i}\ar[d]_{\theta_i} & G_1\ast_{\hat p} G_2 \ar@{-->}[dl]^{\tilde\theta}\\
 \Z_p^\times}
\end{equation}
with $i=1,2$.
Thus, we may define the free pro-$p$ product for oriented pro-$p$ groups, and one has the following.

\begin{thm}\label{4thm:cyclotomic free product}
 Let $(G_1,\theta_1)$ and $(G_2,\theta_2)$ be two cyclo-oriented pro-$p$ groups,
and assume further that $\image(\theta_i)$ has non non-trivial torsion for both $i$.
Then the orientation $\tilde\theta$ defined as in \eqref{1eq:free product} is a cyclotomic orientation for
the free product $G_1\ast_{\hat p}G_2$.
\end{thm}

\begin{proof}
Set $G=G_1\ast_{\hat p}G_2$.
First, as stated above, the group $G$ is a Bloch-Kato pro-$p$ group.
Thus, it is enough to show that $H^2(C,\Z_p(1))$ is torsion-free for every $C$.

By \cite[Theorem~4.1.5]{nsw:cohm}, the group $H^1(G,\F_p)$ decomposes as
$H^1(G_1,\F_p)\oplus H^1(G_2,\F_p)$.
Thus, the inclusions $\iota_1,\iota_2$ induce the commutative diagram
\begin{equation}\label{4eq:commutative diagram freeprod}
 \xymatrix{ H^1(G,\Z_p(1))\ar[rr]\ar[d]^{\res_{G,G_i}^1} && H^1(G_1,\F_p)\oplus H^1(G_2,\F_p)\ar@{->>}[d] \\
 H^1(G_i,\Z_p(1))\ar@{->>}[rr] && H^1(G_i,\F_p) }
\end{equation}
for $i=1,2$, where the lower arrow is surjective by hypothesis and by Lemma~\ref{2lem:H2 torfree iff H1 projects}.
In particular, the groups $G_1$ and $G_2$ are properly embedded in $G$ (cf. \cite[Fact~4.5]{claudio:BK}).

Moreover, by \eqref{1eq:free product}, one has that $\tilde\theta|_{G_i}=\theta_i$ for $i=1,2$.
Thus, the restrictions $H^1(\kernel(\tilde\theta),\Z_p)\to H^1(\kernel(\theta_1),\Z_p)$ are surjective.
Also, the monomorphisms $\bar\iota_i\colon G_i/\kernel(\theta_i)\to G/\kernel(\tilde\theta)$ induce 
epimorphisms $H^1(G/\kernel(\tilde\theta),\Z/\tilde q\Z)\twoheadrightarrow H^1(G_i/\kernel(\theta_i),\Z/q_i\Z)$,
with $\image(\theta_i)=1+q_i\Z_p$ and $\image(\tilde\theta)=1+\tilde q\Z_p$,
if the respective orientation is not trivial.
Therefore, Proposition~\ref{4prop:H1Zp1} implies that
\begin{equation}\label{1eq:restriction free product}
 \res_{G,G_i}^1\colon H^1(G,\Z_p(1))\longrightarrow H^1(G_i,\Z_p(1))
\end{equation}
is surjective for $i=1,2$. 
Thus the map $H^1(G,\Z_p(1))\to H^1(G_i,\F_p)$ is surjective for each $i$, and therefore
also the upper arrow of \eqref{4eq:commutative diagram freeprod} is surjective,
and by Lemma~\ref{2lem:H2 torfree iff H1 projects} $H^2(G,\Z_p(1))$ is torsion-free.

Now let $C$ be a closed subgroup of $G$.
Then by the proof of \cite[Theorem~5.2]{claudio:BK}, one has $C=F\ast_{\hat p}C_1\ast_{\hat p}C_2$,
with $F$ a free pro-$p$ group and for every $i$,
\[C_i=\coprod_{r\in\mathcal{R}_i}(C\cap G_i^r)\]
where $\mathcal{R}_i$ is a set of representatives of the coset space $C\backslash G/G_i$. \footnote{
This construction is called the free pro-$p$ product of a {\it sheaf} of pro-$p$ groups,
see \cite[\S~5.1]{claudio:BK} for further details and references.}
Note that every $C\cap G_i^r$ is a subgroup of a group isomorphic to $G_i$, thus it is again cyclo-oriented.
Thus, for every $r'\in\mathcal{R}_i$ one has the commutative diagram
\[ \xymatrix{ H^1(C_i,\Z_p(1))\ar[rr]\ar@{->>}[d]^{\res_{C_i,C\cap G_i^{r'}}} && \bigoplus_{\mathcal{R}_i}H^1(C\cap G_i^r,\F_p)\ar@{->>}[d] \\
 H^1(C\cap G_i^{r'},\Z_p(1))\ar@{->>}[rr] && H^1(C\cap G_i^{r'},\F_p) } \]
with the right-hand side vertical arrow is surjective by \cite[Eq.~(5.1)]{claudio:BK}.
Hence, also $C_i$ (and consequently $C$) is cyclo-oriented. 
\end{proof}

\subsection{Free-by-Demushkin groups}

In the short paper \cite{wurfel:remark}, T.~W\"urfel proves the following theorem.

\begin{thm}\label{4thm:wurfel}
Let $K$ be a field such that $\chr(K)\neq p$ and $\mu_{p^\infty}\subseteq K$, 
and suppose that the maximal pro-$p$ Galois group $G_K(p)$ is finitely generated with one defining relation.
Then there exists a normal closed subgroup $N$ of $G$ which is pro-$p$ free such that the quotient $G/N$
is a Demushkin group, and the inflation map
\begin{equation}\label{4eq:wurfel inflation}
 \xymatrix{ H^2(C/N,\Z/p^n\Z)\ar[rr]^-{\inf_{C,N}^2} && H^2(C,\Z/p^n\Z) }
\end{equation}
is an isomorphism for every closed subgroup $C$ of $G$ containing $N$.
\end{thm}

He concludes with the following question: is it always true that such $G_K(p)$ is a {\it free-by-Demushkin group},
i.e., $G$ is decomposable as free pro-$p$ product $G=G/N\ast_{\hat p} F$,
with $F$ a finitely generated free pro-$p$ group, if the map \eqref{4eq:wurfel inflation} is an isomorphism?

In \cite{kocloupavel:freeby}, D.~Kochloukova and P.~Zalesski\u{i} construct a pro-$p$ group
as described in Theorem~\ref{4thm:wurfel} such that \eqref{4eq:wurfel inflation} is an isomorphism,
but which is not a free-by-Demushkin group.
Also they observe that such group is not realizable as maximal pro-$p$ Galois group of any field.
Such group is defined by the presentation
\begin{equation}\label{4eq:kochloupavel presentation}
 G=\left\langle x,y,z\left|z^p[x,y]=1\right.\right\rangle.
\end{equation}
Let $N$ be the normal subgroup of $G$ generated by $z$, and set $D=G/N$.
Then $N$ is an (infinitely generated) free pro-$p$ group, and $D$ is a 2-generated abelian group.
In particular, $D$ is a 2-generated Demushkin group with $q_D=0$.

We want to show that no orientations $\theta\colon G\to\Z_p^\times$ make $(G,\theta)$
a cyclo-oriented pro-$p$ group.
Assume for contradiction that $G$ has a cyclotomic orientation $\theta$.
Since the abelianization $G/[G,G]$ is isomorphic to $\Z_p^2\oplus\Z/p\Z$,
by Remark~\ref{4rem:Gab torsionfree thetatrivial} the orientation $\theta$ can not be trivial.
Also, 
\[1=\theta\left(z^p[x,y]\right)=\theta\left(z^p\right)\theta([x,y])=\theta(z)^p,\]
thus $\theta(z)=1$. 
Let $f\colon G\to\Z_p(1)$ be the continuous crossed homomorphism representing a lift of
the dual of $z$ in $H^1(G,\F_p)$, i.e., $f(z)=1$ and $f(x),f(y)=0$
(such $f$ exists by Proposition~\ref{4prop:H1Zp1} and Fact~\ref{4fact:crossed homomorphisms and relations},
as we are assuming that $\theta$ is cyclotomic).
Then, using \eqref{4eq:c on commutators}, one obtains
\begin{eqnarray*}
 f\left(z^p[x,y]\right)&=& f(z^p)+\theta(z)f([x,y]) \\
 &=& pf(z)+(\theta(x)-1)f(y)-(\theta(y)-1)f(x) \\ &=&p.
\end{eqnarray*}
But $f(z^p[x,y])=f(1)=0$, a contradiction.

\section{The cyclotomic fibre product}

The other basic operation involved in the ETC is the {\it cyclotomic fibre product} induced by
the arithmetic orientation of a maximal pro-$p$ Galois group.

\begin{defi}\label{4defi:cyclotomic fibre product}
Let $(G,\theta)$ be an oriented pro-$p$ group, and let $Z$ be a pro-$p$-cyclic group, i.e., $Z\simeq\Z_p$.
The {\bf cyclotomic fibre product} $\tilde G=G\ltimes_\theta Z$ is the fibre product induced by
the split short exact sequence
\begin{equation}\label{4eq:ses cyclotomic fibreprod}
 \xymatrix{ 1\ar[r] & Z\ar[r] & \tilde G\ar[r] & G\ar[r] & 1 }
\end{equation}
with $Z\simeq\Z_p(1)$ as $\ZpG$-module.
The group $\tilde G$ comes equipped with an orientation $\tilde\theta\colon\tilde G\to\Z_p^\times$
such that $\tilde\theta|_G=\theta$ and $\tilde\theta|_Z\equiv\mathbf{0}$.
\end{defi}

The class of maximal pro-$p$ Galois groups is closed under iterated cyclotomic fibre products,
as shown by the following.

\begin{prop}\label{4prop:Galois fibre prod}
 Let $K$ be a field containing a primitive $p$-th root of unity with cyclotomic orientation
$\theta\colon G_K(p)\to\Z_p^\times$, and let $A$ be an abelian free pro-$p$ group, i.e., $A\simeq\Z_p^m$,
for some $m\geq1$.
Then the cyclotomic fibre product $G_K(p)\ltimes_\theta A$ is realizable as maximal pro-$p$ Galois group $G_L(p)$
for some field $L\supset K$.
\end{prop}

The above result is well known, and it can be deduced, e.g., from \cite[\S~1, p.~512]{engler:etc}.
Also, it is clear that the cyclotomic fibre product ``enlarges'' the $\theta$-centre of an oriented pro-$p$ group.
In particular, one has the following elementary fact.

\begin{fact}\label{4fact:fribre product elementary}
If $(G,\theta)$ is an oriented pro-$p$ group and $Z$ is as above,
then for $\tilde G=G\ltimes_\theta Z$ one has $\Zen_{\tilde\theta}(\tilde G)=\Zen_\theta(G)\times Z$.
\end{fact}

As Theorem~\ref{4thm:Zcentresplit} will show, every $\theta$-centre of oriented Bloch-Kato pro-$p$ groups
can be obtained as the result of a cyclotomic fibre products.

\subsection{The $\theta$-centre}

Let $C\leq G$ be a closed subgroup of a pro-$p$ group $G$.
Then $C$ is called {\it isolated} if for every element $g\in G$
there exists $m\geq 1$ such that $g^{p^m}\in C$ if, and only if, $g\in C$.
It follows that a closed normal subgroup $N\vartriangleleft G$ is isolated if, and only if, $G/N$ is torsion free.

\begin{prop}\label{4prop:isoZtheta}
Let $(G,\theta)$ be an oriented torsion free Bloch-Kato pro-$p$ group. 
Then $\Zen_\theta(G)$ is an isolated subgroup of $G$.
\end{prop}

\begin{proof}
Assume first that $\theta$ is trivial.
Then $\Zen_\theta(G)$ is the centre $\Zen(G)$.
Suppose there exists $x\in G\smallsetminus \Zen(G)$ and $m\geq 1$ such that $x^{p^m}\in\Zen(G)$.
Without loss of generality, we may assume that $m=1$, i.e., $x^p=z\in\Zen(G)$.
As $G$ is torsion free, one has $z\neq1$.
Since $x\not\in\Zen(G)$, there exists an element $g\in G$ such that $gxg^{-1}\neq x$.
Set $C$ to be the closed subgroup of $G$ generated by $x$ and $g$.
As $C$ is a Bloch-Kato pro-$p$ group and $z\in\Zen(C)$,
$C$ must be abelian by Theorem~\ref{3thm:equivalence theta-abelian}.
Hence $gxg^{-1}=x$, a contradiction, and this yields the claim.

Assume now that $\image(\theta)\simeq\Z_p$, and let $p$ be odd.
As $\Zen_\theta(G)$ is a closed normal subgroup of $G$, it suffices to show
that $G/\Zen_\theta(G)$ is torsion free.
Since $\Zen_\theta(G)=\Zen(\kernel(\theta))$, the $\theta$-centre of $G$ is isolated in $\kernel(\theta)$.
Thus $\kernel(\theta)/\Zen_\theta(G)$ is torsion free.
As $G/\kernel(\theta)\simeq\image(\theta)$ is torsion free, this yields the claim in this case.

Now let $p=2$, and put $K=\kernel(\theta)$.
Since $\image(\theta)\simeq\Z_2$, $\Zen(K)$ is isolated in $K$.
Moreover, $\Zen(K)$ is a profinite left $\Z_2\dbl G/K\dbr$-module which is a torsion free abelian pro-$2$ group.
As
\begin{equation}\label{eq:tZ2}
\Zen_\theta(G)=\left\{z\in\Zen(K)\:\left|\: (g-\theta(g)\iid_K)z=0\text{ for all }g\in G\right.\right\},
\end{equation}
$\Zen_\theta(G)$ is an isolated subgroup of $\Zen(K)$.
Hence $\Zen_\theta(G)$ is an isolated subgroup of $K$.

Suppose $\Zen_\theta(G)$ is not isolated in $G$.
Then there exists $ x \in G\smallsetminus\Zen_\theta(G)$ such that
$ x ^2=z\in\Zen_\theta(G)$ and $z\not=1$. As $ x \not\in\Zen_\theta(G)$,
the previously mentioned remark shows that $ x \not\in K$,
i.e., $\theta(x)\not=1$. Hence
$z=xzx^{-1}=z^{\theta(x)}$.
As every element in $\Z_2^\times$ acts fixed point freely on the closed subgroup generated by $z$,
one concludes that $\theta(x)=1$, a contradiction, and this yields the claim.
\end{proof}

The following fact is straightforward. 

\begin{fact}\label{4fact:sur}
Let $\phi_\bullet\colon A_\bullet\to B_\bullet$ be a morphism
of positively graded $\F_p$-algebras such that
$\phi_1\colon A_1\to B_1$ is surjective, and that $B_\bullet$ is generated
by $B_1$. Then $\phi_n\colon A_n\to B_n$ is surjective for all $n\geq 0$.
\end{fact}

\begin{prop}\label{4prop:norm}
Let $G$ be a Bloch-Kato pro-$p$ group of finite cohomological dimension,
and let $Z$ be a closed normal subgroup of $G$ isomorphic to $\Z_p$ such that $G/Z$ is torsion-free.
Then $Z\not\subseteq\Phi(G)$.
\end{prop}

\begin{proof}
Let $d=\ccd(G)$.
As $Z\simeq\Z_p$, one has $\ccd(Z)=1$ and $H^1(U,\F_p)\simeq \F_p$ for all open subgroups $U$ of $Z$.
Thus, the {\it virtual} cohomological dimension of $G/Z$ is $\text{vcd}(G/Z)=d-1$
(cf. \cite[Theorem~3.3.9]{nsw:cohm}).
Since $G/Z$ is torsion-free one has $\ccd(G/Z)<\infty$,
hence \cite[Prop.~3.3.5]{nsw:cohm} implies that $\ccd(G/Z)=d-1$.

Suppose that $Z\subseteq \Phi(G)$.
Then $\inf_{G,Z}^1\colon H^1(G/Z,\F_p)\to H^1(G,\F_p)$ is an isomorphism.
Since $G$ is a Bloch-Kato pro-$p$ group, Fact~\ref{4fact:sur} implies that
\begin{equation}\label{4eq:norm}
\text{inf}_{G,Z}^n\colon H^n(G/Z,\F_p)\longrightarrow H^n(G,\F_p)
\end{equation} 
is surjective for all $n\geq 0$.
In particular, if $(E^{s,t}_r,d_r)$ denotes the Hochschild-Serre spectral sequence associated to
the extension of pro-$p$ groups $Z\to G\to G/Z$ with $\F_p$-coefficients,
then $E^{s,t}_\infty$ must be concentrated on the bottom row.
However, $ H^1(Z,\F_p)$ is a trivial $G/Z$-module isomorphic to $\F_p$, and thus 
\begin{equation}\label{eq:e2}
E^{d-1,1}_2=H^{d-1}\left(G/Z,H^1(Z,\F_p)\right)\neq0.
\end{equation}
For $r\geq 2$ the pages of the spectral sequence $(E_r^{s,t},d_r)$ are concentrated in the rectangle 
$\{0,\ldots,d-1\}\times\{0,1\}$. Thus the place $(d-1,1)$ can never be hit by a differential non-trivially,
and $d_r^{d-1,1}=0$ for all $r\geq 2$.
Hence $E^{d-1,1}_\infty=E^{d-1,1}_2\not=0$. A contradiction, and this yields the claim.
\end{proof}

Proposition~\ref{4prop:norm} has the following consequence.

\begin{prop}\label{4prop:split}
Let $G$ be a finitely generated Bloch-Kato pro-$p$ group, and let $Z$ be a closed normal subgroup of $G$
isomorphic to $\Z_p$ such that $G/Z$ is torsion-free.
Then there exists a $Z$-complement in $G$, i.e., the extension of pro-$p$ groups
$1\to Z\to G\to G/Z\to 1$ splits.
\end{prop}

\begin{proof}
By Proposition~\ref{4prop:norm} one has $Z\not\subseteq\Phi(G)$.
Hence there exists an open subgroup $C_1$ of index $p$ such that $C_1Z=G$ and $Z_1=C_1\cap Z=Z^p$.
Moreover, $Z_1$ is a closed normal subgroup in $C_1$ such that $C_1/Z$ is torsion-free and $Z_1\simeq \Z_p$.
Thus again by Proposition~\ref{4prop:norm} one has $Z_1\not\subseteq\Phi(C_1)$.
Repeating this process one finds open subgroup $C_m$ of $G$ of index $p^k$ such that
\[C_mZ=G\quad \text{and} \quad Z_m=C_m\cap Z=Z^{p^m}.\] 
Hence $\bigcap_{m\geq 1} C_m$ is a $Z$-complement in $G$.
\end{proof}

\begin{thm}\label{4thm:Zcentresplit}
Let $(G,\theta)$ be a finitely generated cyclo-oriented pro-$p$ group.
Then there exists a closed subgroup $G_\circ$ of $G$ which is a complement for $\Zen_\theta(G)$, i.e., 
the short exact sequence
\begin{equation}\label{4eq:split ext thetacentre}
\xymatrix{ 1\ar[r] & \Zen_\theta(G)\ar[r] & G \ar[r] & G/\Zen_\theta(G)\ar[r] & 1 }
\end{equation}
splits, with $G_\circ\simeq G/\Zen_\theta(G)$, so that $G=G_\circ\ltimes_\theta\Zen_\theta(G)$.
\end{thm}

We want now to compute the Zassenhaus filtration for a cyclo-oriented pro-$p$ group with non-trivial $\theta$-centre,
as we did for $\theta$-abelian groups in Section~3.5.

Let $(G,\theta)$ be as in Theorem~\ref{4thm:Zcentresplit}, with $G=G_\circ\ltimes_\theta\Zen_\theta(G)$
and $\Zen_\theta(G)$ non-trivial.
Also let $p^k$ be the $p$-power such that $\image(\theta)=1+p^k\Z_p$. 
Then it is easy to see that for each $i\geq1$ one has
\[\gamma_i(G)=\gamma_i(G_\circ)\ltimes_{\theta_{\gamma_i}}\Zen_\theta(G)^{p^{k(i-1)}},\]
with $\theta_{\gamma_i}=\theta|_{\gamma_i(G_\circ)}$.
Therefore, an argument as in \eqref{3eq:Zassenhaus filtration for thetabelian} shows that
\begin{equation}\label{4eq:Zassenhaus filtration for fibreproduct}
  D_n(G)=D_n(G_\circ)\ltimes_{\theta_n}\Zen_\theta(G)^{p^\ell}\quad \text{for } p^{\ell-1}<n\leq p^\ell,
\end{equation}
with $\theta_n=\theta|_{D_n(G_\circ)}$.
Consequently, the quotients induced by the Zassenhaus filtration are
\begin{equation}\label{4eq:quotients_Zassenhaus}
 L_i(G)\simeq\left\{\begin{array}{cc} L_i(G_\circ)\oplus\Zen_\theta(G)/\Zen_\theta(G)^p & \text{for $i$ a $p^{th}$-power} \\
L_i(G_\circ) & \text{otherwise} \end{array}\right.
\end{equation}
as vector spaces over $\F_p$.

\begin{prop}\label{4prop:Zassenhaus algebra fibreprod}
 Let $(G,\theta)$ be a cyclo-oriented pro-$p$ group, and assume that $\image(\theta)\simeq\Z_2$ for $p=2$.
Then the restricted Lie algebra $L_\bullet(G)$ induced by the Zassenhaus filtration decomposes as
\[L_\bullet(G)=L_\bullet(G_\circ)\oplus A_\bullet\]
with $G_\circ\leq G$ as in Theorem~\ref{4thm:Zcentresplit},
and $A_\bullet$ an abelian restricted Lie algebra such that $\dim(A_1)=d(\Zen_\theta(G))=d$.
Also, one has
\begin{equation}\label{4eq:lie algebra fibreprod}
 \grad_\bullet(G)\simeq\mathcal{U}_p(L_\bullet(G))\simeq\grad_\bullet(G_\circ)\otimes^1\F_p[Z_1,\ldots,Z_d],
\end{equation}
where $\otimes^1$ denotes the symmetric tensor product (defined in Subsection~2.1.1),
and the grading in the polynomial part is induced by the degrees of the monomials.
\end{prop}

\begin{proof}
 The first statement follows from \eqref{4eq:quotients_Zassenhaus}.
As for the second one, since $(L_\bullet(G_\circ),A_\bullet)=0$, one has that $\psi_L(L_\bullet(G_\circ))$
and $\psi_L(A_\bullet)$ commute in $\mathcal{U}_p(L_\bullet(G))$, and \eqref{4eq:lie algebra fibreprod} holds.
\end{proof}

\subsection{Cohomology of cyclotomic fibre products}

The next result is the cohomological ``translation'' of Theorem~\ref{4thm:Zcentresplit}, namely,
it shows how the presence of the $\theta$-centre decomposes the $\F_p$-cohomology ring.

\begin{thm}\label{4thm:cohomology fibre product}
Let $(G,\theta)$ be a finitely generated Bloch-Kato pro-$p$ group with orientation
such that $\Zen_\theta(G)$ is non-trivial, and let $G_\circ$ be the complement of the $\theta$-centre. 
Then one has the isomorphism of $\F_p$-vector spaces
\begin{equation}\label{4eq:H2 for thetacentre}
 H^2(G,\F_p)\simeq H^2(G_\circ,\F_p)\oplus(V\wedge W)\oplus(W\wedge W),
\end{equation}
with $V=H^1(G_\circ,\F_p)$ and $W=H^1(\Zen_\theta(G),\F_p)$.
\end{thm}

\begin{proof}
We shall prove the result by induction on the minimal number of generators of $\Zen_\theta(G)$.

Assume first that $\Zen_\theta(G)\simeq \Z_p$, and let $z\in \Zen_\theta(G)$ be a generator.
Moreover, let $\psi\in W$ be the dual of $z$, and pick a $\F_p$-basis $\{\chi_i,i\in\mathcal{I}\}$ of $V$,
so that $\{\psi,\chi_i\}$ is a basis for $H^1(G,\F_p)$.

Let $\chi\in V$ be any non-trivial $\F_p$-linear combination of the $\chi_i$'s,
and let $x\in G_\circ\smallsetminus \Phi(G)$ be such that $\chi$ is dual to $x$.
Set $C\leq G$ be the closed subgroup generated by $x$ and $z$.
By construction, $C$ is a $2$-generated $\theta|_C$-abelian pro-$p$ group,
and by Theorem~\ref{3thm:equivalence theta-abelian}  the cohomology ring $H^\bullet(C,\F_p)$ is 
the exterior $\F_p$-algebra generated by $\res_{G,C}^1(\chi)$ and $\res_{G,C}^1(\psi)$.
In particular, one has
\[\res_{G,C}^2(\chi\cup\psi)=\res_{G,C}^1(\chi)\cup\res_{G,C}^1(\psi)\neq0,\]
hence $\chi\cup\psi\neq0$ in $H^2(G,\F_p)$.
Therefore, all the cup products $\chi_i\cup\psi$, with $i\in\mathcal{I}$, are $\F_p$-linearly independent,
and $V\wedge W$ is a subspace of $H^2(G,\F_p)$.

Moreover, the short exact sequence of pro-$p$ groups \eqref{4eq:split ext thetacentre}
induces the five term exact sequence
\begin{equation}
 \xymatrix@C=1.2truecm{0\ar[r] & H^1(G_\circ,\F_p)\ar[r]^-{\inf_{G,G_\circ}^1} & V\oplus W\ar[r]^-{\res_{G,\Zen_\theta(G)}^1} &
W^G \ar`r[d]`[l]^-{\text{tg}_{G,\Zen_\theta(G)}} `[dlll] `[dll] [dll] \\
  & H^2(G_\circ,\F_p)\ar[r]^-{\inf_{G,G_\circ}^2} & H^2(G,\F_p). & }
\end{equation}
Since $W^G=H^1(\Zen_\theta(G),\F_p)$, one has that $\res_{G,\Zen_\theta(G)}^1$ is surjective,
for the $\theta$-centre is properly embedded in $G$ \cite[Fact 4.5]{claudio:BK}.
Thus, the transgression map $\text{tg}_{G,\Zen_\theta(G)}$ is trivial, and $\inf_{G,G_\circ}^2$ is injective.
In particular, the image of $\inf_{G,G_\circ}^2$ is isomorphic to $H^2(G_\circ,\F_p)$.

Since the intersection between $V\wedge W$ and $\inf_{G,G_\circ}^2(H^2(G_\circ,\F_p))$ is trivial, one has 
\begin{eqnarray*}
 r(G)=\dim_{\F_p}H^2(G,\F_p) &\geq& 
\dim_{\F_p}(V\wedge W)+\dim_{\F_p}H^2(G_\circ,\F_p)\\
&=& \dim_{\F_p}(V)+r(G_\circ),
\end{eqnarray*}
by Definition~\ref{1defi:pro-p groups}.
By hypothesis, the minimal number of relations of $G$ is exactly $r(G)=r(G_\circ)+\dim_{\F_p}(V)$,
thus $H^2(G,\F_p)$ is generated by $\inf_{G,G_\circ}^2H^2(G_\circ,\F_p)$ and $V\wedge W$.

Assume now that $d(\Zen_\theta(G))>1$, and let $Z_1,Z_2$ be subgroups of the $\theta$-centre
such that $Z_2\simeq\Z_p$ and $\Zen_\theta(G)=Z_1\times Z_2$.
Moreover, set $W_i=H^1(Z_i,\F_p)$, with $i=1,2$, and $\tilde G_\circ=G_\circ\ltimes_\theta Z_1$.
Thus, $H^1(\tilde G_\circ,\F_p)=V\oplus W_1$, and by induction one has
\[H^2(\tilde G_\circ,\F_p)\simeq H^2(G_\circ,\F_p)\oplus(V\wedge W_1)\oplus(W_1\wedge W_1).\]
Also, the above argument shows that 
\[H^2(G,\F_p)\simeq H^2(\tilde G_\circ,\F_p)\oplus ((V\oplus W_1)\wedge W_2),\]
and the claim follows.
\end{proof}

\begin{rem}\label{4rem:graded tensor product}
By Theorem~\ref{4thm:cohomology fibre product}, the $\F_p$-cohomology ring of the cyclo-oriented pro-$p$ group
is the skew-commutative tensor product
\[H^\bullet(G,\F_p)=H^\bullet(G_\circ,\F_p)\otimes^{-1}\left(\bigwedge_{n\geq0}W\right),\]
which can be considered also as a {\it graded tensor product}\footnote{
Let $A_\bullet$ and $B_\bullet$ be two (non-negatively) graded rings.
The {\it tensor product} in the category of graded rings is defined by
\[ A_\bullet\otimes_{\grad}B_\bullet=\coprod_{k\geq0}\left(\bigoplus_{i+j=k}A_i\otimes B_j\right),\]
with the multiplication law $(a\otimes b)(a'\otimes b')=(-1)^{i'j}aa'\otimes bb'$
for $a\in A_i$, $a'\in A_{i'}$, $b\in B_j$ and $b'\in B_{j'}$ (cf. \cite[\S~23.2]{efrat:libro}).}.
In particular, the $\F_p$-cohomology ring of a Bloch-Kato pro-$p$ group can be considered as a
{\it $\kappa$-structure} (cf. \cite[Ch.~23]{efrat:libro}), and $H^\bullet(G,\F_p)$ is the {\it extension}
of $\kappa$-structures $H^\bullet(G_\circ,\F_p)[W]$ \cite[p.~211]{efrat:libro}. 
\end{rem}

\begin{example}\label{4exa:fibreprod}
Let $G_\circ$ be a Demushkin group with cyclotomic orientation $\theta\colon G_\circ\to\Z_p^\times$ such that
$\image(\theta)\simeq\Z_p$, and let $\{\chi_1,\ldots,\chi_d\}$ be a suitable basis for $H^1(G_\circ,\F_p)$,
with $d=d(G_\circ)$.
Also, let $Z$ be isomorphic to $\Z_p(1)$ as $\ZpG$-modules, and let $\psi$ be a generator of $H^1(Z,\F_p)$.
Then for $G=G_\circ\ltimes_\theta Z$, the $\F_p$-vector space $H^2(G,\F_p)$ is generated by the elements
\[\chi_1\cup\chi_2=\ldots=\chi_{d-1}\cup\chi_d \quad\text{and}\quad
\chi_i\cup\psi \ \text{for every }i,\]
and $\dim(H^2(G,\F_p))=d+1$.
\end{example}

Theorem~\ref{4thm:cohomology fibre product} shows that the class of Bloch-Kato pro-$p$ groups is closed
under cyclotomic fibre products.
We want to extend the closure to the class of cyclo-oriented pro-$p$ groups.

\begin{thm}\label{4thm:cyclotomic fibreprod}
Let $(G,\theta)$ be a finitely generated cyclo-oriented pro-$p$ group,
and let $\tilde G=G\ltimes_\theta Z$ be the cyclotomic fibre product
with $Z$ and $\tilde\theta$ as in Definition~\ref{4defi:cyclotomic fibre product}.
Then $(\tilde G,\tilde\theta)$ is again a cyclo-oriented pro-$p$ group.
\end{thm}

\begin{proof}
By Theorem~\ref{4thm:cohomology fibre product}, $\tilde G$ is a Bloch-Kato pro-$p$ group,
so it remains to show that $H^2(\tilde G,\Z_p(1))$ is a torsion-free $\Z_p$-module.
Thus, by Lemma~\ref{2lem:H2 torfree iff H1 projects}, it is enough tho show that the morphism
\begin{equation}\label{4eq:projection cyclotomic fibre product}
 H^1(\tilde G,\Z_p(1))\to H^1(\tilde G,\F_p)
\end{equation}
induced by \eqref{2eq:ses Zp1Zp1Fp} is surjective.

Let $z$ be a generator of $Z$.
Then $H^1(Z,\Z_p(1))$ is the $\Z_p$-module generated by the morphism $f\colon Z\to \Z_p(1)$
defined by $f(z)=1$.
Note that by \eqref{4eq:action H1} $\tilde G$ fixes $f$, as 
\[(g.f)(z)=\theta(g)\cdot f\left(z^{\theta(g^{-1})}g^{-1}\right)=\theta(g)\theta(g)^{-1}\cdot1+0=1\]
for every $g\in G$.
Thus, $H^1(Z,\Z_p(1))^{G}=H^1(Z,\Z_p(1))$.
In particular, this implies that the restriction map
\[\res_{\tilde G,Z}^1\colon H^1(\tilde G,\Z_p(1))\longrightarrow H^1(Z,\Z_p(1))\]
is surjective.

Therefore, by the five term exact sequence induced by \eqref{4eq:ses cyclotomic fibreprod}
with coefficients in $\Z_p(1)$, one has the commutative diagram
 \[\xymatrix@C=.27truecm{ 0\ar[r] & H^1(G,\Z_p(1))\ar[r]^-{\inf_{\tilde G,Z}^1}\ar@{->>}[d] &
H^1(\tilde G,\Z_p(1))\ar[r]^-{\res_{\tilde G,Z}^1}\ar[d] & H^1(Z,\Z_p(1))\ar[r]\ar@{->>}[d] & 0\\
0\ar[r] & H^1(G,\F_p)\ar[r] & H^1(G,\F_p)\oplus H^1(Z,\F_p)\ar[r] & H^1(Z,\F_p)\ar[r] & 0 } \]
where the left-hand side vertical arrow is surjective by hypothesis, and the right-hand side one is surjective
since $f\bmod p$ generates $H^1(Z,\F_p)$.
Hence, by the snake lemma, also \eqref{4eq:projection cyclotomic fibre product} is surjective.

Moreover, every closed subgroup $\tilde C$ of $\tilde G$ decomposes as cyclotomic fibre product
$\tilde C=C\ltimes_{\theta|_C}Z'$, with $C$ a subgroup of $G$ and $Z'$ a subgroup of $Z$,
and by the above argument $H^2(\tilde C,\Z_p(1))$ is torsion-free.
\end{proof}

\subsection{Arithmetic}

Let $K$ be a field.
Recall that a {\bf valuation} of $K$ is a group homomorphism $v$ from $K^\times$ into an ordered abelian group
$(\Gamma,\leq)$ such that for every $\alpha,\beta\in K^\times$ with $x\neq-y$, one has
\[v(x+y)\geq\min\{v(x),v(y)\}.\]
The {\it rank} of the valuation $v$ is the rank of the abelian group $\image(v)\leq\Gamma$.

A valuation $v\colon K^\times\rightarrow\Gamma$ is said to be {\bf $p$-Henselian}
if it extends uniquely to the $p$-closure $K(p)$, i.e., there exists a unique valuation
$w\colon K(p)^\times\to\Gamma$ such that $w|_{K^\times}=v$.
(For more details on valuation theory see \cite{efrat:libro}.)

Assume now that $K$ has characteristic prime to $p$.
We are to show the relation between the structure of the maximal pro-$p$ Galois group $G_K(p)$ of $K$
(and of its $\F_p$-cohomology ring) and the existence of certain valuations of $K$.

\begin{prop}\label{4prop:Zmxlabeliannormal}
Let $(G,\theta)$ be a cyclo-oriented pro-$p$ group,
and assume that $\image(\theta)$ is either pro-$p$-cyclic or trivial.
Then the $\theta$-centre $\Zen_\theta(G)$ is the maximal normal abelian subgroup of $G$
\end{prop}

\begin{proof}
If $G$ is abelian, then $\theta$ is trivial, and $\Zen_\theta(G)=G$.
Thus, assume that $G$ is not abelian, and let $N$ be the maximal normal abelian subgroup of $G$.
Note that $N\leq\kernel(\theta)$, as $(N,\theta|_N)$ is cyclo-oriented.

Fix an element $z\in N$, and choose any $x\in G$.
Let $C\leq G$ be the closed subgroup generated by $x$ and $z$: then either $C\simeq \Z_p$ or $d(C)=2$.
In the former case, $x$ and $z$ commute, and moreover $C\leq\kernel(\theta)$: thus, one has 
\begin{equation}\label{4eq:Zmxlabeliannormal1}
 {^xz}=z=z^{\theta(x)}.
\end{equation}
In the latter case, Theorem~\ref{3thm:equivalence theta-abelian} implies that either $C$ is
a 2-generated free pro-$p$ group, or $C$ is $\theta'$-abelian, for some orientation $\theta'\colon C\to\Z_p^\times$.
Since $N$ is normal, the commutator subgroup $[C,C]$ of $C$ is contained in $N$,
hence $[C,C]$ is abelian, and $C$ is $\theta'$-abelian.
Since $\theta|_C$ is a cyclotomic orientation for $C$, one has $\theta'=\theta|_C$,
and $z\in\Zen_{\theta|_C}(C)$:  thus, one has 
\begin{equation}\label{4eq:Zmxlabeliannormal2}
 {^xz}=z^{\theta|_C(x)}=z^{\theta(x)}.
\end{equation}

Thus for every $z\in N$ and $x\in G$ one has \eqref{4eq:Zmxlabeliannormal1} or \eqref{4eq:Zmxlabeliannormal2}.
Therefore, $N\leq Z_\theta(G)$, and thus they coincide, as the $\theta$-centre is abelian.
\end{proof}

The following result shows the arithmetic meaning of the $\theta$-centre of the maximal pro-$p$ Galois group
of a field, with $\theta$ the arithmetic orientation.

\begin{thm}\label{4thm:arithmetic thetacentre}
Let $K$ be a field of characteristic $\chr K\neq p$ containing a primitive $p$-th root of unity,
and also $\sqrt{-1}$ if $p=2$.
Let $\theta\colon G_K(p)\to\Z_p^\times$ be the cyclotomic orientation,
and assume that the maximal pro-$p$ Galois group $G_K(p)$ is finitely generated.
Then the following are equivalent:
\begin{itemize}
 \item[(i)] the $\theta$-centre $\Zen_\theta(G)$ is non-trivial;
 \item[(ii)] there exist two $\F_p$-vector subspaces $V,W<H^1(G_K(p),\F_p)$ such that
\begin{eqnarray*}
 && H^1\left(G_K(p),\F_p\right)=V\oplus W, \quad \text{and}\\
 && H^2\left(G_K(p),\F_p\right)=\wedge_2(\cup)(V)\oplus(V\wedge W)\oplus(W\wedge W),
\end{eqnarray*}
with $\wedge_2(\cup)$ as in \eqref{3eq:wedge cup product}.
 \item[(iii)] there exists a non-trivial $p$-Henselian valuation $v\colon K^\times\rightarrow\Gamma$
with residue characteristic different to $p$.
\end{itemize}
Moreover, if the conditions above hold, then
\[d\left(\Zen_\theta(G_K(p))\right)=\dim_{\F_p}(W)=\rank(v).\]
\end{thm}

\begin{proof}
Assume first that $d(G_K(p))=1$ and $K\supseteq\mu_{p^\infty}$.
Then $G_K(p)$ is isomorphic to $\Z_p$, and the $p$-closure of $K$ is $K(p)=K(\mu_{p^\infty})$.
Thus, $K$ has non non-trivial valuations.
Moreover, $\theta$ is not trivial and $\Zen_\theta(G_K(p))$ is trivial.
Therefore we may assume that $d(G_K(p))\geq2$.

We have the second statement as a consequence of Theorem~\ref{4thm:cohomology fibre product},
so that the implication (i) $\Rightarrow$ (ii) is established.

If condition (ii) hold then \cite[Theorem~2.11]{hwangjacob:brauer} implies that the field $K$ is $V$-{\it rigid}
(cf. \cite[Definition~2.1]{hwangjacob:brauer}), and there exists a non-trivial $p$-Henselian (and non-$p$-adic)
valuation $v$ of $K$ of rank $\rank(v)=\dim(W)$.

Thus by \cite[Corollary~26.6.2]{efrat:libro}, there exists a valuation
$v\colon K^\times\rightarrow\Gamma$ which is $(K^\times)^p$-compatible,
and therefore it is also $p$-Henselian.
Moreover, one has that $|v(K^\times):v(K^\times)^p|\geq|V|$,
hence the rank of $v$ is at least $\dim_{\F_p}H^1(\Zen_\theta,\F_p)=\rank(\Zen_\theta(G))$.
This establishes the implication from (ii) to (iii).

The equivalence between (i) and (iii) follows from Proposition~\ref{4prop:Zmxlabeliannormal} and 
from \cite[Theorem~3.1]{ek98:abeliansbgps} in the case $p\neq2$,
and from \cite[\S~4]{en94:pro2gps} in the case $p=2$.
In particular, the equality $\rank(v)=d(\Zen_\theta(G))$ follows by induction.
\end{proof}

\section{Relations}

We want now to analise more in depth the relations of a minimal presentation \eqref{4eq:presentation} for a
finitely generated cyclo-oriented pro-$p$ group $(G,\theta)$, 
with orientation $\theta$ such that either $\image(\theta)\simeq\Z_p$ or $\theta\equiv\mathbf{1}$.
So far we have studied explicitly the group-structure of two types of cyclo-oriented pro-$p$ groups
with non-trivial relations:
\begin{itemize}
 \item[(a)] $\theta$-abelian groups, and more in general cyclo-oriented groups $(G,\theta)$
with non-trivial $\theta$-centre;
 \item[(b)] Demushkin groups, with an orientation as in Theorem~\ref{2thm:Demushkin}.
\end{itemize}

In the former case, we may find a minimal system of generators $\{x_1,\ldots,x_{d-c},z_1,\ldots,z_c\}$
with the $x_i$'s which generate $G_\circ$ and the $z_i$'s which generate $\Zen_\theta(G)$,
such that $\theta(x_1)=1+q$, with $q=p^k$ for some $k=\N\cup\{\infty\}$, and $\theta(x_i)=1$ for $i\geq2$.
Thus one has the relations
\begin{equation}\label{4eq:relations thetacentre}
 z_i^{-q}[x_1,z_i]=1 \quad\text{and}\quad [x_j,z_i]=1,
\end{equation}
with $i=1,\ldots,c$ and $j=2,\ldots,d-c$.
In fact, the relations \eqref{4eq:relations thetacentre} are the ``commutator translation'' of the cyclotomic
action induced by $\theta$ on the $\theta$-centre.

In the latter case, by Theorem~\ref{2thm:Demushkin} we may find a minimal system of generators
$\{x_1,\ldots,x_d\}\subset G$, with $\theta(x_2)=1-q_G$, such that one has the relation
\[x_1^{-q_G}[x_1,x_2]\cdots[x_{d-1},x_d] =1, \]
thus 
\begin{equation}\label{4eq:relation demushkin}
 x_1^{q_G}[x_2,x_1] \equiv 1 \mod [\kernel(\theta),\kernel(\theta)].
\end{equation}
Namely, we have the cyclotomic action, as in \eqref{4eq:relations thetacentre},
but ``perturbed'' by the commutator subgroup of the kernel of the orientation $\theta$.

Recall that $\theta$-abelian pro-$p$ groups and Demushkin groups (together with free pro-$p$ groups) are the 
two ``bounds'' for cyclo-oriented pro-$p$ groups.
Therefore, it is natural to ask the following question: for a finitely generated cyclo-oriented pro-$p$ group
$(G,\theta)$, with $\image(\theta)\simeq\Z_p$, do \eqref{4eq:relations thetacentre} and \eqref{4eq:relation demushkin}
represent ``models'' for the behavior of the relations of $G$, such that they are induced by the 
cyclotomic action modulo $[\kernel(\theta),\kernel(\theta)]$? Namely...

\begin{ques}\label{4ques:relations}
Given a finitely generated cyclo-oriented pro-$p$ group $(G,\theta)$, with $\image(\theta)\simeq\Z_p$,
is it possible to find a minimal presentation \eqref{4eq:presentation}
with generating system $\{x_1,\ldots,x_{d(G)}\}$ and defining relations $r_1,\ldots,r_{r(G)}$ such that
$x_i\in\kernel(\theta)$ for $i=2,\ldots,d(G)$ and either $r_j\in[F,F]$ or
\begin{equation}\label{4eq:cyclotomic relation}
 r_j\equiv x_i^{1-\theta(x_1)}[x_1,x_i]\mod [\kernel(\theta),\kernel(\theta)]
\end{equation}
with $i\in\{2,\ldots,d(G)\}$?
 \end{ques} 

We shall call a relation as in \eqref{4eq:cyclotomic relation} a relation of {\bf cyclotomic type}.
For a cyclo-oriented pro-$p$ group with trivial orientation, the answer to Question~\ref{4ques:relations}
is clearly positive, as Corollary~\ref{2cor:kernelab torfree}
and Remark~\ref{4rem:Gab torsionfree thetatrivial} imply the following.

\begin{prop}\label{4prop:relations theta trivial}
 Let $(G,\theta)$ be a finitely generated pro-$p$ group with trivial cyclotomic orientation $\theta\equiv\mathbf{0}$,
and let \eqref{4eq:presentation} be a minimal presentation for $G$.
Then $R\leq [F,F]$.
\end{prop}

The following collects some useful facts.

\begin{fact}\label{4fact:abelianization G}
 Let $(G,\theta)$ a finitely generated cyclo oriented pro-$p$ group, and let $G^{\ab}$ be the abelianization
$G/[G,G]$.
\begin{itemize}
 \item[(i)] There are positive integers $f,t,k_j$, with $f+t=d(G)=d$, such that
\[G^{\ab}\simeq\Z_p^f\oplus\left(\bigoplus_{j=1}\Z/p^{k_j}\Z\right).\]
 \item[(ii)] One may find a minimal system of generators $\{x_1,\ldots,x_d\}$ such that
for $1\leq i\leq t$ the images $\bar x_i$ in $G^{\ab}$ generate the torsion-free part,
and for $t+1\leq i\leq d$ the $\bar x_i$'s generate the torsion part (in particular
the order of $\bar x_i$ is $p^{k_j}$ for $j=i-f$).
 \item[(iii)] Also, one has $\theta(x_i)=1$ for $i>f$, and one may choose the $x_i$'s for $i\leq f$
such that $\theta(x_i)=1$ also for $2\leq i\leq f$.
\end{itemize}
\end{fact}

\subsection{The module $N^{\ab}$}

For a finitely generated pro-$p$ group $G$ with cyclotomic orientation $\theta$ such that $\image(\theta)\simeq\Z_p$,
let $N$ be the kernel of $\theta$, and set
\[N^{\ab}=\frac{N}{[N,N]},\]
i.e., $N^{\ab}$ is the abelianization of the $\kernel(\theta)$.
Recall that by Corollary~\ref{2cor:kernelab torfree}, $N^{\ab}$ is a torsion-free abelian pro-$p$ group.

The action induced by the conjugation -- i.e.,
\[x.\bar y=\overline{^xy}\equiv x yx^{-1}\mod [N,N],\]
with $x\in G$ and $y\in N$ -- makes $N$ a continuous $\ZpG$-module, and a trivial $\Z_p\dbl N\dbr$-module.
Hence, set $\Gamma=G/N$, and let $k$ be such that $\image(\theta)=1+q\Z_p$, with $q=p^k$.
Then, $\Gamma$ is isomorphic to $\Z_p$ as pro $p$-group, and it may be identified with $1+q\Z_p=\image(\theta)$.
In particular, the complete group ring $\Z_p\dbl\Gamma\dbr$ is the {\it Iwasawa algebra}, and $N^{\ab}$
becomes a $\Z_p\dbl\Gamma\dbr$-module, i.e., an {\bf Iwasawa module} (cf. \cite[5.3.6]{nsw:cohm}).

Let $g$ be an element of $G$ such that $\theta(g)=1+q$. 
Without loss of generality, we may assume that $g=x_1$, with $x_1$ as in Fact~\ref{4fact:abelianization G}.
Let $\gamma$ be te image of $x_1$ in $\Gamma$.
Then $\gamma$ is a generator of $\Gamma$.
Also, the Iwasawa algebra $\Z_p\dbl\Gamma\dbr$ is isomorphic to the power series ring $\Z_p\dbl X\dbr$,
via the morphism\footnote{Note that this is the same map as in Remark~\ref{1rem:magnus algebra}} given
by $\gamma\mapsto1+X$, so that $X\mapsto\gamma-1$, and $N^{\ab}$ becomes a $\Z_p\dbl X\dbr$-module via
\begin{equation}\label{4eq:action ZpX}
\begin{aligned}
 X.\bar y =(\gamma-1).\bar y=x_1.\bar y-\bar y &=\overline{^{x_1}y\cdot y^{-1}}\\ &\equiv [x_1,y]\mod[N,N]
\end{aligned}
\end{equation}
for every $y\in N$.
Hence, one has 
\begin{equation}\label{4eq:action ZpX general}
\begin{aligned}
 \lambda X^m.\bar y &\equiv[x_1,_my]^\lambda\mod[N,N] \\ &\equiv\left[x_1,_my^{\lambda}\right]\mod[N,N]
\end{aligned}
\end{equation}
with $\lambda\in\Z_p$, $m\geq1$ and $[x_i,_my]=[x_1[x_1\ldots[x_1,y]]]$.

Recall that a polynomial $\wp\in\Z_p\dbl X\dbr$ is said to be a {\it Weierstra{\ss} polynomial} if
\[\wp=X^m+\alpha_{m-1}X^{m-1}+\ldots+\alpha_2X^2+\alpha_1X+\alpha_0,\]
with $\alpha_0,\ldots,\alpha_{m-1}\in p\Z_p$ (cf. \cite[Definition~5.3.2]{nsw:cohm}).
Thus, one has the following.

\begin{prop}
 \label{4prop:iwasawa1}
For a finitely generated cyclo-oriented pro-$p$ group $G$ with $N=\kernel(\theta)$, there are positive integers $f,t$
and irreducible Weierstra{\ss} polynomials $\wp_j$, with $j=1,\ldots,t$, such that the $\Z_p\dbl\Gamma\dbr$-module
$N^{\ab}$ is isomorphic to a submodule of finite index of the $\Z_p\dbl X\dbr$-module
\begin{equation}\label{4eq:iwasawa module} 
M=\Z_p\dbl X\dbr^f\oplus\left(\bigoplus_{j=1}^t\frac{\Z_p\dbl X\dbr}{\wp_j^{m_j}}\right),
\end{equation}
with $m_j\geq1$.
\end{prop}

\begin{proof}
Since $N^{\ab}$ is an Iwasawa module, the Structure Theorem for Iwasawa modules implies that there exist positive
integers $a_h$, with $h=1,\ldots,s$, together with suitable $f,t,s\geq0$ and $\wp_j$'s such that one has a
morphism of $\Z_p\dbl X\dbr$-modules
\[\phi\colon N^{\ab}\longrightarrow \Z_p\dbl X\dbr^f\oplus
\left(\bigoplus_{h=1}^s\dfrac{\Z_p\dbl X\dbr}{p^{a_h}}\right)\oplus
\left(\bigoplus_{j=1}^t\dfrac{\Z_p\dbl X\dbr}{\wp_j^{m_j}}\right),\]
and $\phi$ has finite kernel and finite cokernel (cf. \cite[Theorem~5.3.8]{nsw:cohm}).

Since $N^{\ab}$ is a torsion-free $\Z_p$-module, the kernel of $\phi$ is trivial, and the $p$-torsion part 
(the quotients $\Z_p\dbl X\dbr/p^{a_h}$) disappears.
Thus, one has that $\phi\colon N^{\ab}\to M$ is injective, and $N^{\ab}$ is isomorphic
to a $\Z_p\dbl X\dbr$-submodule of $M$ of finite index.
\end{proof}

For every $j=1,\ldots,t$, set
\[\wp_j^{m_j}=X^{h_j}+\alpha_{(j,h_j-1)}X^{h_j-1}+\ldots+\alpha_{(j,1)}X+\alpha_{j,0}.\]
Therefore, for every $\bar y\in N^{\ab}$, one has $\wp_j^{m_j}.\bar y=0$, and thus by \eqref{4eq:action ZpX general}
\begin{equation}\label{4eq:commutator polynomial}
r=[x_1,_{h_j}y]\cdot[x_1,_{h_j-1}y]^{\alpha(j,h_j-1)}\cdots[x_1,y]^{\alpha_{(j,1)}}y^{\alpha_{j,0}}\in [N,N].
\end{equation}

Now let $y\in N$ belong to a generating system of $G$ (e.g., $y=x_i$ for an $i\geq2$, with $x_i$ as in
Fact~\ref{4fact:abelianization G}), and let $f_y\colon G\to\Z_p(1)$ be the continuous crossed homomorphism
such that $f_y(y)=1$ and $f_y(x)=0$ for any other generator of $G$ (in particular, $f_y(x_1)=0$).
Such a crossed homomorphism exists by Fact~\ref{4fact:crossed homomorphisms and relations}.
Since $f_y|_N$ is a group homomorphism and $f_y|_{[N,N]}\equiv\mathbf{0}$, by \eqref{4eq:c on commutators} one has 
\[ f_y\left([x_1,_my]^\alpha\right)=\alpha(\theta(x_1)-1)f_y([x_1,_{m-1}y])=\ldots=\alpha q^m \]
(notice that $[x_1,_my]\in N$ for any $m\geq0$).
Hence, by \eqref{4eq:commutator polynomial},
\begin{equation}\label{4eq:c on commutator polynomial}
 f_y(r)=\alpha_{j,0}+\alpha_{(j,1)}q+\ldots+\alpha_{(j,h_j-1)}q^{h_j-1}+q^{h_j}=\wp_j^{m_j}(q).
\end{equation}
On the other hand, one has $f_y(r)=0$, as $r\in[N,N]$, so that $\wp_j^{m_j}(q)=0$, and $X-q$ divides $\wp_j$.
Since $\wp_j$ is an irreducible polynomial, it follows that in fact $\wp_j=X-q$ for every $j=1,\ldots,t$.
In particular, one has the monomorphism of $\Z_p\dbl X\dbr$-modules
\begin{equation}\label{4eq:M module}
 \phi\colon N^{\ab} \longrightarrow 
M=\Z_p\dbl X\dbr^f\oplus\left(\bigoplus_{j=1}^t\frac{\Z_p\dbl X\dbr}{(X-q)^{m_j}}\right), 
\end{equation}
with $m_j\geq1$.

\begin{prop}\label{4prop:wrong quotient}
Let $(G,\theta)$ be a finitely generated pro-$p$ group with cyclotomic orientation $\theta$
such that $\image(\theta)\simeq\Z_p$, and with minimal presentation \eqref{4eq:presentation}.
Let $x_1,y\in G$ be as above, and let $r\in R$ be a relation induced
by the polynomial $(X-q)^m$, with $m\geq1$, i.e.,
\begin{equation}\label{4eq:relation power binomial}
 r\equiv y^{(-q)^m}[x_1,y]^{m(-q)^{m-1}}\cdots[x_1,_{m-1}y]^{-mq}[x_1,_my]\mod[N,N].
\end{equation}
Then $r$ can occur as defining relation only if $m=1$.
\end{prop}

\begin{proof}
Consider the short exact sequence
\begin{equation}\label{4eq:ses wrong quotient}
\xymatrix{ 1\ar[r] & C\ar[r] & G\ar[r] & \bar G\ar[r] & 1 }, 
\end{equation}
where $C$ is the closed normal subgroup of $G$ generated by all the generators $x_i$'s with $i\neq1$ and $x_i\neq y$,
and $\bar G=G/C$.
Then $\bar G$ has a minimal presentation $1\to \bar R\to \bar F\to \bar G\to1$,
where $d(\bar F)=d(\bar G)=2$, and $\bar R$ is generated as normal subgroup of $\bar F$ by
\[\bar r=y^{(-q)^m}[x_1,y]^{m(-q)^{m-1}}\cdots[x_1,_{m-1}y]^{-mq}[x_1,_my],\]
where we consider $x_1$ and $y$ as elements of $\bar G$ and of $\bar F$ as well with an abuse of notation.
Then \eqref{4eq:ses wrong quotient} induces the five term exact sequence
\begin{equation}\label{4eq:biscia prop wrong quotient}
 \xymatrix@C=1.2truecm{0\ar[r] & H^1(\bar G,\F_p)\ar[r]^-{\inf_{G,\bar G}^1} & H^1(G,\F_p)\ar[r]^-{\res_{G,C}^1} & 
 H^1(C,\F_p)^G \ar`r[d]`[l]^-{\text{tg}_{G,C}} `[dlll] `[dll] [dll] \\
  & H^2(\bar G,\F_p)\ar[r]^-{\inf_{G,\bar G}^2} & H^2(G,\F_p). &} 
\end{equation}
By construction of $C$, the map $\res_{G,C}^1$ is surjective, so that the transgression map is trivial,
and $\inf_{G,\bar G}^2$ is injective.
In particular, one has the commutative diagram
\begin{equation}\label{4eq:commutative diagram wrong quotient}
 \xymatrix@R=1.1truecm{ H^1(\bar G,\F_p)\times H^1(\bar G,\F_p)\ar@<6ex>[d]^-{\inf_{G,\bar G}^1}\ar@<-6ex>[d]^-{\inf_{G,\bar G}^1}
\ar[rr]^-{\cup} && H^2(\bar G,\F_p)\ar[d]^-{\inf_{G,\bar G}^2} \\
 H^1(G,\F_p)\times H^1(G,\F_p)\ar[rr]^-{\cup} && H^2(G,\F_p) }
\end{equation}
where the vertical arrows are monomorphisms by \eqref{4eq:biscia prop wrong quotient}.

Assume now that $m>1$.
In this case the relation $\bar r$ lies in $\lambda_3(\bar F)$, thus by \cite[Corollary~9.2]{cem:quotients}
the cohomology ring $H^\bullet(\bar G,\F_p)$ is not a quadratic $\F_p$-algebra.
In particular, let $\chi,\psi\in H^1(\bar G,\F_p)$ be the duals of $x_1$ and $y$.
Since $\bar r\in\lambda_3(\bar F)$ and $r(\bar G)=1$, \cite[Propositions~3.9.12-13]{nsw:cohm}
imply that $\chi\cup\psi=0$.
On the other hand, one has 
\[H^2(\bar G,\F_p)\simeq\F_p, \quad \text{as } \dim_{\F_p}\left(H^2(\bar G,\F_p)\right)=r(\bar G)=1. \]
Set $A=\image(\inf_{G,\bar G}^2)\leq H^2(G,\F_p)$. 
By \eqref{4eq:commutative diagram wrong quotient} one has that 
\[ A\neq0 \quad\text{and}\quad  \text{inf}_{G,\bar G}^1(\chi)\cup\text{inf}_{G,\bar G}^1(\psi)=0,\]
thus, $A$ is not induced by the cup product of $\chi$ and $\psi$.
Also, by the Hochschild-Serre spectral sequence $E_2^{st}=H^s(\bar G,H^t(C,\F_p))$
associated to \eqref{4eq:ses wrong quotient}, $A$ is not generated by cup products of elements of $H^1(C,\F_p)^G$,
nor by cup products of $\inf_{G,\bar G}^1(\chi)$ and $\inf_{G,\bar G}^1(\psi)$ with elements of $H^1(C,\F_p)^G$.

Therefore, for $m>1$ $A$ is not generated by elements of $H^1(G,\F_p)$, i.e., the cohomology ring
$H^\bullet(G,\F_p)$ is not a quadratic $\F_p$-algebra, a contradiction.
Thus, $m=1$ and one has $r\equiv y^{-q}[x_1,y]\bmod[N,N]$.\footnote{ In fact this argument does not work properly
(but the commutative diagrams are cool!),
but this one does: assume that one has a defining relation as \eqref{4eq:relation power binomial}, with $m\geq2$;
then \[y^{-q}[x_1,y]\not\equiv 1\bmod[N,N]\quad \text{and} \quad  \left(y^{-q}[x_1,y]\right)^m\equiv1\bmod[N,N],\]
i.e., $N^{\ab}$ has non-trivial torsion, a contradiction.}
\end{proof}

Therefore, the relations of $G$ which induce the torsion-module part in the image of $N^{\ab}$ in $M$ via $\phi$
must be of cyclotomic type, and from \eqref{4eq:M module} and Proposition~\ref{4prop:wrong quotient}
one deduces the following.

\begin{thm}\label{4thm:Nab}
Let $(G,\theta)$ be a cyclo-oriented pro-$p$ with $\image(\theta)=1+q\Z_p$ and $N=\kernel(\theta)$,
and let $N/[N,N]$ be a $\Z_p\dbl X\dbr$-module with the action induced
by \eqref{4eq:action ZpX} and \eqref{4eq:action ZpX general}.
Then \eqref{4eq:M module} induces the isomorphism of $\Z_p\dbl X\dbr$-modules
\begin{equation}\label{5eq:Nab}
 \xymatrix {\dfrac{N}{[N,N]}\ar[r]^-{\sim} & \Z_p\dbl X\dbr^f\oplus\left(\dfrac{\Z_p\dbl X\dbr}{X-q}\right)^t }
\end{equation}
with $f+t=d(G)-1$.
\end{thm}

In particular, Theorem~\ref{4thm:Nab} provides a positive answer to Question~\ref{4ques:relations}.

\begin{exs}\label{4exs:Nab}
\begin{itemize}
 \item[(a)] Let $(G,\theta)$ be a finitely generated $\theta$-abelian pro-$p$ group with $\image(\theta)=1+p\Z_p$.
Then $N=\Zen_\theta(G)$, and one has
\[N^{\ab}\simeq\left(\frac{\Z_p\dbl X\dbr}{X-p}\right)^{d(G)-1}.\]
 \item[(b)] Let $(G,\theta)$ be a Demushkin group with $p\neq2$ and $\image(\theta)=1+p\Z_p$,
and let $x_i$, $1\leq i\leq d(G)$, be as in Theorem~\ref{2thm:Demushkin}.
Then $N$ is generated by $x_1,x_3,\ldots,x_{d(G)}$, and one has
\[N^{\ab}\simeq\frac{\Z_p\dbl X\dbr}{X-p}\oplus\Z_p\dbl X\dbr^{d(G)-2},\]
where the torsion-module part is induced by $\bar x_1$.
\end{itemize}
\end{exs}

\begin{cor}\label{4cor:abelianization}
 Let $(G,\theta)$ be a finitely generated cyclo-oriented pro-$p$ group with $\image(\theta)=1+q\Z_p$,
with $q=p^k$ and $k\in\N\cup\{\infty\}$.
Then one has 
\[G/[G,G]\simeq \Z_p^{d(G)-t}\oplus\left(\Z_p/q\Z_p\right)^t,\]
with $t$ as in Theorem~\ref{4thm:Nab}.
In particular, if $G_K(p)$ is the maximal pro-$p$ Galois group of a field $K$ containing a primitive $p$-th
root of unity (and also $\sqrt{-1}$ if $p=2$), then
\[\frac{G_K(p)}{[G_K(p),G_K(p)]}\simeq\Z_p^{d-t}\oplus(\Z/q\Z)^t \quad \text{or} \quad
\frac{G_K(p)}{[G_K(p),G_K(p)]}\simeq\Z_p^t,\]
the former case if $K$ contains a primitive $q$-th root of unity but not a $pq$-th one, and the latter if 
$K$ contains all roots of unity of $p$-power order, with $d=\dim_{\F_p}(K^\times/(K^\times)^p)$.
\end{cor}

Thus, for a finitely generated pro-$p$ group $G$ with non-trivial cyclotomic orientation $\theta\colon G\to\Z_p^\times$,
one may define the {\it 1-torsion rank} $t_1(G)$ and the {\it 1-free rank} $f_1(G)$ of $G$ to be the rank
as $\Z_p$-module of the $\Z_p\dbl\Gamma\dbr$-torsion part of $H_1(\kernel(\theta),\Z_p)$,
resp. of the $\Z_p\dbl\Gamma\dbr$-free part of $H_1(\kernel(\theta),\Z_p)$;
i.e., \[t_1(G)=t \quad \text{and} \quad f_1(G)=f,\] with $t$ and $f$ as in Theorem~\ref{4thm:Nab}.

By Examples~\ref{4exs:Nab} one easily deduces that if $\Zen_\theta(G)$ is not trivial, then also the 
$\Z_p\dbl\Gamma\dbr$-torsion part of $H_1(\kernel(\theta),\Z_p)$ is non-trivial, i.e., $t_1(G)\geq d(\Zen_\theta(G))$.

Then, one may formulate the following questions.

\begin{ques}\label{4ques:torsionrank}
Let $G$ be a finitely generated pro-$p$ group with non-trivial cyclotomic orientation $\theta\colon G\to\Z_p^\times$.
\begin{itemize}
 \item[(i)] Assume that $t_1(U)=0$ for every open subgroup $U$ of $G$. Then $G$ is a free pro-$p$ group.
 \item[(ii)] Assume that $t_1(U)=1$ for every open subgroup $U$ of $G$, and assume also that $\Zen_\theta(G)$
is trivial. Then $G$ is a cyclotomic fibre product $F\ltimes_\theta Z$, with $F$ a free pro-$p$ group and
$Z\simeq\Z_p$, or $G$ is a Demushkin group, or $G$ is the free pro-$p$ product $F\ast_{\hat p}\tilde G$,
with $F$ a free pro-$p$ group and $\tilde G$ one of the former two groups (fibre product or Demushkin).
 \item[(iii)] Assume that $\Zen_\theta(G)$ is trivial and $t_1(G)\geq2$.
Then there exist two closed subgroups $G_1$ and $G_2$ of $G$ such that $G=G_1\ast_{\hat p}G_2$.
\end{itemize}
\end{ques}

These questions will provide me some work (and hopefully a job) in the next future.

\section{Cyclo-oriented ETC}

Now it makes sense to try to extend the Elementary Type Conjecture to the whole class of finitely generated 
pro-$p$ groups with a cyclotomic orientation.
In order to do this, we define the class of elementary type cyclo-oriented pro-$p$ groups.

\begin{defi}
For $p$ a prime number, let $\mathcal{ET}_p$ be the minimal class of
finitely generated pro-$p$ groups with an orientation such that
\begin{itemize}
 \item[(i)] the pro-cyclic group $\Z_p$, together with an orientation
$\theta\colon\Z_p\rightarrow\Z_p^\times$ (possibly the trivial one), is in $\mathcal{ET}_p$;
 \item[(ii)] if $p=2$, then the cyclic group $\Z/2\Z$ of order 2, together with the non-trivial orientation 
$\theta\colon \Z/2\Z\rightarrow\{\pm1\}\subset\Z_2^\times$, is in $\mathcal{ET}_2$;
 \item[(iii)] every Demushkin group equipped with an orientation as described in Theorem~\ref{2thm:Demushkin} 
is in $\mathcal{ET}_p$;
\item[(iv)] if the group $(G,\theta)$  is in $\mathcal{ET}_p$, then also the cyclotomic fibre product
$G\ltimes_\theta\Z_p(1)$ is in $\mathcal{ET}_p$;
\item[(v)] if the groups $(G_1,\theta_1)$ and $(G_2,\theta_2)$ are in $\mathcal{ET}_p$,
then also the free pro-$p$ product $G_1\ast_{\hat{p}}G_2$ with the induced orientation $\tilde\theta$
is in $\mathcal{ET}_p$.
\end{itemize}
Then $\mathcal{ET}_p$ is called the class of {\bf elementary type} pro-$p$ groups with a cyclotomic orientation.
\end{defi}

By Theorem~\ref{4thm:cyclotomic free product} and Theorem~\ref{4thm:cyclotomic fibreprod}, every 
element of $\mathcal{ET}_p$ is a finitely generated cyclo-oriented pro-$p$ group.
In particular, every element of $\mathcal{ET}_p$ can be realized as maximal pro-$p$ Galois group of a field.
We can now state a new formulation of the ETC. 

\begin{ques}
Let $K$ be a field.
 \begin{itemize}
  \item[(i)] Assume that the absolute Galois group $G_K$ of $K$ is a finitely generated pro-$p$ group.
Then $G_K$ is in $\mathcal{ET}_p$.
\item[(ii)] Assume that $K$ contains a primitive $p$-th root of unity, and that the maximal pro-$p$
Galois group $G_K(p)$ is finitely generated. Then $G_K(p)$ is in $\mathcal{ET}_p$.
 \end{itemize}
Obviously, a positive answer to (ii) provides also a positive answer to (i).
\end{ques}

Note that Question~\ref{4ques:torsionrank} is coherent with the ETC, and a positive answer would be a contribution
to the ETC.

Moreover, the recent results obtained independently by Th.~Weigel and P.~Zalesski\u{i}, and by K.~Wingberg on
free decomposability of pro-$p$ groups provide new tools also to study the ETC
(cf. \cite{thomaspavel:stalldecomp} and \cite{wingberg:stalldecomp}).


\chapter[Koszulity]{Koszulity for cyclo-oriented pro-$p$ groups}

\section{Koszul duality}

Let $A_\bullet$ be a quadratic algebra over a field $\F$, and let $A_1$ and $\mathcal{R}$ be
as in \eqref{2eq:quadratic algebra}, with $A_1$ of finite dimension over $\F$.
In particular, set $\mathcal{R}_1$ to be the $\F$-vector subspace of $A_1\otimes A_1$ which generates
$\mathcal{R}$ as two-sided ideal of $\mathcal{T}^\bullet(A_1)$.

Since $A_1$ is finitely generated, one has the isomorphism of $\F$-vector spaces
\begin{equation}\label{5eq:dual tensor product}
A_1^*\otimes A_1^*\simeq \left(A_1\otimes A_1\right)^*.
\end{equation}
Let $\mathcal{R}_1^\perp\leq (A_1\otimes A_1)^*$ be the annihilator of $\mathcal{R}_1$, namely,
\[\mathcal{R}_1^\perp=\{f\in (A_1\otimes A_1)^*\;|\;f(v)=0\text{ for all }v\in\mathcal{R}_1\}.\]
By \eqref{5eq:dual tensor product} we may consider $\mathcal{R}_1^\perp$ as a subspace of $A_1^*\otimes A_1^*$.
In particular, one has that $\mathcal{R}_1^\perp\simeq(A_1^{\otimes2}/R_1)^*$.
Then, the short exact sequences of $\F$-vector spaces
\[ \xymatrix@R=.33truecm{ 0\ar[r] & \mathcal{R}_1\ar[r] & A_1^{\otimes2}\ar[r] & A_2\ar[r] & 0, \\
              0 & \mathcal{R}_1^*\ar[l] & \left(A_1^{\otimes2}\right)^*\ar[l] & \mathcal{R}_1^\perp\ar[l] & 0,\ar[l]} \]
where one is the dual of the other, induce the following definition.

\begin{defi}\label{5defi:koszul dual}
Let $A_\bullet$ be a quadratic algebra over a field $\F$ with $\dim(A_1)$ finite.
The {\bf Koszul dual} $A_\bullet^!$ of $A_\bullet$ is the quadratic algebra over $\F$ given by
\[A_\bullet^!=\frac{\mathcal{T}^\bullet(A_1^*)}{\mathcal{R}^\perp},\]
where $\mathcal{R}^\perp$ is the two-sided ideal of $\mathcal{T}^\bullet(A_1^*)$ generated by
$\mathcal{R}_1^\perp$.
\end{defi}

Note that the Koszul dual is an involution, as $(A_\bullet^!)^!=A_\bullet$ for every quadratic algebra.

\begin{exs}\label{5exs:koszul duals}
\begin{itemize}
 \item[(a)] Let $V$ be a finite-dimensional $\F$-vector space, and let $A_\bullet$ be the tensor algebra
$\mathcal{T}^\bullet(V)$. Then $\mathcal{R}=0$, and the Koszul dual $A_\bullet^!$ is the 
algebra of {\it dual numbers} \[\mathcal{T}^\bullet(V)^!=\F\oplus V^*,\] with trivial multiplication.
 \item[(b)] Let $\{v_1,\ldots,v_n\}$ be a basis for the $\F$-vector space $V$,
and let $A_\bullet$ be the symmetric algebra $S_\bullet(V)$.
Then a basis for $\mathcal{R}_1^\perp$ is 
\[\left\{v_i^*\otimes v_j^*+v_j^*\otimes v_i^*,1\leq i,j\leq n\right\},\]
with $v_i^*\in V^*$ the dual of $v_i$.
Hence, the Koszul dual of $S_\bullet(V)$ is the exterior algebra $\bigwedge_\bullet(V^*)$,
and conversely $\bigwedge_\bullet(V)^!=S_\bullet(V^*)$.
\end{itemize}
\end{exs}

The Koszul dual behaves with respect to the constructions of quadratic algebras defined in Subsection~2.1.1
in the following way (cf. \cite[\S~3.1, Corollary~2.1, p.~58]{pp:quadratic algebras}).

\begin{prop}\label{5prop:koszulduality products algebras}
 Let $A_\bullet$ and $B_\bullet$ be two quadratic algebras over a field $\F$.
\begin{itemize}
 \item[(i)] The Koszul dual of the direct sum is the free product of the Koszul duals, i.e.
\[ \left(A_\bullet\sqcap B_\bullet\right)^!=A_\bullet^!\sqcup B_\bullet^! \quad\text{and}\quad 
 \left(A_\bullet\sqcup B_\bullet\right)^!=A_\bullet^!\sqcap B_\bullet^!. \]
 \item[(ii)] The Koszul dual of the symmetric tensor product is the skew-commutative tensor product
of the Koszul duals, i.e.
\[ \left(A_\bullet\otimes^1 B_\bullet\right)^!=A_\bullet^!\otimes^{-1} B_\bullet^! \quad\text{and}\quad 
 \left(A_\bullet\otimes^{-1} B_\bullet\right)^!=A_\bullet^!\otimes^1 B_\bullet^!. \]
\end{itemize}
\end{prop}

The $\F_p$-cohomology ring of a Bloch-Kato pro-$p$ group $G$ is a quadratic algebra over $\F_p$:
thus, it is worth asking what is the Koszul dual of such algebra.
Also, other important graded algebras for a pro-$p$ group $G$ are the restricted Lie algebra $L_\bullet(G)$
and the algebra $\grad_\bullet(G)$, which are related by Theorem~\ref{1thm:envelope Zassenhaus filtration}
-- and they happen to be quadratic in some important cases, such as Demushkin groups and $\theta$-abelian groups.
The following definition relates $H^\bullet(G,\F_p)$ and $\grad_\bullet(G)$.

\begin{defi}\label{5defi:koszulduality gps}
 Let $G$ be a Bloch-Kato pro-$p$ group.
Then $G$ is said to be a {\bf Koszul duality} group if the algebra $\grad_\bullet(G)$ is the Koszul dual
of the algebra $H^\bullet(G,\F_p)$, i.e., 
\[H^\bullet(G,\F_p)^!\simeq\grad_\bullet(G).\]
\end{defi}

It follows that if $G$ is a Koszul duality pro-$p$ group, then $\grad_\bullet(G)$ is a quadratic $\F_p$-algebra.
In particular, let 
\begin{equation}\label{5eq:group presentation}
 \xymatrix{1\ar[r]&R\ar[r]&F\ar[r]&G\ar[r]&1}
\end{equation}
be a minimal presentation of $G$, with $R$ generated by $\{r_1,\ldots,r_{r(G)}\}$ as normal subgroup of $F$.
If $G$ is a Koszul duality pro-$p$ group, one has that the initial forms of the relations $r_i$'s 
have degree 2 in $L_\bullet(G)$, i.e., 
\begin{equation}\label{5eq:quadratic relations}
 r_i\in D_2(F)\smallsetminus D_3(G)(G)\quad \text{for every }i=1,\ldots,r(G).
\end{equation}
Thus, we say that a pro-$p$ group $G$ is {\bf $\mathcal{U}$-quadratic} if it has a presentation
\eqref{5eq:group presentation} such that \eqref{5eq:quadratic relations} holds.

We are to see in the following that the basic blocks of the ETC -- i.e., free pro-$p$ groups, Demushkin groups
and groups obtained via cyclotomic fibre products and free pro-$p$ products -- are Koszul duality groups.

\begin{rem}\label{5rem:pontryagin}
 Note that for a vector space $V$ over $\F_p$, the Pontryagin dual and the $\F_p$-dual are the same, i.e.,
$V^*=V^\vee$.
\end{rem}

\section{Koszul duality and the ETC}
\subsection{Free pro-$p$ groups and Demushkin groups}
Let $F$ be a free pro-$p$ group.
(recall that by Remark~\ref{2rem:cyclo-oriented free gps} every orientation is cyclotomic for $F$.)
Then $F$ has cohomological dimension $1$,
and the cohomology ring $H^\bullet(F,\F_p)$ is concentrated in degree 0 and 1.
In particular, one has
\[H^\bullet(F,\F_p)=\frac{\mathcal{T}^\bullet\left(F/\Phi(F)^\vee\right)}{\mathcal{R}},
\quad\text{with }\mathcal{R}_1=\left(F/\Phi(F)^\vee\right)^{\otimes2}.\]
Therefore, by Example~\ref{5exs:koszul duals} the Koszul dual of $H^\bullet(F,\F_p)$ is the tensor algebra 
$\mathcal{T}^\bullet(F/\Phi(F))$.

Consider now the restricted Lie algebra induced by the Zassenhaus filtration of $F$.
By Theorem~\ref{1thm:L for free groups} the algebra $L_\bullet(F)$ is a free restricted Lie algebra over $\F_p$,
and by Proposition~\ref{1prop:free restricted lie algebras} the universal envelope a free non-commutative 
$\F_p$-algebra, i.e., 
\begin{equation}\label{5eq:free envelope}
 \grad_\bullet(F)\simeq\mathcal{U}_p(L_\bullet(F))\simeq\F_p\langle\mathcal{X}\rangle=
\mathcal{T}^\bullet\left(F/\Phi(F)^\vee\right),
\end{equation}
with $\mathcal{X}$ a minimal generating system of $F$, where $\F_p\langle\mathcal{X}\rangle$
is to be considered a graded algebra with the grading induced by the degrees of the monomials.
Therefore, one has the following.

\begin{prop}\label{5prop:koszul dual free groups}
 Let $F$ be a finitely generated free pro-$p$ group. 
Then the graded $\F_p$-algebra $\grad_\bullet(F)$ is the Koszul dual of the cohomology ring $H^\bullet(F,\F_p)$,
i.e., \[H^\bullet(F,\F_p)^!\simeq\grad_\bullet(F).\]
\end{prop}

Now let $G$ be a Demushkin group, with $q_G\neq2$.
Let $\mathcal{X}=\{x_1,\ldots,x_d\}$ be a minimal generating system for $G$ as in Theorem~\ref{2thm:Demushkin},
and for each $i=1,\ldots,d$, let $\chi_i\in H^1(G,\F_p)$ be the dual of $x_i$.
Then, one may write 
\[H^\bullet(G,\F_p)=\frac{\mathcal{T}^\bullet\left(H^1(G,\F_p)\right)}{\langle\mathcal{R}_1\rangle},\]
with $\mathcal{R}_1\leq H^1(G,\F_p)^{\otimes2}$.

Since the element $\chi_{2i-1}\cup\chi_{2i}=-\chi_{2i}\cup\chi_{2i-1}$ generates $H^2(G,\F_p)$,
with $i=1,\ldots,d/2$, one has that $\mathcal{R}_1^\perp\simeq H^2(G,\F_p)^*$ is generated by the element
\[\rho=\chi_1^*\otimes \chi_2^*-\chi_2^*\otimes \chi_1^*+\chi_3^*\otimes \chi_4^*-\ldots
+\chi_{d-1}^*\otimes \chi_d^*-\chi_d^*\otimes \chi_{d-1}^*,\]
thus
\[ H^\bullet(G,\F_p)^!=\frac{\mathcal{T}^\bullet\left(H^1(G,\F_p)^*\right)}{\langle\rho\rangle}\simeq
\frac{\F_p\langle X_1,\ldots,X_d\rangle}{\langle[X_1,X_2]+\ldots+[X_{d-1},X_d]\rangle}, \]
where $\mathcal{X}=\{X_1\ldots,X_d\}$ a set of non-commutative indeterminates
and $[X_i,X_j]=X_iX_j-X_jX_i$, with the grading induced by the degrees.
Therefore, by \eqref{2eq:restricted envelope demushkin}, one has the following.

\begin{prop}\label{5prop:koszul dual demushkin groups}
 Let $G$ be a finitely generated Demushkin group with $q_G\neq2$. 
Then the $\F_p$-cohomology ring $H^\bullet(G,\F_p)$ is the Koszul dual
of the graded $\F_p$-algebra $\grad_\bullet(G)$, i.e.,
\[H^\bullet(G,\F_p)^!\simeq\grad_\bullet(G).\]
\end{prop}

\subsection{Free products and cyclo-fibre products}
In order to deal with the restricted Lie algebra of a free pro-$p$ product of pro-$p$ groups, we have to define
the free product for restricted Lie algebras.

Let $L_1$ and $L_2$ be two restricted Lie algebras.
The {\it free product} of restricted Lie algebras of $L_1$ and $L_2$ is the restricted Lie algebra
$L=L_1\ast_{\mathcal{L}}L_2$, equipped with two embeddings $\iota_i\colon L_i\to L$, $i=1,2$
such that for every morphism of restricted Lie algebras $\phi_i\colon L_i\to H$ there exists a unique morphism
$\tilde\phi_i\colon L\to H$ such that the diagram
\[ \xymatrix{ L_i\ar[r]^-{\iota_i}\ar[d]_{\phi_i} & L_1\ast_{\mathcal{L}}L_2\ar@{.>}[dl]^{\tilde\phi_i} \\ H & } \]
commutes for both $i$'s (cf. \cite[Remark~1]{lichtman:lie}).

\begin{prop}\label{5prop:lie algebra free products}
Let $G_1$ and $G_2$ be two finitely generated pro-$p$ groups, and set $G$ to be
the free pro-$p$ product $G_1\ast_{\hat p} G_2$.
Then the restricted Lie algebra $L_\bullet(G)$ is the free product of restricted Lie algebras
$L_\bullet(G_1)\ast_{\mathcal{L}}L_\bullet(G_2)$.
Moreover, its universal envelope is the free product
\begin{equation}\label{5eq:envelope pree products}
 \mathcal{U}_p\left(L_\bullet(G)\right)=\mathcal{U}_p\left(L_\bullet(G_1)\right)\sqcup\mathcal{U}_p\left(L_\bullet(G_2)\right).
\end{equation}
\end{prop}

\begin{proof}
 Let $G^{abs}$ be the {\it abstract} free product $G_1\ast G_2$.
Since both $G_1$ and $G_2$ are finitely generated, \eqref{2eq:free prop prod N} implies that the free pro-$p$
product $G=G_1\ast_{\hat p}G_2$ is the pro-$p$ completion of $G^{abs}$.
By \cite[Theorem~2]{lichtman:lie} one has 
\begin{equation}\label{5eq:lie algebra abstract freeprod}
 L_\bullet(G^{abs})=L_\bullet(G_1)\ast_{\mathcal{L}}L_\bullet(G_2).
\end{equation}
Since $G$ is finitely generated, every element of the Zassenhaus filtration $D_i(G)$ is open in $G$, and thus
every $D_i(G^{abs})$ has finite index in $G^{abs}$.
Therefore $D_i(G)/D_{i+1}(G)\simeq D_i(G^{abs})/D_{i+1}(G^{abs})$ for every $i\geq1$,
and one has the isomorphism of restricted Lie algebras $L_\bullet(G)\simeq L_\bullet(G^{abs})$.

Equality \eqref{5eq:envelope pree products} follows directly from \cite[Theorem~1]{lichtman:lie}
and Theorem~\ref{1thm:envelope Zassenhaus filtration}.
\end{proof}

Thus, Proposition~\ref{5prop:lie algebra free products} and 
Proposition~\ref{5prop:koszulduality products algebras} imply the following.

\begin{prop}\label{5prop:koszulduality freeprod}
 Let $(G_1,\theta)$ and $(G_2,\theta)$ be two finitely generated cyclo-oriented pro-$p$ groups,
and assume that both are Koszul duality pro-$p$ groups.
Then also the free pro-$p$ product $G_1\ast_{\hat p}G_2$ is a Koszul duality pro-$p$ group.
\end{prop}

\begin{proof}
By \cite[Theorem~4.1.4]{nsw:cohm}, one has 
\begin{equation}\label{5eq:cohomology freeprod}
 H^\bullet(G_1\ast_{\hat p}G_2,\F_p)=H^\bullet(G_1,\F_p)\sqcap H^\bullet(G_2,\F_p)
\end{equation}
(direct product of quadratic $\F_p$-algebras).
Therefore,
\begin{eqnarray*}
H^\bullet\left(G_1\ast_{\hat p}G_2,\F_p\right)^! &=& \left(H^\bullet(G_1,\F_p)\sqcap H^\bullet(G_2,\F_p)\right)^! \\
 &=& H^\bullet(G_1,\F_p)^!\sqcup H^\bullet(G_2,\F_p)^! \\
 &\simeq& \grad_\bullet(G_1)\sqcup\grad_\bullet(G_2) \\
 &=& \grad_\bullet\left(G_1\ast_{\hat p}G_2\right),
\end{eqnarray*}
and the statement holds.
\end{proof}

Now let $(G,\theta)$ be a finitely generated $\theta$-abelian pro-$p$ group,
with $\image(\theta)\simeq\Z_2$ if $p=2$.
Then by Theorem~\ref{3thm:equivalence theta-abelian} the $\F_p$-cohomology ring of $G$ is the exterior algebra
\[H^\bullet(G,\F_p)\simeq\bigwedge_{i=0}^{d(G)} H^1(G,\F_p).\]
Therefore, by Example~\ref{5exs:koszul duals}, the Koszul dual of $H^\bullet(G,\F_p)$ is the symmetric algebra
$S^\bullet(G/\Phi(G))$.
Note that such symmetric algebra is isomorphic to the commutative
polynomial algebra $\F_p[X_1,\ldots,X_{d(G)}]$, where $\{X_1,\ldots,X_{d(G)}\}$ is a set
of commutative indeterminates, equipped with the grading induced by the degrees of the monomials.
Thus, by \eqref{3eq:restricted envelope theta-abelian} one has 
\[\grad_\bullet(G)\simeq\mathcal{U}_p(L_\bullet(G))\simeq H^\bullet(G,\F_p)^!.\]

We are in a similar situation with $(G,\theta)$ a finitely generated cyclo-oriented pro-$p$ group,
with non-trivial $\theta$-centre and with $\image(\theta)\simeq\Z_2$ if $p=2$.
Indeed, Theorem~\ref{4thm:cohomology fibre product} and
Proposition~\ref{5prop:koszulduality products algebras} imply the following.

\begin{prop}\label{5prop:koszul dual fibreprod}
 Let $(G,\theta)$ be a finitely generated cyclo-oriented pro-$p$ group with non-trivial $\theta$-centre,
and assume $\image(\theta)\simeq\Z_2$ for $p=2$. 
Moreover, assume that $G_\circ$ is a Koszul duality group, with $G_\circ$ as in Theorem~\ref{4thm:Zcentresplit}.
Then the $\F_p$-cohomology ring $H^\bullet(G,\F_p)$ is the Koszul dual
of the graded $\F_p$-algebra $\grad_\bullet(G)$.
\end{prop}

\begin{proof}
Set $W=H^1(\Zen_\theta(G),\F_p)$.
Then by Remark~\ref{4rem:graded tensor product}, Proposition~\ref{4prop:Zassenhaus algebra fibreprod}
and Example~\ref{5exs:koszul duals}, one has
\begin{eqnarray*}
 H^\bullet(G,\F_p)^! &=& \left(H^\bullet(G_\circ,\F_p)\otimes^{-1}\left(\bigwedge W\right)\right)^! \\
 &=& H^\bullet(G_\circ,\F_p)^!\otimes^1\left(\bigwedge W\right)^! \\
 &\simeq& \grad_\bullet(G_\circ)\otimes^1\F_p[Y_1,\ldots,Y_d] \\
 &=& \grad_\bullet(G),
\end{eqnarray*}
with $d=d(\Zen_\theta(G))$.
\end{proof}

\begin{example}\label{5exa:fibreprod}
Let $G$ be the cyclotomic fibre product as in Example~\ref{4exa:fibreprod}.
Recall that the $\F_p$-vector space $H^2(G,\F_p)$ is generated by the elements
\[\chi_1\cup\chi_2=\ldots=\chi_{d-1}\cup\chi_d \quad\text{and}\quad
\chi_i\cup\psi \ \text{for }i=1,\ldots,d.\]
Let $H^1(G,\F_p)^*=\spa_{\F_p}\{X_1,\ldots,X_d,Y\}$.
Then, one has $\mathcal{R}_1^\perp=\mathcal{R}_D\oplus\mathcal{R}_S$, with
\begin{eqnarray*}
 \mathcal{R}_D &=& \langle[X_1,X_2]+[X_3,X_4]+\ldots+[X_{d-1},X_d]\rangle,\\
 \mathcal{R}_S &=& \langle X_1Y-YX_1,\ldots,X_dY-YX_d\rangle.
\end{eqnarray*}
Therefore, \[H^\bullet(G,\F_p)^!\simeq
\frac{\F_p\langle X_1,\ldots,X_d\rangle}{\langle[X_1,X_2]+\ldots+[X_{d-1},X_d]\rangle}\otimes^1\F_p[Y],\]
with the grading induced by the degrees of the monomials.
\end{example}

Altogether, the class of Koszul duality pro-$p$ groups is closed under free pro-$p$ products
and cyclotomic fibre products, which are the basci operations of the ETC.
Also, the fundamental blocks of the ETC (free pro-$p$ groups and Demushkin groups) satisfy Koszul duality as well.
Thus, we may summarize the above results in the following theorem.

\begin{thm}
 Every pro-$p$ group of elementary type (i.e., every finitely generated pro-$p$ group lying in $\mathcal{ET}_p$)
is a Koszul duality group.
\end{thm}

\section{Koszul duality for cyclo-oriented pro-$p$ groups}

Let $G$ be a pro-$p$ group, and let \eqref{5eq:group presentation} be a minimal presentation for $G$.
Then \eqref{5eq:group presentation} induces a short exact sequence of restricted Lie algebras 
\begin{equation}\label{5eq:ses lie algebras}
 \xymatrix{ 0\ar[r] & \mathfrak{R}_\bullet\ar[r] & L_\bullet(F)\ar[r] & L_\bullet(G)\ar[r] & 0 }.
\end{equation}
In particular, $\mathfrak{R}_\bullet$ is the restricted Lie algebra induced by the filtration $R\cap D_i(F)$,
i.e.,
\begin{equation}\label{5eq:definition big frakr}
 \mathfrak{R}_{i-1}=\frac{R\cap D_i(F)}{R\cap D_{i+1}(F)}\quad \text{for every }i\geq2.
\end{equation}

Since \eqref{5eq:group presentation} is minimal, one has that $R\subseteq D_2(F)$.
Thus, one has $\mathfrak{R}_1=R/R\cap D_3(F)\simeq RD_3(F)/D_3(F)$.
Also, by Proposition~\ref{1prop:properties Zassenhaus filtration} the subgroup $R^p[R,F]$ is contained in $D_3(F)$,
and one may define a morphism $\phi\colon R/R^p[R,F]\to D_2(G)/D_3(G)$.

Moreover, $H^2(F,\F_p)=0$, thus by the five term exact sequence the transgression map
$\text{tg}_{F,R}$ is an isomorphism.
Hence, the map $\text{tg}_{F,R}$ and the isomorphism $(R/R^p[R,F])^*\simeq H^1(R,\F_p)^F$
induce a natural perfect pairing 
\begin{equation}\label{5eq:pairing 1}
\xymatrix{ R/R^p[R,F]\times H^2(G,\F_p)\ar[rr]^-{(\argu,\argu)} && \F_p }
\end{equation}
given by $(\bar r,\psi)\mapsto \text{tg}_{F,R}^{-1}(\psi).\bar r$, with $\psi\in H^2(G,\F_p)$ and $r\in R$.
Now let $G$ be a Bloch-Kato pro-$p$ group.
Then \cite[Theorem~8.4]{cem:quotients} implies the following.

\begin{prop}\label{5prop:RcupF3}
 Let $G$ be a Bloch-Kato pro-$p$ group, and let \eqref{5eq:group presentation} be a minimal presentationfor $G$.
Then $R^p[R,F]=R\cap D_3(F)$.
\end{prop}

\begin{rem}\label{5rem:cem}
 Note that all the results produced in \cite{cem:quotients} for maximal pro-$p$ Galois groups can be generalized to
Bloch-Kato pro-$p$ groups, as they are all developed starting from the characteristics of the $\F_p$-cohomology
of such groups, depending on the Bloch-Kato conjecture.
\end{rem}

Thus, by Proposition~\ref{5prop:RcupF3} one has the isomorphism $R/R^p[R,F]= R/R\cap D_3(F)\simeq R D_3(F)/D_3(F)$,
so that \eqref{5eq:pairing 1} induces the perfect pairing
\begin{equation}\label{5eq:pairing 2}
\xymatrix{ R D_3(F)/D_3(F)\times H^2(G,\F_p)\ar[rr]^-{(\argu,\argu)} && \F_p }.
\end{equation}
The quotient $R D_3(F)/D_3(F)$ is a subspace of $D_2(F)/D_3(F)=L_2(F)$.
Recall from Remark~\ref{1rem:magnus algebra} that the universal envelope $\mathcal{U}_p(L_\bullet(F))$ is
the free $\F_p$-algebra $\F_p\langle\mathcal{X}\rangle$, with $\mathcal{X}=\{X_1,\ldots,X_d\}$ a set
of (non-commutative) indeterminates, with $d=d(F)=d(G)$.
We may identify $\F_p\langle\mathcal{X}\rangle$ with the tensor algebra $\mathcal{T}^\bullet(V^*)$, with
\[V=H^1(F,\F_p)=F/\Phi(F)^\vee\simeq G/\Phi(G)^\vee.\]
Then the embedding $\psi_L$, defined as in \eqref{1eq:universal envelope embedding}, together
with \eqref{5eq:dual tensor product}, induces a monomorphism
\[ \xymatrix{ \dfrac{R D_3(F)}{D_3(F)}\ar@{^(->}[r]\ar@/^1.5pc/[rr]^-{\psi_2} & 
L_2(F)\ar@{^(->}[r] & (V\otimes V)^*}.\]

\begin{prop}\label{5prop:commutative diagram pairings}
Let $G$ be a finitely generated Bloch-Kato pro-$p$ group.
Then the diagram of perfect pairings
 \begin{equation}\label{5eq:commutative diagram pairings}
  \xymatrix@R=1.1truecm{ (V\otimes V)^*\times (V\otimes V)\ar@<6ex>[d]_-{\cup}\ar[rr] && \F_p\ar@{=}[d] \\
\dfrac{R D_3(F)}{D_3(F)}\times H^2(G,\F_p)\ar[rr]\ar@<6ex>[u]^-{\psi_2} && \F_p }
\end{equation}
commutes, i.e., $\text{tg}^{-1}(\chi_1\cup\chi_2).\bar r=\psi_2(\bar r)(\chi_1\otimes\chi_2)$ for every
\[\chi_1,\chi_2\in V= H^1(F,\F_p) \quad\text{and}\quad \bar r\in R D_3(F)/D_3(F).\]
\end{prop}

\begin{proof}
Set $d=d(F)$, and let $\{x_1,\ldots,x_d\}\subset F$ be a minimal generating set and
$\{\chi_1,\ldots,\chi_d\}\subset V$ the dual basis.
Then, the set $\{\chi_i\otimes\chi_j, 1\leq i,j\leq d\}$ is a basis for $V\otimes V$, and it is enough to verify
\begin{equation}\label{5eq:equality commutative diagram pairings}
 \text{tg}^{-1}(\chi_i\cup\chi_j).\bar r=\psi_2(\bar r)(\chi_i\otimes\chi_j)
\end{equation}
for every $i,j\in\{1,\ldots,d\}$.

By \cite[Prop.~3.9.13]{nsw:cohm} every element $r\in R$ may be uniquely written as
\[\bar r\equiv\begin{cases}
\displaystyle \:\prod_{h=1}^d x_i^{2a_h} \prod_{h<k}[x_h,x_k]^{a_{hk}}\mod D_3(F), \text{ if } p=2,\\
\displaystyle  
\displaystyle \: \prod_{h<k}[x_i,x_j]^{a_{kh}} \mod D_3(F), \text { if } p\not=2,
\end{cases}\]
with $0\leq a_i,a_{ij}\leq p-1$.
Recall that $\grad_\bullet(F)\simeq\F_p\langle X_1,\ldots,X_d\rangle$.
Thus, for every $i=1,\ldots,d$, we may consider $X_i$ to be the dual of $\chi_i$ in $V^*$.
Then one has 
\[\psi_L\left(x_h^2\right)=X_h\otimes X_h \quad\text{and}\quad \psi_L([x_h,x_k])=X_h\otimes X_k-X_k\otimes X_h.\]
Therefore
\[ \psi_2(\bar r)(\chi_i\otimes\chi_j)= \left\{\begin{array}{cc}
a_{ij}\bmod p & \text{if } i<j \\ -a_{ij}\bmod p & \text{if } i>j \\ -\binom{p}{2}a_i\bmod p & \text{if } i=j 
\end{array}\right.\]
Again by \cite[Prop.~3.9.13]{nsw:cohm}, the above result is the same we get for
$\text{tg}^{-1}(\chi_i\cup\chi_j).\bar r$, so that equality \eqref{5eq:equality commutative diagram pairings} holds.
\end{proof}

Proposition~\ref{5prop:commutative diagram pairings} implies the following.

\begin{thm}\label{5thm:koszul duality 1}
 Let $G$ be a finitely generated Bloch-Kato pro-$p$ group with a minimal presentation \eqref{5eq:group presentation},
and let $\mathfrak{r}$ be the ideal of the restricted Lie algebra $L_\bullet(F)$ generated by the elements
\[ \rho_i\in L_2(F),\quad \rho_i\equiv r_i\mod D_3(F),\]
with $i=1,\ldots,r(G)$.
Then $\mathcal{U}_p(L_\bullet(F)/\mathfrak{r})\simeq H^\bullet(G,\F_p)^!$.
In particular, if $G$ is quadratically defined,\footnote{I.e., $G$ is $\mathcal{U}$-quadratic and 
$\mathfrak{r}=\mathfrak{R}_\bullet$, with $\mathfrak{R}_\bullet$ as in \eqref{5eq:definition big frakr}.}
then $G$ is a Koszul duality pro-$p$ group, i.e., \[\grad_\bullet(G)\simeq H^\bullet(G,\F_p)^!.\]
\end{thm}

\begin{proof}
Let $\mathcal{R}_1\leq V\otimes V$ be the generating relations of the $\F_p$-cohomology ring $H^\bullet(G,\F_p)$.
Since \eqref{5eq:commutative diagram pairings} commutes (and both the horizontal lines are perfect pairings),
one has that $\psi_2(\bar r)\in \mathcal{R}_1^\perp$ for every $\bar r$, i.e.,
the quotient $RD_3(F)/D_3(F)$ embeds in $\mathcal{R}_1^\perp$ via $\psi_2$.
Moreover, one has 
\[\dim_{\F_p}\left(\frac{RD_3(F)}{D_3(F)}\right)=\dim_{\F_p}\left(\frac{R}{R^p[R,F]}\right)=r(G),\]
and, on the other hand,
\begin{eqnarray*}
 \dim_{\F_p}\left(\mathcal{R}_1^\perp\right)&=&\dim_{\F_p}(V\otimes V)-\dim_{\F_p}\left(\mathcal{R}_1\right) \\
&=&\dim_{\F_p}\left(H^2(G,\F_p)\right)\\&=&r(G).
\end{eqnarray*}
Therefore, $RD_3(F)/D_3(F)$ and $\mathcal{R}_1^\perp$ are isomorphic.

Since $RD_3(F)/D_3(F)$ generates $\mathfrak{r}$ as ideal of $L_\bullet(F)$ and $\mathcal{R}^\perp$ as ideal of 
$\mathcal{T}^\bullet(V^*)$, the statement holds by Proposition~\ref{1prop:presentation restricted envelope}.
\end{proof}

\begin{cor}
A finitely-generated Bloch-Kato pro-$p$ group is $\mathcal{U}$-quadratic. 
\end{cor}

Note that in general for a Bloch-Kato pro-$p$ group $G$ the ideal $\mathfrak{r}$ of $L_\bullet(G)$
as defined in Theorem~\ref{5thm:koszul duality 1} is smaller than the ideal
$\mathfrak{R}_\bullet$ as in \eqref{5eq:definition big frakr}.
Namely, one has 
\[\mathfrak{r}\cap L_i(F)\subseteq\frac{R\cap D_i(F)}{R\cap D_{i+1}(F)}\quad\text{for every }i\geq2,\]
and the initial forms of a set $\{r_1,\ldots,r_{r(G)}\}$ of generators of $R$ as normal subgroup of $F$
may be not enough to generate the whole ideal $\mathfrak{R}_\bullet$.

\begin{cor}
Let $G$ be a finitely generated Bloch-Kato pro-$p$ group, and assume that $G$ is quadratically defined.
Then $G$ is a Koszul duality group. 
\end{cor}

The material developed in Section~5.3 is partially motivated and inspired by J.~Labute's work \cite{labute:mild}.
In this paper, Labute proves many properties of {\it mild groups} (i.e., pro-$p$ groups with a {\it strongly free}
presentation).
In particular, he shows that if $G$ is a mild group, then one has $\mathfrak{r}=\mathfrak{R}_\bullet$
(cf. \cite[Theorem~5.1]{labute:mild}).
Yet, in general mild groups are very different to Bloch-Kato pro-$p$ groups.

\section[Koszul algebras]{Koszul algebras and Koszul pro-$p$ groups}

\begin{defi}\label{5defi:koszulness}
\begin{itemize}
 \item[(i)] A quadratic\footnote{In fact this definition holds for all positively graded algebras.
Indeed, a Koszul algebra is necessarily quadratic.}
 algebra $A_\bullet$ over a field $\F$ is said to be a {\bf Koszul algebra} over $\F$ if the homology groups
\[H_{ij}(A_\bullet)=\Tor^{A_\bullet}_{ij}(\F,\F)\]
(where the first grading $i$ is the cohomological grading and the second grading $j$ is the internal grading
induced from the grading of $A_\bullet$) are concentrated on the diagonal, i.e., $H^{ij}(A_\bullet)=0$ for $i\neq j$
(cf. \cite[\S~2.1]{pp:quadratic algebras}, \cite{positselsky:koszul} and \cite{positselsky:new}). 
 \item[(ii)] A Bloch-Kato pro-$p$ group is said to be a {\bf $H$-Koszul} pro-$p$ group if the $\F_p$-cohomology
ring $H^\bullet(G,\F_p)$ is Koszul, and a {\bf $\mathcal{U}$-Koszul} pro-$p$ group if the graded algebra
$\grad_\bullet(G)\simeq\mathcal{U}_p(L_\bullet(G))$ is Koszul.
\end{itemize}
\end{defi}

Note that the conditions on $H_{ij}(A_\bullet)$ for $i=1,2$ imply that $A_\bullet$ is necessarily quadratic.
In particular, a $\mathcal{U}$-Koszul pro-$p$ group is necessarily $\mathcal{U}$-quadratic.
The definition of a Koszul algebra was introduced first in the early '70s by S.~Priddy, and later associated 
to Galois theory and group cohomology.

In fact, Koszul algebras are strongly related to Galois cohomology, as they provide an alternative approach
to the Bloch-Kato conjecture, see, e.g., \cite[\S~0.1 and \S~0.2]{positselsky:new}.
Indeed, one has the following (cf. \cite[Theorem~1.3]{positselsky:koszul}).

\begin{thm}
 Let $K$ be a field containing a primitive $p$-th root of unity, and assume that:
\begin{itemize}
 \item[(i)] the Galois symbol of degree $n$ $\mathcal{K}_n^M(K)/p\to H^n(G_K,\mu_p^{\otimes n})$
is an isomorphism for $n=2$ and a monomorphism for $n=3$;
 \item[(ii)] the Milnor $\mathcal{K}$-ring $\mathcal{K}_\bullet^M(K)/p$ is a Koszul $\F_p$-algebra.
\end{itemize}
 Then the Galois symbol $h_K$ is an isomorphism, i.e., the Bloch-Kato conjecture holds.
\end{thm}

Koszul algebras have the following property.

\begin{prop}\label{5prop:koszul properties}
A quadratic algebra $A_\bullet$ is Koszul if, and only if, the Koszul dual $A_\bullet^!$ is Koszul.
In particular, if $A_\bullet$ is Koszul, then one has the isomorphism
$A_\bullet^!\simeq H^{\bullet\bullet}(A_\bullet)$.
\end{prop}

\begin{example}\label{5exm:koszul algebra}
The symmetric algebra $S^\bullet(V)$ and the exterior algebra $\bigwedge_\bullet(V)$ of a vector space $V$
over $\F$ are Koszul (cf. \cite[p.~20]{pp:quadratic algebras}).
\end{example}

\begin{cor}
 Let $G$ be a Bloch-Kato pro-$p$ group, and assume that $G$ is also a Koszul duality group.
Then $G$ is $H$-Koszul if, and only if, $G$ is $\mathcal{U}$-Koszul. 
\end{cor}

Let $A_\bullet$ be a (positively) graded algebra over a field $\F$.
The {\it Hilbert series} of $A_\bullet$ is the formal power series defined by 
\[h_{A_\bullet}(t)=\sum_{n\geq0}\dim_{\F}(A_n)t^n.\]
For a quadratic algebra one has the following (cf. \cite[Corollary~2.4]{pp:quadratic algebras}).

\begin{lem}\label{5lem:Hilbert series}
 Let $A_\bullet$ be a quadratic $\F$-algebra, and assume that either $A_3=0$ or $A_3^!=0$.
If $h_{A_\bullet}(t)h_{A_\bullet^!}(-t)=1$, then $A_\bullet$ (and thus also $A_\bullet^!$) is Koszul.
\end{lem}

\subsection{$H$-Koszulity, $\mathcal{U}$-Koszulity and the ETC}
It is natural to ask about the behavior of the fundamental blocks of the ETC with respect to $H$-Koszulity
and $\mathcal{U}$-Koszulity.
So far, the following result should not be very surprising.

\begin{prop}\label{5prop:Kosulity free demushkin}
Let $G$ be a free pro-$p$ group or a Demushkin group.
Then $G$ is $H$-Koszul and $\mathcal{U}$-Koszul.
\end{prop}

\begin{proof}
 Assume first that $G=F$ is a free pro-$p$ group.
Then $H^\bullet(F,\F_p)=\F_p\oplus H^1(F,\F_p)$ and $\grad_\bullet(F)\simeq\F_p\langle X_1,\ldots,X_d\rangle$,
with $d=d(F)$, and we may computer their Hilbert series:
\begin{eqnarray*}
 h_{H^\bullet(F,\F_p)}(t) &=& 1+dt, \\
 h_{\grad_\bullet(F)}(t) &=& 1+dt+d^2t^2+d^3t^3+\ldots
\end{eqnarray*}
Thus, we may apply Lemma~\ref{5lem:Hilbert series}, and we obtain
\[h_{H^\bullet(F,\F_p)}(-t)\cdot h_{\grad_\bullet(F)}(t)=(1-dt)(1+dt+d^2t^2+d^3t^3+\ldots),\]
so that $H^\bullet(F,\F_p)$ and $\grad_\bullet(F)$ are Koszul algebras.

Assume now that $G$ is a Demushkin group.
Then one has \[h_{H^\bullet(G,\F_p)}(t)=1+dt+t^2.\]
Moreover, since the canonical presentation of a Demushkin group is strongly free,
\cite[Theorem~5.1]{labute:mild} implies that the Hilbert series of $\grad_\bullet(G)$ is 
\[h_{\grad_\bullet(F)}(t)=\frac{1}{1-(t+t+\ldots+t)+t^2}=\frac{1}{1-dt+t^2}.\]
Therefore, $h_{H^\bullet(G,\F_p)}(t)\cdot h_{\grad_\bullet(G)}(-t)=1$,
and Lemma~\ref{5lem:Hilbert series} yields the claim.
\end{proof}

Moreover, the classes of $H$-Koszul pro-$p$ groups and of $\mathcal{U}$-Koszul pro-$p$ groups is closed under
free pro-$p$ products and cyclotomic fibre products.

\begin{prop}
 Let $G$, $G_1$ and $G_2$ be finitely generated Bloch-Kato pro-$p$ groups.
\begin{itemize}
 \item[(i)] The free pro-$p$ product $G_1\ast_{\hat p}G_2$ is a $H$-Koszul pro-$p$ group,
resp. a $\mathcal{U}$-Koszul pro-$p$ group, if, and only if,
$G_1$ and $G_2$ are both $H$-Koszul, resp. $\mathcal{U}$-Koszul.
 \item[(ii)] Assume that $(G,\theta)$ is cyclo-oriented.
Then the cyclotomic fibre product $G\ltimes_\theta Z$, with $Z\simeq\Z_p$, is again a $H$-Koszul pro-$p$ group,
resp. a $\mathcal{U}$-Koszul pro-$p$ group, if, and only if, $G$ is $H$-Koszul, resp. $\mathcal{U}$-Koszul.
\end{itemize}
\end{prop}

\begin{proof}
The first statement follows from \eqref{5eq:cohomology freeprod} and \eqref{5eq:envelope pree products}
and from \cite[Ch.~2, Corollary 1.2]{pp:quadratic algebras}.
The second statement follows from Proposition~\ref{5prop:koszul dual fibreprod}, Example~\ref{5exm:koszul algebra}
and again \cite[Ch.~2, Corollary 1.2]{pp:quadratic algebras}.
\end{proof}

\begin{thm}
 Every pro-$p$ group of elementary type (i.e., every finitely generated pro-$p$ group lying in $\mathcal{ET}_p$)
is a $H$-Koszul group and a $\mathcal{U}$-Koszul group.
\end{thm}

\begin{ques}\label{5ques:koszul}
Let $K$ be a field containing a primitive $p$-th root of unity such that the maximal pro-$p$ Galois group
$G_K(p)$ is finitely generated.
\begin{itemize}
 \item[(i)] Is $G_K(p)$ a Koszul duality group?
 \item[(ii)] Is $G_K(p)$ $H$-Koszul and/or $\mathcal{U}$-Koszul?
\end{itemize}
\end{ques}

Note that if (i) holds and $G_K(p)$ is $H$-Koszul, resp. $\mathcal{U}$-Koszul, then by Proposition~\ref{5prop:koszul properties}
the group $G_K(p)$ is also $\mathcal{U}$-Koszul, resp. $H$-Koszul.



\begin{thebibliography}{argomento}

\bibitem[AEJ]{arasonelmanjacob:rigid} J.K. Arason, R. Elman and B. Jacob. Rigid elements, valuations, and realizations of Witt rings.
    {\it J. Algebra} 110 (1987), 449-467.

\bibitem[ABW]{winterschool} S. Arias-de-Reyna, L. Bary-Soroker and G. Wiese (eds.). {\it Special issue: Proceedings of the Winter School on Galois theory, Universtiy of Luxembourg, February 2012. Vol. I}.
    Travaux Math\'ematiques 22, University of Luxembourg, Luxembourg, 2013.

\bibitem[BK86]{bk:BK} S. Bloch and K. Kato. $p$-adic \'etale cohomology.
    {\it Publ. Math. IHES} 63 (1986), 107-152.

\bibitem[BCMS]{bcms:bk} D. Benson, S.K. Chebolu, J. Min\'a\v{c} and J. Swallow. Bloch-Kato pro-$p$ groups and a refinement of the Bloch-Kato conjecture.
    Preprint (2007).

\bibitem[BT12]{bogomolovtschinkel:birat} F. Bogomolov and Y. Tschinkel. Introduction to birational anabelian geometry.
    In: {\it Current developments in algebraic geometry}, MSRI Publications vol. 59, Cambridge University Press (2012), 17-63.

\bibitem[Br66]{brumer:pseudocompact} A. Brumer. Pseudocompact algebras, profinite groups and class formations.
    {\it J. Algebra} 4 (1966), 442-470.

\bibitem[CEM]{cem:quotients} S.K. Chebolu, I. Efrat and J. Min\'a\v{c}. Quotients of absolute Galois groups which determine the entire Galois cohomology.
    {\it Math. Ann.} 352 (2012), 205-221.

\bibitem[CM08]{suniljan:quotients} S.K. Chebolu and J. Min\'a\v{c}. Absolute Galois groups viewed from small quotients and the Bloch-Kato conjecture.
    In: {\it New topological contexts for Galois theory and algebraic geometry (BIRS 2008)}, Geom. Topol. Monogr. vol. 16, Geom. Topol. Publ. (2009), 31-47. 

\bibitem[CMQ]{cmq:fast} S.K. Chebolu, J. Min\'a\v{c} and C. Quadrelli. Detecting fast solvability of equations via small powerful Galois groups.
    {\it Trans. Amer. Math. Soc.}, to appear.

\bibitem[DdSMS]{ddsms:analytic} J.D. Dixon, M.P.F. du Sautoy, A. Mann and D. Segal. {\it Analytic Pro-p Groups, Second edition}.
    Cambridge Studies in Advanced Mathematics 61, Cambridge University Press, Cambridge, 1999.

\bibitem[DL83]{dummitlabute:formula} D. Dummit and J. Labute. On a new characterization of Demushkin groups.
    {\it Invent. Math.} 73 (1983), 413-418.

\bibitem[Ef97a]{efrat:freeprod} I. Efrat. Free pro-$p$ product decompositions of Galois groups.
    {\it Math. Z.} 225 (1997), 245-261.

\bibitem[Ef97b]{efrat:etconj} I. Efrat. Pro-$p$ Galois groups of algebraic extensions of $\mathbb{Q}$.
    {\it J. Num. Th.} 64 (1997), 84-99.

\bibitem[Ef06]{efrat:libro} I. Efrat. {\it Valuation, orderings and Milnor $K$-theory}.
    Mathematical surveys and monographs 124, American Mathematical Society, Providence RI, 2006.

\bibitem[Ef14]{efrat:zassenhaus} I. Efrat. The Zassenhaus filtration, Massey products, and representations of profinite groups.
    {\it Adv. Math.} 263 (2014), 389-411.

\bibitem[EM11]{idojan:series} I. Efrat, J. Min\'a\v{c}. On the descending central sequence of absolute Galois groups.
    {\it Amer. J. Math.} 133 (2011), 1503-1532. 

\bibitem[En04]{engler:etc} A.J. Engler. A recursive description of pro-p Galois groups.
    {\it J. Algebra} 274 (2004), no. 2, 511-522. 

\bibitem[EK98]{ek98:abeliansbgps} A.J. Engler and J. Koenigsmann. Abelian subgroups of pro-$p$ Galois groups.
    {\it Trans. Amer. Math. Soc.} 350 (1998), 2473-2485.

\bibitem[EN94]{en94:pro2gps} A.J. Engler and J.B. Nogueira. Maximal abelian normal subgroups of Galois pro-$2$ groups.
    {\it J. Algebra} 166 (1994), 481-505.

\bibitem[Er01]{ershov:products} Y.L. Ershov. On free products of absolute Galois groups II.
    {\it Comm. Algebra} 29 (2001), no. 9, 3773-3779.

\bibitem[FJ08]{friedjarden} M.D. Fried and M. Jarden. {\it Field arithmetic. Third edition}.
    Ergebnisse der Mathematik un ihrer Grenzgebiete 3. Springer, Berlin, 2008.

\bibitem[G\"a11]{jochen:thesis} J. G\"artner. {\it Mild pro-$p$-groups with trivial cup-product}.
    PhD thesis at Heidelberg Universit\"at, Heidelberg, 2011.

\bibitem[Ge13]{wdg:lux} W.-D. Geyer. Field Theory.
    In: {\it Special issue: Proceedings of the Winter School on Galois theory, Universtiy of Luxembourg, February 2012. Vol. I}.
    Travaux Math\'ematiques 22, University of Luxembourg, Luxembourg, 2013.

\bibitem[HJK]{haranjardenkoenigsmann} D. Haran M. Jarden and J. Koenigsmann. Free product of absolute Galois groups.
    Preprint (200?), available at {\tt arXiv:math/0005311v1}.

\bibitem[HJ95]{hwangjacob:brauer} Y.S. Hwang and B. Jacob. Brauer group analogues of results relating the Witt ring to valuations and Galois theory.
    {\it Canad. J. Math.} 47 (1995), no. 3 527-543.

\bibitem[JW89]{jacobware:pro2} B. Jacob and R. Ware. A recursive description of the maximal pro-2 Galois group via Witt rings.
    {\it Math. Z.} 200 (1989), 379-396.

\bibitem[JW91]{jacobware:pro2bis} B. Jacob and R. Ware. Realizing dyadic factors of elementary type Witt rings and pro-2 Galois groups.
    {\it Math. Z.} 208 (1991), 193-208.

\bibitem[KS11]{klopschsnopce:generating} B. Klopsch and I. Snopce. Pro-$p$ groups with constant generating number of open subgroups.
    {\it J. Algebra} 331 (2011), 263-270.

\bibitem[KZ05]{kocloupavel:freeby} D.H. Kochloukova and P.A. Zalesski\u{i}. Free-by-Demushkin pro-$p$ groups.
    {\it Math. Z.} 249 (2005), 731-739.

\bibitem[Ko05]{koenigsmann:products} J. Koenigsmann. Products of absolute Galois groups.
    {\it Int. Math. Res. Not.} (2005), no. 24, 1465-1486.

\bibitem[La67]{labute:classification} J.P. Labute. Classification of Demushkin groups.
    {\it Canad. J. Math.} 19 (1967), 106-132.

\bibitem[La70]{labute:dcs1} J.P. Labute. On the descending central series of groups with a single defining relation. 
    {\it J. Algebra} 14 (1970) 16-23.

\bibitem[La85]{labute:dcs2} J.P. Labute. The determination of the Lie algebra associated to the lower central series of a group.
    {\it Trans. Amer. Math. Soc.} 288 (1985), no. 1, 51-57.

\bibitem[La06]{labute:mild} J.P. Labute. Mild pro-$p$-groups and Galois groups of $p$-extensions of $\Q$.
    {\it J. reine angew. Math.} 596 (2006), 155-182.

\bibitem[Ld05]{ledet:btep} A. Ledet. {\it Brauer type embedding problems}.
    Volume 21 of {\it Fields Institute Monographs}. American Mathematical Society, Providence, RI, 2005.

\bibitem[LS02]{leepsmith:2rigidity} D.B. Leep and T.L. Smith. Multiquadratic extensions, rigid fields and pythagorean fields.
    {\it Bull. London Math. Soc.} 34 (2002), 140-148.

\bibitem[Le55]{leptin:galois theorem} H. Leptin. Ein Darstellungssatz f\"ur kompakte, total unzusammenh\"angende Gruppen.
    {\it Arch. Math.} 6 (1955), 371-373 .

\bibitem[Li80]{lichtman:lie} A.I. Lichtman. On Lie algebras of free products of groups.
    {\it J. Pure Appl. Algebra} 18 (1980), 67-74.

\bibitem[MS82]{merkurjevsuslin:1} A.S. Merkur'ev and A.A. Suslin. $K$-cohomology of Brauer-Severi varieties and the norm residue isomorphism.
    {\it Izv. Akad. Nauk SSSR}, Ser. Mat. 46 (1982), 1132-1146 (Russian).

\bibitem[MS90]{merkurjevsuslin:2} A.S. Merkur'ev and A.A. Suslin. The norm residue symbol of degree 3.
    {\it Izv. Akad. Nauk SSSR}, Ser. Mat. 54 (1990), (339-356) (Russian).

\bibitem[MQRTW]{mqrtw:koszul} J. Min\'a\v{c}, C. Quadrelli, D.M. Riley, N.D. Tan and Th. Weigel. Koszul dual in Galois cohomology of a field.
    In preparation (2014).

\bibitem[MSp90]{minacspira:pytagoras} J. Min\'a\v{c} and M. Spira. Formally real fields, pythagorean fields, $C$-fields and $W$-groups.
    {\it Math. Z.} 205 (1990), no. 4, 519-530.

\bibitem[MSp96]{minacspira:witt} J. Min\'a\v{c} and M. Spira. Witt rings and Galois groups.
    {\it Ann. of Math. (2)} 144 (1996), no. 1, 35-60.

\bibitem[MT14]{jantan:massey} J. Min\'a\v{c} and N.D. T\^an. Triple Massey products and Galois theory.
    {\it J. Eur. Math. Soc.}, to appear.

\bibitem[Ne99]{ANT} J. Neukirch. {\it Algebraic Number Theory}.
    Grundlehren der mathematischen Wissenschaften, Vol. 322, Springer, Berlin Heidelberg, 1999.

\bibitem[NSW]{nsw:cohm} J. Neukirch, A. Schmidt and K. Wingberg. \textit{Cohomology of Number Fields}.
    Grundlehren der mathematischen Wissenschaften, Vol. 323, Springer, Berlin Heidelberg, 2000.

\bibitem[PP05]{pp:quadratic algebras} A. Polishchuck and L. Positselski. {\it Quadratic algebras}.
    University lecture series 37, American Mathematical Society, Providence RI, 2005.

\bibitem[Po94]{pop:birat} F. Pop. On Grothendieck's conjecture of birational anabelian geometry.
    {\it Ann. of Math.} 139 (1994), no. 2, 145-182.

\bibitem[Ps05]{positselsky:koszul} L. Positselski. Koszul property and Bogomolov's conjecture.
    {\it Int. Math. Research Not.} 31 (2005), 1901-1936.

\bibitem[Ps14]{positselsky:new} L. Positselski. Galois cohomology of a number field is Koszul.
    Preprint (2014), available at {\tt arXiv:1008.0095}.

\bibitem[Qu14]{claudio:BK} C. Quadrelli. Bloch-Kato pro-$p$ groups and locally powerful groups.
    {\it Forum Math.} 26 (2014), no. 3, 793-814.

\bibitem[QW1]{qw:diederichsen} C. Quadrelli and Th. Weigel. Diederichsen's theorem for cohomological Mackey functors and finitely generated free pro-$p$ groups.
    Preprint (2014).

\bibitem[QW2]{qw:orientations} C. Quadrelli and Th. Weigel. Profinite groups with cyclotomic $p$-orientations.
    Preprint (2014).

\bibitem[Ri13]{ribes:lux} L. Ribes. Introduction to profinite groups.
    In: {\it Special issue: Proceedings of the Winter School on Galois theory, Universtiy of Luxembourg, February 2012. Vol. I}.
    Travaux Math\'ematiques 22, University of Luxembourg, Luxembourg, 2013.

\bibitem[RZ10]{ribeszalesskii:profinite} L. Ribes and P. Zalesski\u{i}. {\it Profinite groups, 2nd edition}.
    Ergebnisse der Mathematik und ihrer Grenzgebiete 40, Springer, Heidelberg, 2010.

\bibitem[SW00]{symonds_thomas:prop} P. Symonds and Th. Weigel. Cohomology of $p$-adic analytic groups.
    In: {\it New horizons in pro-$p$ groups}, Progress in Mathematics 184, Birkh\"auser Boston, Boston, MA, (2000), 349-410.

\bibitem[Sz77]{szym:rigid} K. Szymiczek. Quadratic forms over fields.
    {\it Dissertationes Math. (Rozprawy Mat.)} 152 (1977).

\bibitem[Ta76]{tate:K2cohomology} J. Tate. Relations between $K_2$ and Galois cohomology.
    {\it Invent. Math.} 36 (1976), 257-274.

\bibitem[Vo03]{voevodsky:1} V. Voevodsky. Motivic cohomology with $\Z/2$-coefficients.
    {\it Publ. Math. IEHS} 98 (2003), 59-104.

\bibitem[Vo11]{voevodsky:2} V. Voevodsky. On motivic cohomology with $\Z/l$-coefficients.
    {\it Ann. of Math.} (2) 174 (2011), no. 1, 401-438.

\bibitem[Wr81]{ware:rigid} R. Ware. Valuation rings and rigid elements in fields.
    {it Canad. J. Math.} 33 (1981), no. 6, 1338-1355.

\bibitem[Wr92]{ware:galp} R. Ware, Galois groups of maximal p-extensions.
    {\it Trans. Amer. Math. Soc.} 33 (1992), no. 2, 721-728.

\bibitem[We08]{weibel:proof1} C. Weibel. The proof of the Bloch-Kato Conjecture.
    ICTP Lecture Notes Series 23 (2008), 1-28.

\bibitem[We09]{weibel:proof2} C. Weibel. The norm residue isomorphism theorem.
    {\it J. Topol.} 2 (2009), no. 2, 346-372.

\bibitem[Wg13]{thomas:projectivedim} Th. Weigel. The projective dimension of profinite modules for pro-$p$ groups.
    Preprint (2013), available at {\tt arXiv:1303.5872v1}.

\bibitem[WZ13]{thomaspavel:stalldecomp} Th. Weigel and P.A. Zalesski\u{i}. Stallings' decomposition theorem for finitely generated pro-$p$ groups.
    Preprint (2013), available at {\tt arXiv:1305.4887}.

\bibitem[Wn13]{wingberg:stalldecomp} K. Wingberg. On the theory of ends of a pro-$p$ group.
    Preprint (2013), available at {\tt arXiv:1305.0634}.

\bibitem[W\"u85]{wurfel:remark} T. W\"urfel. A remark on the structure of absolute Galois groups.
    {\it Proc. Amer. Math. Soc.} 95 (1985), no. 3, 353-356.

\end{thebibliography}
\end{document}